\documentclass[11pt]{article}

\usepackage{amsfonts,latexsym,amsthm,amssymb,amsmath,amscd,euscript,tikz,mathtools}
\usepackage{framed}
\usepackage{fullpage}
\usepackage{color}
\usepackage[colorlinks=true,citecolor=blue,linkcolor=blue]{hyperref}
\usepackage{enumitem}
\usepackage{siunitx}
\usepackage{textcomp}
\usepackage{physics}
\usepackage{mathrsfs}
\usepackage{dsfont}
\usepackage[ruled,vlined]{algorithm2e}
\usepackage{titlesec}
\usepackage{tikz-cd}
\usepackage{thmtools}
\usepackage{thm-restate}
\usepackage{cleveref}
\usepackage{graphicx}

\allowdisplaybreaks[1]

\newtheorem{theorem}{Theorem}
\newtheorem{proposition}[theorem]{Proposition}
\newtheorem{lemma}[theorem]{Lemma}

\newtheorem{corollary}[theorem]{Corollary}

\theoremstyle{definition}
\newtheorem{definition}[theorem]{Definition}
\newtheorem{example}[theorem]{Example}
\newtheorem{remark}[theorem]{Remark}

\theoremstyle{remark}

\newcommand{\cC}{\mathcal{C}}

\newcommand{\cM}{\mathcal{M}}

\newcommand{\bE}{\mathbb{E}}\newcommand{\bF}{\mathbb{F}}

\newcommand{\bN}{\mathbb{N}}

\newcommand{\bR}{\mathbb{R}}

\newcommand{\bZ}{\mathbb{Z}}

\newcommand{\bfG}{\mathbf{G}}
\newcommand{\bfI}{\mathbf{I}}
\newcommand{\bfK}{\mathbf{K}}
\newcommand{\bfL}{\mathbf{L}}
\newcommand{\bfW}{\mathbf{W}}

\newcommand{\1}{\mathds{1}}

\newcommand{\Ber}{\operatorname{Ber}}

\newcommand{\poly}{\operatorname{poly}}

\newcommand{\disTV}{d_{\text{TV}}}

\newcommand{\Cay}{\operatorname{Cay}}

\newcommand{\ud}{\uparrow\downarrow}
\newcommand{\du}{\downarrow\uparrow}

\newcommand{\nc}{\newcommand}

\nc{\on}{\operatorname}
\nc{\Spec}{\on{Spec}}
\nc{\Aut}{\textit{Aut}}
\nc{\id}{\textit{id}}
\nc{\chr}{\on{char}}
\nc{\im}{\on{im}}
\nc{\Hom}{\on{Hom}}
\nc{\lcm}{\on{lcm}}
\nc{\dual}[1]{\prescript{t}{}{#1}}
\nc{\transpose}[1]{{#1}^{\intercal}}
\nc{\Sym}{\on{Sym}}
\nc{\End}{\on{End}}
\nc{\stab}{\on{stab}}
\nc{\Li}{\on{Li}}
\nc{\spn}{\on{span}}
\nc{\sgn}{\on{sgn}}
\nc{\supp}{\on{supp}}
\nc{\Unif}{\on{Unif}}

\title{From Grassmannian to Simplicial High-Dimensional Expanders}
\author{Louis Golowich\thanks{UC Berkeley. Email: \texttt{lgolowich@berkeley.edu}. Supported by an NSF Graduate Research Fellowship.}}

\begin{document}

\pagenumbering{gobble}

\maketitle

\begin{abstract}
  In this paper, we present a new construction of simplicial complexes of subpolynomial degree with arbitrarily good local spectral expansion. Previously, the only known high-dimensional expanders (HDXs) with arbitrarily good expansion and less than polynomial degree were based on one of two constructions, namely Ramanujan complexes and coset complexes. In contrast, our construction is a Cayley complex over the group $\bF_2^k$, with Cayley generating set given by a Grassmannian HDX.

  Our construction is in part motivated by a coding-theoretic interpretation of Grassmannian HDXs that we present, which provides a formal connection between Grassmannian HDXs, simplicial HDXs, and LDPC codes. We apply this interpretation to prove a general characterization of the 1-homology groups over $\bF_2$ of Cayley simplicial complexes over $\bF_2^k$. Using this result, we construct simplicial complexes on $N$ vertices with arbitrarily good local expansion for which the dimension of the 1-homology group grows as $\Omega(\log^2N)$. No prior constructions in the literature have been shown to achieve as large a 1-homology group.
\end{abstract}

\newpage

\tableofcontents

\newpage

\pagenumbering{arabic}


\section{Introduction}
\label{sec:intro}
High-dimensional expanders, which generalize expander graphs to higher-dimensional objects such as hypergraphs, have been of significant recent interest due to various applications, including to coding theory \cite{evra_decodable_2020,kaufman_new_2021,dikstein_locally_2020-1,dinur_locally_2022,panteleev_asymptotically_2022,leverrier_quantum_2022-1,dinur_good_2023,first_good_2022}, sampling algorithms \cite{anari_log-concave_2019,alev_improved_2020,anari_spectral_2020,chen_optimal_2021,chen_localization_2022}, and constraint satisfaction problems \cite{alev_approximating_2019,dinur_explicit_2021,hopkins_explicit_2022-1} (among many more papers). Yet there have only been essentially two constructions of high-dimensional expanders (HDXs) with arbitrarily good expansion and constant degree, namely Ramanujan complexes \cite{cartwright_ramanujan_2003,li_ramanujan_2004,lubotzky_ramanujan_2005,lubotzky_explicit_2005} and coset complexes \cite{kaufman_construction_2018,odonnell_high-dimensional_2022-1}. Various extensions and spinoffs of these two constructions, both of which are group theoretic in nature, have been introduced, but they rely on similar underlying machinery. Even upon allowing the degree to grow as a small polynomial, these two constructions and their extensions remain the only known ones with arbitrarily good expansion; one recent construction \cite{liu_local_2022} of 2-dimensional complexes using random geometric graphs achieves nontrivial (but non-optimal) expansion with small polynomial degree.

It is a major open question to obtain new constructions of low-degree HDXs with good expansion. The current literature on HDX constructions mirrors the early literature on expander graph constructions, which were also group theoretic in nature, e.g.~\cite{margulis_explicit_1973,lubotzky_ramanujan_1988}. However, more elementary and combinatorial constructions were later constructed \cite{reingold_entropy_2002}, which inspired multiple breakthroughs in complexity and coding theory \cite{dinur_pcp_2007,reingold_undirected_2008,ta-shma_explicit_2017}. Additional constructions of HDXs with new properties may similarly yield further applications.

The HDX constructions described above give expanding simplicial complexes, a class of hypergraphs that are the traditional object studied in the HDX literature. However, there has been recent interest in more general posets with high-dimensional expansion. A notable example is Grassmannian complexes, whose elements are subspaces of a vector space. The expansion of such complexes was used in the proof of the 2-to-1 games conjecture \cite{dinur_towards_2018,khot_pseudorandom_2018,barak_small-set_2018}, which has helped motivate more recent work on Grassmannian HDXs \cite{dikstein_boolean_2022,kaufman_garlands_2022,gaitonde_eigenstripping_2023}, with the hope of further applications to complexity theory. Following these recent works, we will specifically be interested in Grassmannian complexes that have good expansion properties despite being sparser than the complete complex, meaning they consist of only a subset of all the subspaces of a vector space.

In this paper, we introduce a new construction of simplicial high-dimensional expanders of all dimensions with arbitrarily good expansion and subpolynomial degree. We construct these simplicial HDXs by proving a relationship between simplicial and Grassmannian HDXs: a Grassmannian HDX over $\bF_2$ generates a simplicial HDX as a Cayley complex. We then construct Grassmannian HDXs using a third type of high-dimensional expanding poset that we introduce and analyze, called the \textit{matrix poset}.

Our approach is perhaps surprising given that low-degree Grassmannian HDXs are strictly more difficult to construct than simplicial HDXs, in a formal sense (see Section~\ref{sec:basisification}). Indeed, the Grassmannian HDXs we construct have polynomially large degree. However, we observe that Grassmannian HDXs have an alternative notion of sparsity, namely sparsity within their ambient vector space. We show that if a Grassmannian complex is sparse within its ambient vector space, then the Cayley simplicial complex it generates has low degree. Thus we are able to obtain subpolynomial-degree Cayley simplicial HDXs from polynomial-degree Grassmannian HDXs.

Our construction is in part motivated by a coding theoretic interpretation of Grassmannian HDXs that we introduce. We show how a Grassmannian HDX $X$ is naturally associated to a linear LDPC code. We then show that good expansion of $X$ implies its code has good distance, while good sparsity of $X$ within the ambient vector space is equivalent to its code having good rate. It turns out that $X$ can only have good expansion if its code has many redundant parity checks, a property satisfied by codes with locality properties such as locally testable codes and locally correctable codes. From this viewpoint, our Grassmannian HDX construction is related to degree-2 Reed-Muller codes, or more specifically, to the tensor product of 2 Hadamard codes.

We make use of our coding theoretic interpretation to analyze the 1-homology groups over $\bF_2$ of our Cayley simplicial HDXs. Specifically, for a Grassmannian complex $X$, we present a general homomorphism from the 1-homology group of the Cayley simplicial complex to the quotient of two different linear codes associated to $X$. Using this result, we construct $N$-vertex simplicial HDXs with 1-homology of dimension $\Omega(\log^2N)$. To the best of our knowledge prior HDX constructions all had smaller 1-homology groups. This problem of constructing HDXs with large homology groups is interesting from a coding theoretic perspective. For instance, \cite{evra_decodable_2020,kaufman_new_2021} construct quantum LDPC codes from HDXs, for which the homology dimension corresponds to the dimension of the associated code.

We remark that our coding theoretic view of Cayley simplicial HDXs builds on a construction of \cite{vadhan_construction_2018}, which constructs 2-dimensional Cayley HDXs over the group $\bF_2^k$ using random LDPC codes. Related constructions of Cayley HDXs are also given in \cite{conlon_hypergraph_2019,conlon_hypergraph_2020}. These constructions achieve logarithmic degree, which is better than our construction's subpolynomial degree. However, we achieve arbitrarily good expansion, whereas the expansion exhibited by the constructions in \cite{vadhan_construction_2018,conlon_hypergraph_2019,conlon_hypergraph_2020} is not strong enough to yield certain local-to-global properties inherent to high-dimensional expansion (see for instance Theorem~\ref{thm:tricklesimpinf} and the subsequent discussion). By considering such local-to-global properties, we build upon \cite{vadhan_construction_2018} to show an even closer connection between expansion and coding theoretic properties than was previously known (see for instance Proposition~\ref{prop:exptodisinf}).

\subsection{Background}
\label{sec:background}
This section describes background definitions necessary to describe our main results, along with some relevant prior work. We begin with the following definitions for posets, which generally follow \cite{kaufman_garlands_2022}.

\begin{definition}
  A \textbf{poset $X$} is a set with a binary relation $\prec$ such that if $x\prec x'$ and $x'\prec x''$, then $x\prec x''$ and $x'\not\prec x$. If either $x\prec x'$ or $x=x'$, we say that \textbf{$x'$ dominates $x$}, which we denote $x\preceq x'$. The poset is \textbf{graded} if it has a rank function $\rank:X\rightarrow\bZ_{\geq-1}$ such that there is a single element of rank $-1$, and such that if $x\prec x'$ with no $x''$ having $x\prec x''\prec x'$, then $\rank(x)=\rank(x')-1$. We let $X(i)=\rank^{-1}(i)$ and $\rank(X)=\max_{x\in X}\{\rank(x)\}$. We refer to elements of $X(i)$ as \textbf{vertices}, and to elements of $X$ as \textbf{faces}. The poset if \textbf{pure} if every element is dominated by a top-rank element. A \textbf{standard weight function} is then a function $m_X:X\rightarrow[0,1]$ such that each restriction $m_{X(i)}=m|_{X(i)}:X(i)\rightarrow[0,1]$ is a probability distribution, and such that for each $i<r$ we may sample $x_i\sim m_{X(i)}$ by sampling $x_{i+1}\sim m_{X(i+1)}$, then sampling $x_i\sim\Unif\{x_i'\in X(i):x_i'\prec x_{i+1}\}$.
\end{definition}

In this paper we restrict attention to pure graded posets with standard weight functions, and simply refer to such objects as ``posets.'' Most posets we consider have all distributions $m_{X(i)}$ be uniform, which will be the case unless otherwise specified.

In this paper we consider three types of posets. The first two, simplicial and Grassmannian complexes, are two of the most prominent objects of study in the high-dimensional expander literature.

\begin{definition}
  For a vertex set $V$, a \textbf{simplicial complex $X$} is a poset whose elements are subsets of $V$, such that if $x\in X$ then every $x'\subseteq x$ has $x'\in X$. Furthermore $x\prec x'$ if and only if $x\subseteq x'$, and $\rank(x)=|x|-1$.
\end{definition}

\begin{definition}
  For a vector space $\bF_q^k$, a \textbf{$\bF_q$-Grassmannian complex $X$} is a poset whose elements are linear subspaces of the \textbf{ambient vector space} $\bF_q^k$, such that if $x\in X$ then every subspace $x'\subseteq x$ has $x'\in X$. Furthermore $x\prec x'$ if and only if $x\subseteq x'$ is a linear subspace, and $\rank(x)=\dim(x)-1$.
\end{definition}

We emphasize that for a Grassmannian complex $X$, we do not require $X(i)$ to contain all $(i+1)$-dimensional subspaces of $\bF_q^k$. If $X(i)$ does contain all $(i+1)$-dimensional subspaces for all $i$, we say that $X$ is a complete Grassmannian complex. Note that elsewhere in the literature, the term ``Grassmannian'' is often used to refer to such complete Grassmannian complexes.

The third type of poset we study, which we call the matrix poset, is the poset of matrices in $\bF_q^{n\times n}$, with rank given by matrix rank. To the best of our knowledge, this object has not been previously studied in the HDX literature. We will introduce and analyze this poset to prove that our Grassmann HDX construction has good expansion.

The local structure of a poset is measured by its links:

\begin{definition}
  For a poset $X$, the \textbf{link} of an element $x\in X$ is the subposet $X_x=\{x'\in X:x\preceq x'\}$ with weight function $m_{X_x(i)}$ proportional to $m_X|_{X_x(i+\rank(x)+1)}$.
\end{definition}

To define poset expansion, we first associate the following graph to each poset, which captures the low-rank structure.

\begin{definition}
  The \textbf{1-skeleton graph $\bfG_X$} of a poset $X$ has vertex set $V(\bfG_X)=X(0)$, edge set $E(\bfG_X) = \{\{x_0,x_0'\}:x_0\neq x_0'\in X(0),\; \exists x_1\in X(1):x_0,x_0'\prec x_1\}$, and weight function $m_{\bfG_X}(\{x_0,x_0'\}) = \sum_{x_1\in X(1):x_0,x_0'\prec x_1}m_X(x_1)$.
\end{definition}

The expansion of posets is then measured by local expansion, defined below as the worst expansion of the 1-skeleton graph of any link. Recall that the (spectral) expansion $\lambda(\bfG)\in[0,1]$ of a graph $\bfG$ is the second largest absolute value of an eigenvalue of its random walk matrix.

\begin{definition}
  For a rank-$r$ poset $X$, the \textbf{rank-$i$ local expansion $\lambda^{(i)}(X)$} is defined to be
  \begin{equation*}
    \lambda^{(i)}(X) = \max_{x_i\in X(i)}\{\lambda(\bfG_{X_{x_i}})\}.
  \end{equation*}
  The \textbf{local expansion $\lambda(X)$} is defined to be
  \begin{equation*}
    \lambda(X) = \max_{-1\leq i\leq r-2}\{\lambda^{(i)}(X)\}.
  \end{equation*}
\end{definition}

Just as the expander graph literature aims to achieve good expansion with low degree, we are interested in posets with good local expansion and low degree as defined below.

\begin{definition}
  For a poset $X$, the \textbf{degree of $x\in X$} is given by $\deg(x)=|X_x|$, and the \textbf{degree of $X$} is $\deg(X)=\max_{x\in X:\rank(x)\geq 0}\{\deg(x)\}$.
\end{definition}

A central question in the high-dimensional expander literature is to construct families of low-degree simplicial complexes with good local expansion. That is, we are interested in infinite families of simplicial complexes $X$ such that as the number of vertices $|X(0)|\rightarrow\infty$, the local expansion $\lambda(X)$ stays below some arbitrarily small fixed constant $\epsilon$, and the rank $\rank(X)$ stays at some arbitrarily large fixed constant $r$. Furthermore, we want the degree to grow slowly (ideally as a constant) with respect to the number of vertices.

This question is of interest in part due to the following ``trickle-down theorem,'' which was proven by Oppenheim~\cite{oppenheim_local_2018} for simplicial complexes, and extended to more general posets such as Grassmannian complexes by Kaufman and Tessler~\cite{kaufman_garlands_2022}.

\begin{theorem}[Trickle-down for simplicial complexes \cite{oppenheim_local_2018}]
  \label{thm:tricklesimpinf}
  Let $X$ be a rank-$r$ simplicial complex such that for every $x\in X$ with $\rank(x)\leq r-2$, the 1-skeleton graph $\bfG_{X_x}$ of the link $X_x$ is connected. Then for every $-1\leq i\leq r-3$, it holds that
  \begin{equation*}
    \lambda^{(i)}(X) \leq \frac{\lambda^{(i+1)}(X)}{1-\lambda^{(i+1)}(X)}.
  \end{equation*}
\end{theorem}

\begin{theorem}[Trickle-down for Grassmannian complexes \cite{kaufman_garlands_2022}]
  \label{thm:tricklegrassinf}
  Let $X$ be a rank-$r$ $\bF_q$-Grassmannian complex such that for every $x\in X$ with $\rank(x)\leq r-2$, the 1-skeleton graph $\bfG_{X_x}$ of the link $X_x$ is connected. Then for every $-1\leq i\leq r-3$, it holds that
  \begin{equation*}
    \lambda^{(i)}(X) \leq \frac{\lambda^{(i+1)}(X)}{q(1-\lambda^{(i+1)}(X))}.
  \end{equation*}
\end{theorem}

Theorem~\ref{thm:tricklesimpinf} in particular implies that if $\lambda^{(0)}(X)<1/2$, then $\lambda^{(-1)}(X)<1$, so it shows that local expansion in rank $0$ implies global expansion of the entire 1-skeleton graph. Indeed, we use this property to bound $\lambda^{(-1)}$ for our complexes.

This ``local-to-global'' phenomenon is one of the most notable properties of HDXs, and for this reason, there is particular interest in simplicial complexes with very small local expansion, or at least with $\lambda^{(0)}(X)<1/2$. Yet only a few such constructions of low-degree complexes with such good expansion are known. As described earlier, in Section~\ref{sec:intro}, there are essentially only two constructions of constant-degree simplicial complexes of arbitrarily good expansion and arbitrarily large rank, namely Ramanujan complexes \cite{cartwright_ramanujan_2003,li_ramanujan_2004,lubotzky_ramanujan_2005,lubotzky_explicit_2005} and the Kaufman-Oppenheim coset complexes \cite{kaufman_construction_2018}, though several extensions and spinoffs that use these constructions to obtain more HDXs have been introduced \cite{friedgut_hyper-regular_2020,odonnell_high-dimensional_2022-1,dikstein_new_2022}. These constructions are group theoretic in nature, as they are based on matrix groups over finite fields.

Even upon allowing the degree to grow as a small polynomial $N^\epsilon$ in the number of vertices $N$, to the best of our knowledge all prior known simplicial HDX constructions of arbitrarily good expansion and large rank are based on either Ramanujan complexes or coset complexes. Further restricting our requirements to complexes just of rank $2$ and allowing any expansion $\lambda^{(0)}<1/2$, there is one additional recent construction of \cite{liu_local_2022}, which obtains rank-2 simplicial complexes from random geometric graphs on $N$ vertices with degree $N^\epsilon$ and $\lambda^{(0)}=1/2-\delta$ for some $\delta=\delta(\epsilon)>0$.

We remark that there have been several more constructions of simplicial HDXs that achieve local expansion $\lambda(X)<1$, but do not achieve $\lambda^{(0)}(X)<1/2$ \cite{vadhan_construction_2018,conlon_hypergraph_2019,conlon_hypergraph_2020,chapman_expander_2020,liu_high-dimensional_2020,golowich_improved_2021}. In particular, \cite{liu_high-dimensional_2020,golowich_improved_2021} show that such objects with constant degree and arbitrarily large rank can be obtained from arbitrary expander graphs, and can achieve $\lambda^{(0)}(X)=1/2$. But getting arbitrarily good local expansion, or even just any $\lambda^{(0)}(X)<1/2$, seems significantly more challenging, as described above.

To present our main results, we will also need the following standard definitions from algebraic topology. Here we restrict attention to chain complexes over $\bF_2$ that are associated to simplicial complexes.

\begin{definition}
  Let $X$ be a rank-$r$ simplicial complex. For $0\leq i\leq r$, let $\partial_i:\bF_2^{X(i)}\rightarrow\bF_2^{X(i-1)}$ be the linear map defined by $\partial_i\1_{x_i}=\sum_{x_{i-1}\in X(i-1):x_{i-1}\prec x_i}\1_{x_{i-1}}$ for $x_i\in X(i)$, and extended to all of $\bF_2^{X(i)}$ by linearity. Then we define:
  \begin{itemize}
  \item Elements of $Z_i(X)=\ker\partial_i$ are called \textbf{$i$-cycles}.
  \item Elements of $B_i(X)=\im\partial_{i+1}$ are called \textbf{$i$-boundaries}.
  \item Because $\partial_{i-1}\partial_i=0$, it holds that $B_i(X)\subseteq Z_i(X)$. Therefore there is a well defined group $H_i(X)=Z_i(X)/B_i(X)$, which is called the \textbf{$i$-homology group}.
  \end{itemize}
\end{definition}

Intuitively, the homology groups of a complex measure how topologically ``connected'' it is, with smaller homology groups corresponding to more well-connected complexes. For instance, complete simplicial complexes, as well as other sufficiently dense complexes, are known to have vanishing homology groups (see for instance \cite{meshulam_homological_2009,kahle_random_2011}). As described below, motivated by applications to (quantum) coding theory, we focus on constructing complexes with good spectral expansion and large homology groups. The fact that such constructions exist is perhaps surprising, as intuitively such complexes are simultaneously well-connected in a spectral sense, and yet poorly-connected in a topological sense.

\subsection{Main results}
This section presents our main theorems. We state the parameters of our main simplicial HDX construction below.

\begin{theorem}[Informal statement of Theorem~\ref{thm:maincayley}]
  \label{thm:maincayleyinf}
  For every $r\in\bN$ and every $\lambda>0$, there exists an infinite family of rank-$r$ simplicial complexes $Y$ with local expansion $\lambda(Y)\leq\lambda$ and degree $\deg(Y)\leq 2^{O_{r,\lambda}(\sqrt{\log N})}$, where $N$ denotes the number of vertices.
\end{theorem}

Thus we obtain simplicial complexes with arbitrarily good local expansion, arbitrarily large rank, and subpolynomial degree. As discussed above, the only other such constructions are based on Ramanujan complexes or coset complexes, both of which in fact achieve constant degree. However, as we will explain below, our construction is based on different techniques, which reveal a close relationship between simplicial and Grassmannian HDXs, and permits a coding theoretic interpretation that yields a clean characterization of the 1-homology group over $\bF_2$. Using this characterization, we obtain the following result by quotienting our construction from Theorem~\ref{thm:maincayleyinf}.

\begin{theorem}[Informal statement of Theorem~\ref{thm:bighom}]
  \label{thm:bighominf}
  For every $r\in\bN$ and every $\lambda>0$, there exists an infinite family of rank-$r$ simplicial complexes $Y$ with local expansion $\lambda(Y)\leq\lambda$ and 1-homology over $\bF_2$ of dimension $\Omega_{r,\lambda}(\log^2N)$, where $N$ denotes the number of vertices.
\end{theorem}

This problem of constructing HDXs with large homology groups is motivated by applications to coding theory, and specifically to quantum codes. For instance, \cite{evra_decodable_2020,kaufman_new_2021} construct quantum LDPC codes from HDXs, where the homology dimension corresponds to the code's dimension, that is, the number of qubits it encodes. Because our HDXs do not have constant degree, they do not give quantum LDPC codes of constant locality as in \cite{evra_decodable_2020,kaufman_new_2021}. However, Theorem~\ref{thm:bighominf} highlights the potential of our quotienting technique (see Section~\ref{sec:homcharinf} and Section~\ref{sec:cayleyhom}) to improve the dimension of such codes.

Our construction technique presents a tradeoff between 1-homology dimension and degree. Therefore although the complexes in Theorem~\ref{thm:bighominf} have degree $\deg(Y)=\poly(N)$, we also for instance obtain $N$-vertex complexes of subpolynomial degree $2^{O_{r,\lambda}(\log N/\log\log N)}=N^{o_{r,\lambda}(1)}$ and nearly as large 1-homology dimension $\tilde{\Omega}_{r,\lambda}(\log^2N)$ as in Theorem~\ref{thm:bighominf}. Similarly, we can get $N$-vertex complexes of even smaller subpolynomial degree $\deg(Y)\leq 2^{O_{r,\lambda}(\sqrt{\log N})}$ with 1-homology of dimension $\Omega_{r,\lambda}(\log N)$. To the best of our knowledge, Theorem~\ref{thm:bighominf} is new, in that no prior known simplicial complexes of any degree with good local expansion were shown to have 1-homology of dimension $\Omega(\log^2N)$; see the remark below.

\begin{remark}
  \label{rem:homdim}
  To the best of our knowledge, prior $N$-vertex HDX constructions such as Ramanujan complexes and the Kaufman-Oppenheim coset complexes are at best known to have 1-homology of dimension order $\log N$. The only reference in the literature we were able to find is \cite{kaufman_isoperimetric_2016}, which shows that there are Ramanujan complexes with nonvanishing 1- and 2-homology groups over $\bF_2$. Word-of-mouth discussions have suggested some prior constructions have homology of logarithmically large dimension, though we do not have a formal reference.
\end{remark}


\subsection{Cayley simplicial HDXs generated by Grassmannian HDXs}
\label{sec:cayleyinf}
The HDXs in Theorem~\ref{thm:maincayleyinf} and Theorem~\ref{thm:bighominf} are constructed as Cayley simplicial complexes generated by Grassmannian HDXs. Below, we present the framework for such Cayley complexes. However, we first need the following definition.

\begin{definition}
  Let $X$ be rank-$r$ a Grassmannian complex. The \textbf{basisification} $\beta(X)$ is the rank-$r$ simplicial complex containing all sets of linearly independent elements of $X(0)$ whose span lies in $X$.
\end{definition}

Thus $\beta(X)(0)=X(0)$, and $\beta(X)(i)$ consists of all bases of elements $x\in X$. We show that basisification preserves local expansion, that is $\lambda^{(i)}(\beta(X))=\lambda^{(i)}(X)$, which leads to the following framework for constructing simplicial HDXs.

\begin{definition}
    Let $X$ be a rank-$r$ Grassmannian complex in ambient vector space $\bF_2^k$. Define the rank-$(r+1)$ \textbf{Cayley simplicial complex} $\Cay(\bF_2^k,\beta(X))$ to have vertex set $\bF_2^k$, and rank-$i$ faces
  \begin{equation*}
    \Cay(\bF_2^k,\beta(X))(i) = \left\{\{v,v+v^{(0)},\dots,v+v^{(i)}\}:v,v^{(0)},\dots,v^{(i)}\in\bF_2^k,\;\{v^{(0)},\dots,v^{(i)}\}\in \beta(X)(i)\right\}.
  \end{equation*}
\end{definition}

We show that Cayley simplicial complexes inherit their local expansion from the generating Grassmannian complex:

\begin{lemma}[Informal statement of Lemma~\ref{lem:cayleyF2k}]
  \label{lem:cayleyF2kinf}
  Let $X$ be a rank-$r$ Grassmannian complex in ambient vector space $\bF_2^k$ such that $\spn X(0)=\bF_2^k$. Then for all $0\leq i\leq r+1$,
  \begin{equation*}
    \lambda^{(i)}(\Cay(\bF_2^k,\beta(X))) = \lambda^{(i-1)}(X),
  \end{equation*}
  so that in particular
  \begin{equation*}
    \lambda(\Cay(\bF_2^k,\beta(X))) \geq \frac{\lambda(X)}{1-\lambda(X)}.
  \end{equation*}
\end{lemma}

The proof of Lemma~\ref{lem:cayleyF2kinf} follows from the observation that for each vertex $v\in\Cay(\bF_2^k,\beta(X))$, the link $\Cay(\bF_2^k,\beta(X))_v$ is isomorphic to $\beta(X)$. The final expansion statement in the lemma is simply an application of the trickle-down Theorem~\ref{thm:tricklesimpinf}.

While Lemma~\ref{lem:cayleyF2kinf} shows that the Cayley complex $\Cay(\bF_2^k,\beta(X))$ has the same local expansion as $\beta(X)$, we show that $\Cay(\bF_2^k,\beta(X))$ can have much lower degree than $\beta(X)$. Specifically, while we were not able to construct low-degree Grassmannian HDXs (see Section~\ref{sec:openquestions}), we do construct Grassmannian HDXs $X$ that possess a different form of sparsity, namely, that only contain a small fraction of vectors in their ambient vector space. We then show that if $X$ possess this form of sparsity within the ambient vector space, the associated Cayley complex $\Cay(\bF_2^k,\beta(X))$ will have low degree.

Prior works have also studied Cayley simplicial complexes over $\bF_2^k$ in the context of high-dimensional expansion \cite{vadhan_construction_2018,conlon_hypergraph_2019,conlon_hypergraph_2020}. However, none of these works obtained local expansion $<1/2$, which is needed to apply the trickle-down Theorem~\ref{thm:tricklesimpinf} in order to achieve the notable ``local-to-global'' behavior associated with high-dimensional expansion.

Furthermore, \cite{vadhan_construction_2018} presents a coding-theoretic formulation of Cayley HDXs over $\bF_2^k$ similar to our framework described in Section~\ref{sec:codinginf} below; we further develop this connection and show additional coding theoretic implications in the regime where local expansion $<1/2$. For instance, in Proposition~\ref{prop:exptodisinf} we show that local expansion $<1/2$ for such a Cayley HDX implies good distance for an associated LDPC code. We also use our coding theoretic interpretation to study the 1-homology groups of Cayley HDXs over $\bF_2^k$ in Section~\ref{sec:homcharinf}.

\subsection{Construction of Grassmannian HDXs with low-rank matrices}
\label{sec:grassconstructinf}
We prove Theorem~\ref{thm:maincayleyinf} by applying Lemma~\ref{lem:cayleyF2kinf} to the Grassmannian complex in the following definition. Below, we think of $r,b$ as fixed constants while $n\rightarrow\infty$.

\begin{definition}[Informal statement of Definition~\ref{def:grassconstruct}]
  \label{def:grassconstructinf}
  Given integers $r\geq 1$, $b\geq 1$, $n\geq 2^{r+1}$, let $q=2^b$, and let $X=X^{r,b,n}$ be the rank-$r$ $\bF_2$-Grassmannian complex in ambient vector space $\bF_q^{n\times n}\cong\bF_2^{bn^2}$ defined as follows. Let $C_{\text{Had}}^{(r)}\subseteq\bF_2^{2^{r+1}-1}$ denote the Hadamard code of dimension $r+1$. Then the rank-$r$ faces in $X(r)$ are the $(r+1)$-dimensional $\bF_2$-subspaces that can be expressed in the form
  \begin{equation}
    \label{eq:Xrinf}
    \left\{\sum_{k=1}^{2^{r+1}-1}c_iM_i:(c_1,\dots,c_{2^{r+1}-1})\in C_{\text{Had}}^{(r)}\right\}
  \end{equation}
  for any rank-1 matrices $M_1,\dots,M_{2^{r+1}-1}\in\bF_q^{n\times n}$ whose sum $\sum_{i=1}^{2^{r+1}-1}M_i$ is a rank-$(2^{r+1}-1)$ matrix.
\end{definition}

Note that in Definition~\ref{def:grassconstructinf}, it is sufficient to define $X(r)$, which is a set of $(r+1)$-dimensional subspaces of $\bF_2^{bn^2}$, as then $X$ is the downward closure of $X(r)$. 

Our use of rank-1 matrices in Definition~\ref{def:grassconstructinf} stems from the coding theoretic interpretation described in Section~\ref{sec:codinginf}. In particular, we will show that there is a close relationship between Grassmannian HDXs and LDPC codes with many redundant parity checks, such as locally testable or locally correctable codes. If we view the rank-1 matrices in $\bF_q^{n\times n}$ as vectors in $\bF_q^{n^2}$, they form the rows of the generating matrix for the tensor product of two $n$-dimensional $q$-ary Hadamard codes. This code indeed has many redundant parity checks, as can be seen intuitively because it is similar in nature to a degree-2 Reed-Muller code, which is well known to be locally testable and locally correctable.

Recall that a rank-$0$ face of a $\bF_2$-Grassmannian complex is a 1-dimensional subspace $\{0,x_0\}$, which by abuse of notation may be viewed as a vertex given by the vector $x_0$. Then for $X=X^{r,b,n}$ the rank-0 faces $X(0)$ are precisely the rank-$2^r$ matrices in $\bF_q^{n\times n}$.

We now provide intuition for why we let $X(0)$ consist of rank-$2^r$ matrices, and for why we use the Hadamard code in the construction in Definition~\ref{def:grassconstructinf}. Recall that each face in $X(r)$ must be an $(r+1)$-dimensional subspace of $\bF_2^{bn^2}$ supported entirely on elements of $X(0)$. To ensure this requirement is satisfied in Definition~\ref{def:grassconstructinf}, we construct faces in $X(r)$ by summing together rank-1 matrices according to codewords of an $(r+1)$-dimensional Hadamard code. Because all such Hadamard codewords have weight exactly $2^r$, the resulting $(r+1)$-dimensional subspaces we construct consist entirely of rank-$2^r$ matrices (excluding the all-$0$s matrix), and thus are valid faces supported on $X(0)$.

The following example illustrates Definition~\ref{def:grassconstructinf} in the rank-1 case. Below, we let $\cM_q^n(s)\subseteq\bF_q^{n\times n}$ denote the set of rank-$s$ matrices in $\bF_q^{n\times n}$.

\begin{example}
  \label{ex:rank1}
  Given integers $b\geq 1$, $n\geq 2^{r+1}$, let $q=2^b$, and let $X=X^{1,b,n}$ be the $\bF_2$-Grassmannian complex in Definition~\ref{def:grassconstructinf}. Then
  \begin{align*}
    X(0) &= \cM_q^n(2) \\
    X(1) &= \left\{\{0,L_1+L_2,L_1+L_3,L_2+L_3\}:L_1,L_2,L_3\in\cM_q^n(1),\;L_1+L_2+L_3\in\cM_q^n(3)\right\}.
  \end{align*}
\end{example}

Theorem~\ref{thm:maincayleyinf} then follows from Lemma~\ref{lem:cayleyF2kinf} along with the following result showing that $X^{r,b,n}$ has good local expansion.

\begin{theorem}[Restatement of Corollary~\ref{cor:applytrickle}]
  \label{thm:applytrickleinf}
  For integers $r\geq 1$, $b\geq 5$, $n\geq 2^{r+1}$, let $q=2^b$, and let $X=X^{r,b,n}$ be the Grassmannian poset in Definition~\ref{def:grassconstructinf}. Then
  \begin{align*}
    \lambda(X) &\leq \frac{11}{q}.
  \end{align*}
\end{theorem}

In Section~\ref{sec:proofoverview}, we outline a proof of the rank $r=1$ case of Theorem~\ref{thm:applytrickleinf}, which consists of showing that the 1-skeleton graph of the complex $X^{1,b,n}$ described in Example~\ref{ex:rank1} has good expansion. The proof of this $r=1$ case is relatively concise, as it essentially consists of a coupling-style argument that relates $X^{1,b,n}$ to a random walk that we can prove mixes rapidly. These rank-1 Grassmannian expanders $X^{1,b,n}$ already give the rank-2 simplicial HDXs in Theorem~\ref{thm:maincayleyinf}.

The proof of the more general high-rank case of Theorem~\ref{thm:applytrickleinf} is more involved. By the trickle-down Theorem~\ref{thm:tricklegrassinf}, it suffices to prove expansion of the top-level links. For this purpose, for $x_{r-2}\in X^{r,b,n}(r-2)$, we express the 1-skeleton graph of the link $X_{x_{r-2}}$ as a tensor product $\bfG_1^{\otimes 2^{r-1}-1}\otimes\bfG_2$ for two specific graphs $\bfG_1,\bfG_2$ that we define (see Proposition~\ref{prop:tensornice}). We then show that both $\bfG_1,\bfG_2$ have good expansion. The analysis of $\bfG_2$ is similar to the analysis of the $r=1$ case described in Section~\ref{sec:proofoverview}. The analysis of $\bfG_1$ is more involved. We first consider the object $\cM_q^m$ defined above as a poset:

\begin{definition}
  \label{def:matrixposetinf}
  Given a vector space $\bF_q^m$, the \textbf{matrix poset $\cM_q^m$} is the graded poset consisting of all matrices in $\bF_q^{m\times m}$, where rank is given by matrix rank, that is $\cM_q^m(r)$ is the set of rank-$r$ matrices in $\bF_q^{m\times m}$, and $M_1\preceq M_2$ if and only if $\rank(M_2-M_1)=\rank(M_2)-\rank(M_1)$.
\end{definition}

It is a basic exercise to verify that the matrix poset is a well defined poset (see Lemma~\ref{lem:matrixposetok}).

We express the random walk on $\bfG_1$ as a sort of high-dimensional walk on the matrix poset. We then show that this walk has good expansion using a localization argument, in which we show that an appropriate form of local expansion implies global expansion of the desired walk. This localization technique has been developed for analyzing high-dimensional walks on HDXs, first for simplicial complexes \cite{kaufman_high_2017,dinur_high_2017,kaufman_high_2018,alev_improved_2020,gotlib_fine_2023} but also for more general posets \cite{kaufman_garlands_2022}. Our localization argument adapts such techniques, and specifically those of \cite{alev_improved_2020}, to the relevant walk on the matrix poset.

To the best of our knowledge, our paper is the first to consider the matrix poset in the context of high-dimensional expansion. We remark that while Kaufman and Tessler~\cite{kaufman_garlands_2022} apply localization arguments to show general expansion bounds on high-dimensional walks on posets, their results rely on the poset satisfying certain regularity conditions. The matrix poset does not satisfy these conditions; it in particular violates what \cite{kaufman_garlands_2022} calls ``$\wedge\rightarrow\vee$ regularity,'' which requires that for every $i$, there is some constant $N_i$ such that for all distinct $x_i,x_i'\in X(i)$ and $x_{i+1}\in X(i+1)$ where $x_i,x_i'\prec x_{i+1}$, there are exactly $N_i$ faces $x_{i-1}\in X(i-1)$ satisfying $x_{i-1}\prec x_i,x_i'$. For instance, in simplicial and Grassmannian complexes, all $N_i=1$, as there is always a unique such $x_{i-1}$, and it is given by $x_{i-1}=x_i\cap x_i'$. The following example shows that the matrix poset violates this regularity condition.

\begin{example}
  In the matrix poset $\cM_q^3$ we have
  \begin{equation*}
    \begin{pmatrix}1&0&0\\0&0&0\\0&0&0\end{pmatrix},\begin{pmatrix}1&1&0\\0&0&0\\0&0&0\end{pmatrix} \prec \begin{pmatrix}1&0&0\\0&1&0\\0&0&0\end{pmatrix},\begin{pmatrix}1&0&0\\0&1&0\\0&1&0\end{pmatrix} \prec \begin{pmatrix}1&0&0\\0&1&0\\0&0&1\end{pmatrix},
  \end{equation*}
  whereas 
  \begin{equation*}
    \begin{pmatrix}1&0&0\\0&0&0\\0&0&0\end{pmatrix} \prec \begin{pmatrix}1&0&0\\0&1&0\\0&0&0\end{pmatrix},\begin{pmatrix}1&0&0\\0&0&0\\0&0&1\end{pmatrix} \prec \begin{pmatrix}1&0&0\\0&1&0\\0&0&1\end{pmatrix},
  \end{equation*}
  where the matrix on the left above is the unique rank-1 matrix dominated by the two rank-2 matrices in the middle. Thus in $\cM_q^3$, given a pair of distinct rank-2 matrices $x_2,x_2'\prec I_3$, there can be a variable number of rank-1 matrices $x_1\prec x_2,x_2'$ depending on the choice of $x_2,x_2'$.
\end{example}

Therefore it may seem that high-dimensional walks on $\cM_q^m$ cannot be analyzed using standard localization techniques. However, it turns out that the specific high-dimensional walk we analyze is not the standard ``up-down walk'' considered in \cite{kaufman_garlands_2022} (see Definition~\ref{def:highwalks}), but is rather a modified walk for which the necessary regularity conditions hold.

We believe the high-dimensional expansion of the matrix poset may be of independent interest, as it does not seem to be captured by prior works, and yet lies at the core of our HDX construction. Many open questions remain in this topic that we do not address, such as the expansion of the up-down walks (Definition~\ref{def:highwalks}).

\subsection{Coding theoretic interpretation}
\label{sec:codinginf}
In this section, we describe our coding theoretic interpretation of rank-1 Grassmannian expanders. This viewpoint motivates our construction in Definition~\ref{def:grassconstructinf}, and also forms the basis for our analysis of homology groups described in Section~\ref{sec:homcharinf}.

Let $X$ be a $\bF_2$-Grassmannian complex in ambient vectors space $\bF_2^k$. We define two matrices naturally associated to $X$:
\begin{itemize}
\item Define a generating matrix $G_X\in\bF_2^{X(0)\times k}$ to have $x_0$th row equal to $x_0\in X(0)$.
\item Define a parity check $H_X\in\bF_2^{X(1)\times X(0)}$ to have $x_1$th row equal to $\1_{x_1\setminus\{0\}}$ for $x_1\in X(1)$.
\end{itemize}
Thus $X$ has two naturally associated codes $\im G_X,\ker H_X\subseteq\bF_2^{X(0)}$. Observe that as each $x_1\in X(1)$ is a 2-dimensional $\bF_2$-vector space, each row in $H_X$ has $|x_1\setminus\{0\}|=3$ nonzero entries, so $\ker H_X$ is LDPC with parity checks of weight 3. Because the three nonzero elements of any face $x_1\in X(1)$ sum to $0$, we have $H_XG_X=0$, that is, $\im G_X\subseteq\ker H_X$, or equivalently,
\begin{equation}
  \label{eq:HsubG}
  \im H_X^\top\subseteq\ker G_X^\top.
\end{equation}

To understand the connection to coding theory, consider that the rate of $\im G_X$ is the ratio of $k$ to $|X(0)|$, which is a measure of the sparsity of $X$ within its ambient vector space. For instance, assuming that $\bF_2^k=\spn X(0)$, then if $X$ is optimally sparse, $|X(0)|$ could be as small as $O(k)$, which would mean $\im G_X$ has constant rate. Meanwhile, if $X$ is maximally dense in $\bF_2^k$, then $|X(0)|=2^k$ is exponentially large in $k$, and $\im G_X$ has inverse exponential rate.

The above coding theoretic formulation of rank-1 Grassmannian complexes was implicit in \cite{vadhan_construction_2018}, though the presentation in \cite{vadhan_construction_2018} did not use the language of Grassmannian complexes, and rather focused on constructing 2-dimensional Cayley simplicial complexes over the group $\bF_2^k$.

Below, we further develop this coding theoretic perspective. Specifically, we apply the trickle-down theorem to show that the distance of the code $\im G_X$ is determined by the expansion of $X$.

\begin{proposition}[Informal statement of Proposition~\ref{prop:exptodis}]
  \label{prop:exptodisinf}
  Let $X$ be a $\bF_2$-Grassmannian complex in ambient vector space $\bF_2^k=\spn X(0)$, such that all vertices have the same weight. Let $\lambda=\lambda^{(-1)}(X)$. Then the Hamming weight of every nonzero codeword $c\in\im G_X$ satisfies
  \begin{equation*}
    \frac12-\frac{\lambda}{2(1-\lambda)} \leq \frac{|c|}{|X(0)|} \leq \frac12+\frac{\lambda}{2(1-\lambda)}.
  \end{equation*}
\end{proposition}
\begin{proof}
  By Lemma~\ref{lem:cayleyF2kinf}, $\Cay(\spn X(0),\beta(X))$ has rank-0 local expansion $\lambda$, so by the trickle-down Theorem~\ref{thm:tricklesimpinf}, $\Cay(\spn X(0),\beta(X))$ has rank-$(-1)$ local expansion $\leq\lambda/(1-\lambda)$. But the 1-skeleton of $\Cay(\spn X(0),\beta(X))$ is by definition the Cayley graph $\Cay(\spn X(0),X(0))$. The desired result then follows by the well known equivalence (see Lemma~\ref{lem:epbias}) between Cayley $\epsilon$-expanders over $\bF_2^k$, $\epsilon$-biased sets, and $\epsilon$-balanced codes.
\end{proof}

Thus we have reduced the problem of constructing rank-2 Cayley simplicial HDXs, or equivalently rank-1 Grassmannian expanders, to the problem of constructing LDPC codes with weight-3 checks of good rate and distance, such that the graph induced by the checks is expanding.

To get good expansion here, the Alon-Boppana bound implies that there must be many more than $|X(0)|$ rank-1 faces, that is, the desired codes must have many redundant parity checks. Thus a natural place to look for such codes is locally testable or locally correctable codes, which have such redundant low-weight parity checks. Our construction in Definition~\ref{def:grassconstructinf} is then similar in flavor to the tensor product of two Hadamard codes, which is similar to a degree-2 Reed Muller code. Indeed, as described in Section~\ref{sec:grassconstructinf} above, our construction is built from rank-$2^r$ matrices in $\bF_q^{n\times n}$; the rank-$1$ matrices in $\bF_q^{n\times n}$ form the rows of a generating matrix for the tensor product of two $q$-ary Hadamard codes.

It is an interesting question whether more sophisticated codes with locality properties could give better HDXs. Indeed, recent breakthrough constructions of asymptotically good locally testable codes (LTCs) \cite{dinur_locally_2022,panteleev_asymptotically_2022} may appear to be promising candidates for our purposes. However, we were unable to obtain Grassmannian complexes with good expansion from these LTCs. The difficulty arises from the fact that these LTCs have fewer redundant parity checks, which results in fewer edges in the associated Grassmannian complexes, than the Reed-Muller/tensor product of Hadamard codes.

\subsection{Homology characterization}
\label{sec:homcharinf}
In this section, we explain how we apply the above coding theoretic interpretation to characterize the 1-homology of the Cayley simplicial complexes in Section~\ref{sec:cayleyinf}. We then describe how we use this characterization to prove Theorem~\ref{thm:bighominf}.

We show below that the extent to which the two spaces in~(\ref{eq:HsubG}) differ is closely related to the 1-homology of the Cayley complex generated by $X$.

\begin{theorem}[Informal statement of Theorem~\ref{thm:hommodswap}]
  \label{thm:hommodswapinf}
  Let $X$ be a $\bF_2$-Grassmannian complex and let $Y=\Cay(\spn X(0),\beta(X))$. Then there is a natural surjective homomorphism
  \begin{equation*}
    H_1(Y) \rightarrow \ker G_X^\top/\im H_X^\top.
  \end{equation*}

  Futhermore, the kernel of this homomorphism is generated by those 1-cycles of the form
  \begin{equation*}
    \1_{\{\{v,v+x_0\},\{v+x_0,v+x_0+x_0'\},\{v+x_0',v+x_0+x_0'\},\{v,v+x_0'\}\}}
  \end{equation*}
  for $v\in\spn X(0)$ and for $x_0,x_0'\in X(0)$.
\end{theorem}

Observe that $\ker G_X^\top$ is the space of linear dependencies satisfied by elements of $X(0)$, while $\im H_X^\top$ is the space of linear dependencies generated by faces in $X(1)$, where $\{0,x_0,x_0',x_0''\}\in X(1)$ generates the constraint $x_0+x_0'+x_0''=0$. Thus Theorem~\ref{thm:hommodswapinf} implies that every linear dependency satisfied by vertices in $X(0)$ (i.e.~every element of $\ker G_X^\top$) that is not implied by the face incidence structure of $X$ (i.e.~not generated by the constraints $x_0+x_0'+x_0''=0$ for $\{0,x_0,x_0',x_0''\}\in X(1)$) induces nontrivial homology.

We prove Theorem~\ref{thm:bighominf} by applying quotients to artifically introduce linear dependencies among the vertices of a complex $X$. Specifically, if $X$ is an $\bF_2$-Grassmannian complex with $|\spn X(0)|\gg|X(0)|^2$, then we can find a 1-dimensional subspace $V\subseteq\spn X(0)$ such that quotienting by $V$ induces a one-to-one mapping on $X(0)$, so we obtain a quotiented complex $X'$ with the same poset incidence structure as $X$, so that $H_{X'}=H_X$ while $\ker G_X^\top$ is a codimension-1 subspace of $\ker G_{X'}^\top$. Theorem~\ref{thm:hommodswapinf} then implies that $\dim H_1(\Cay(\spn X'(0),\beta(X'(0))))\geq 1$. If we repeat this quotienting procedure $t$ times, we obtain a quotiented complex $X'$ with $\dim H_1(\Cay(\spn X'(0),\beta(X'(0))))\geq t$.

To prove Theorem~\ref{thm:bighominf}, we set our initial complex $X$ to be our construction in Definition~\ref{def:grassconstructinf}, and then we perform $t=bn^2-2^{r+3}bn$ iterations of the quotienting procedure to obtain the desired complex $X'$. We cannot perform more quotients because $X'$ is sufficiently dense in its ambient vector space that $|\spn X'(0)|\not\gg|'X(0)|^2$.

\subsection{Overview of expansion proof}
\label{sec:proofoverview}
In this section, we outline the proof of the $r=1$ case of Theorem~\ref{thm:applytrickleinf}. As described in Section~\ref{sec:grassconstructinf}, this proof for $r=1$ is simpler than the proof for general $r$, so by Lemma~\ref{lem:cayleyF2kinf} we obtain a simpler construction of rank-2 simplicial HDXs of subpolynomial degree.

Our goal is to show that the 1-skeleton graph $\bfG_X$ of the rank-1 $\bF_2$-Grassmannian complex $X=X^{1,b,n}$ described in Example~\ref{ex:rank1} is a good expander. Letting $\cM=\cM_q^n$ denote the matrix poset, recall that $\bfG_X$ is defined by
\begin{align*}
  V(\bfG_X) &= \cM_q^n(2) \\
  E(\bfG_X) &= \left\{\{L_1+L_2,L_1+L_3\}:L_1,L_2,L_3\in\cM_q^n(1),\;L_1+L_2+L_3\in\cM_q^n(3)\right\}.
\end{align*}
Note that our notation in the definition of $E(\bfG_X)$ above allows for permuting $L_1,L_2,L_3$, so the edges $\{L_1+L_2,L_2+L_3\}$ and $\{L_1+L_3,L_2+L_3\}$ also lie in $E(\bfG_X)$.

To show that $\bfG_X$ is a good expander, we will relate it to the following random walks:
\begin{itemize}
\item The \textbf{up-walk} $\bfW_{\cM(1)}^{\uparrow}\in\bR^{\cM(1)\times \cM(2)}$ that walks from $M\in\cM(1)$ to a uniformly random $M'\in\cM(2)$ such that $M\prec M'$.
\item The \textbf{down-walk} $\bfW_{\cM(2)}^{\downarrow}\in\bR^{\cM(2)\times \cM(1)}$ that walks from $M\in\cM(2)$ to a uniformly random $M'\in\cM(1)$ such that $M\succ M'$.
\item The \textbf{up-down} walk $\bfW_{\cM(1)}^{\ud}=\bfW_{\cM(1)}^{\uparrow}\bfW_{\cM(2)}^{\downarrow}\in\bR^{\cM(1)\times\cM(1)}$ that performs an up-walk step followed by a down-walk step.
\item The \textbf{down-up} walk $\bfW_{\cM(2)}^{\du}=\bfW_{\cM(2)}^{\downarrow}\bfW_{\cM(1)}^{\uparrow}\in\bR^{\cM(2)\times\cM(2)}$ that performs a down-walk step followed by an up-walk step.
\end{itemize}
Note that these walks generalize to arbitrary ranks and posets; see Section~\ref{sec:posetprelim}.

To show that $\bfG_X$ has good expansion, we first show that the random walk $\bfW_X$ on $\bfG_X$ is closely correlated to the down-up walk defined above, so that the expansion of $\bfG_X$ is close to the expansion of the down-up walk. Now the expansion of the down-up walk by definition equals the expansion of the up-down walk, as $\bfW_{\cM(2)}^{\du}=\bfW_{X(i)}^{\downarrow}\bfW_{X(i-1)}^{\uparrow}\in\bR^{\cM(2)\times\cM(2)}$ and $\bfW_{\cM(1)}^{\ud}=\bfW_{X(i)}^{\uparrow}\bfW_{X(i+1)}^{\downarrow}\in\bR^{\cM(1)\times\cM(1)}$ must have the same nonzero eigenvalues. We finally show that the up-down walk is closely correlated to the random walk on a nearly complete graph, so that the up-down walk must have good expansion. Combining the above statements, we deduce that $\bfG_X$ has good expansion.

The following technical lemma formalizes the idea that closely correlated random walks have similar expansion.

\begin{lemma}[Restatement of Lemma~\ref{lem:expclose}]
  \label{lem:expcloseinf}
  For vertex sets $V\subseteq V'$, let $\bfW\in\bR^{V\times V}$ and $\bfW'\in\bR^{V'\times V'}$ be symmetric random walk matrices and let $\epsilon>0$ be such that for every $v\in V$, the distributions obtained from starting at vertex $v$ and taking a step in $\bfW$ vs.~in $\bfW'$ differ by a total variation distance of at most $\epsilon/2$. Then
  \begin{equation*}
    \lambda(\bfW)\leq\lambda(\bfW')+\epsilon.
  \end{equation*}
\end{lemma}

Using Lemma~\ref{lem:expcloseinf}, we prove the two necessary expansion bounds as described above:

\begin{lemma}[Implied by Proposition~\ref{prop:G2exp}]
  \label{lem:G2expinf}
  It holds that $\lambda(\bfW_X)\leq\lambda(\bfW_{\cM(2)}^{\du})+1/q^{\Omega(n)}$.
\end{lemma}
\begin{proof}[Proof sketch]
  Given $M\in\cM(2)$, by definition we may sample $M'$ from the distribution obtained by starting at $M$ and taking a step in $\bfW_{\cM(2)}^{\du}$ as follows. First sample a uniformly random element $L_1\in\cM(1)$ such that $L_1\prec M$, and let $L_2=M-L_1$. Then sample a uniformly random $L_3\in\cM(1)$ such that $\rank(L_1+L_3)=2$, and let $M'=L_1+L_3$.

  Let $E$ be the event that $\rank(L_1+L_2+L_3)=3$. Conditioned on $E$, then $M'$ is distributed as a step in $\bfW_X$ starting from $M$. Thus by Lemma~\ref{lem:expcloseinf},
  \begin{align*}
    \lambda(\bfW_X)
    &\leq \lambda(\bfW_{\cM(2)}^{\du})+\Pr[E] \\
    &\leq \lambda(\bfW_{\cM(2)}^{\du})+\frac{1}{q^{\Omega(n)}}.
  \end{align*}
\end{proof}

\begin{lemma}[Informal statement of Lemma~\ref{lem:mat1ud}]
  \label{lem:mat1udinf}
  It holds that $\lambda(\bfW_{\cM(1)}^{\ud}) \leq O(1/q)$.
\end{lemma}
\begin{proof}[Proof sketch]
  Given $L=e_1e_2^\top\in\cM(1)$ for $e_1,e_2\in\bF_q^n$, we may sample $L'$ from the distribution obtained by starting at $L$ and taking a step in $\bfW_{\cM(1)}^{\ud}$ as follows. Sample uniformly random vectors $e_1^{(1)}\in\bF_q^n\setminus\spn\{e_1\}$ and $e_2^{(1)}\in\bF_q^n\setminus\spn\{e_2\}$. Then sample a uniformly random pair of elements $(\alpha_1,\beta_1),(\alpha_2,\beta_2)\in\bF_q^2\setminus\{(0,0)\}$ for which $\alpha_1\alpha_2+\beta_1\beta_2=1$, and let $e_1^{(2)}=\alpha_1e_1+\beta_1e_1^{(1)}$, $e_2^{(2)}=\alpha_2e_2+\beta_2e_2^{(1)}$, and $L'=e_1^{(2)}{e_2^{(2)}}^\top$.

  By definition the distribution of $\spn\{e_1e_2^\top,e_1^{(1)}{e_2^{(2)}}^\top\}$ is simply that of an up-walk step starting at $L=e_1e_2^\top$, and we show that (see Lemma~\ref{lem:domI}) the distribution of $L'$ is that of a down-walk step starting at $\spn\{e_1e_2^\top,e_1^{(1)}\otimes {e_2^{(2)}}^\top\}$, so $L'$ is indeed distributed as an up-down walk step starting at $L$.

  Let $E$ be the event that $\alpha_1,\alpha_2=0$. Conditioned on $E^C$, then $e_1^{(2)}$ and $e_2^{(2)}$ are independently distributed uniformly over the respective sets $\bF_q^n\setminus\spn\{e_1\}$ and $\bF_q^n\setminus\spn\{e_2\}$, so the distribution of $L'$ conditioned on $E^C$ is that of a step in the tensor product of two nearly complete graphs on $\bF_q^n\setminus\{0\}$, which has expansion $o(1)$ as $n\rightarrow\infty$. Thus by Lemma~\ref{lem:expcloseinf}, 
  \begin{align*}
    \lambda(\bfW_{\cM(1)}^{\ud})
    &\leq o(1)+2\Pr[E] \\
    &\leq O\left(\frac{1}{q}\right).
  \end{align*}
\end{proof}

Combining the above lemmas, we immediately obtain the $r=1$ case of Theorem~\ref{thm:applytrickleinf}:

\begin{proof}[Proof of $r=1$ case of Theorem~\ref{thm:applytrickleinf}]
  We have that
  \begin{align*}
    \lambda(X)
    &= \lambda(\bfW_X) \\
    &\leq \lambda(\bfW_{\cM(2)}^{\du})+1/q^{\Omega(n)} \\
    &= \lambda(\bfW_{\cM(1)}^{\ud})+1/q^{\Omega(n)} \\
    &\leq O(1/q),
  \end{align*}
  where the first equality above holds because $X$ has rank $1$, the first inequality holds by Lemma~\ref{lem:G2expinf}, the second equality holds because $\bfW_{\cM(2)}^{\du}$ and $\bfW_{\cM(1)}^{\ud}$ by definition have the same nonzero eigenvalues, and the second inequality holds by Lemma~\ref{lem:mat1udinf}.
\end{proof}

\subsection{Open questions}
\label{sec:openquestions}
Our work raises the following open questions:
\begin{itemize}
\item Are there constructions of Grassmannian HDXs that are sparser within their ambient vector spaces than our construction in Definition~\ref{def:grassconstructinf}? Specifically, is the optimal sparsity $|X(0)|=\Theta(\dim\spn X(0))$ achievable? Such objects would yield logarithmic degree Cayley simplicial HDXs by Lemma~\ref{lem:cayleyF2k}, asymptotically good binary LDPC codes with check weight 3 and distance near $1/2$ by the results in Section~\ref{sec:coding}, and simplicial HDXs with polynomially large 1-homology dimension by Corollary~\ref{cor:increasehom}.
\item There have been various works on constructing codes from high-dimensional expanders, e.g.~\cite{evra_decodable_2020,kaufman_new_2021,dikstein_locally_2020-1,first_good_2022}, including recent breakthrough constructions of locally testable codes and quantum LDPC codes \cite{dinur_locally_2022,panteleev_asymptotically_2022,leverrier_quantum_2022-1,dinur_good_2023}. Our coding theoretic interpretation of Grassmannian HDXs and the relationship to simplicial Cayley HDXs suggests the possibility of a two-way relationship. That is, can good HDXs be constructed from codes such as locally testable or quantum LDPC codes?
\item Are there constructions of low-degree Grassmannian HDXs? We obtain our subpolynomial-degree simplicial HDXs in Theorem~\ref{thm:maincayleyinf} by taking the Cayley simplicial complex generated by our Grassmannian HDXs in Definition~\ref{def:grassconstructinf}, which in turn are sparse within their ambient vector space. However, the Grassmannian HDXs we construct have polynomially large degree. Low-degree Grassmannian HDXs could be interesting of their own right (see Conjecture~8.29 of \cite{dikstein_boolean_2022}), and their basisification would furthermore immediately yield low-degree simplicial HDXs, without the need for passing to Cayley complexes.
\item Our study of the expansion of graphs based on low-rank matrices (Definition~\ref{def:grassconstructinf}) and the matrix poset (Definition~\ref{def:matrixposetinf}) seems similar in spirit to prior work on the short code that is often studied in the context of the unique games conjecture \cite{barak_making_2015,barak_small-set_2018}. Is there a way to formalize this connection, or to apply our results to unique games or hardness of approximation? For instance, the degree-2 shortcode graph studied in \cite{barak_small-set_2018} has vertices given by matrices, and edges given by pairs of matrices whose difference is a rank-1 matrix. While we do not study this precise graph, it is similar to the matrix poset in Definition~\ref{def:matrixposetinf}. Can tools from high-dimensional expanders be leveraged in the analysis of this shortcode graph?
\end{itemize}


\subsection{Organization}
The organization of the remainder of this paper is as follows. Section~\ref{sec:prelim} presents definitions and describes some prior work. Section~\ref{sec:caycomplex} presents the notion of Cayley simplicial complexes, and shows how they can be generated by Grasmannian HDXs. Section~\ref{sec:grassHDX} presents our main construction of Grassmannian HDXs using low-rank matrices, and analyzes their sparsity and expansion. Section~\ref{sec:cayconstruct} analyzes the Cayley simplicial complexes generated by these Grassmannian HDXs. Section~\ref{sec:coding} presents our coding theoretic interpretation of Grassmannian HDXs. Section~\ref{sec:cayleyhom} presents our general characterization of 1-homology of Cayley simplicial complexes, and applies it to our construction.

\section{Preliminaries}
\label{sec:prelim}
This section presents preliminary definitions and results.

\subsection{Expander graphs}
We begin by defining random walk matrices and spectral expansion for graphs.

\begin{definition}
  Let $\bfG=(V,E,m)$ be a graph with weight function $m:E\rightarrow\bR_{\geq 0}$. For a vertex $v\in V$, denote the degree by $\deg(v)=\sum_{v'\in V}m(v,v')$. Then the \textbf{random walk matrix $\bfW\in\bR^{V\times V}$} of $\bfG$ is the matrix with
  \begin{equation*}
    \bfW_{v,v'} = \frac{m(v,v')}{\deg(v)}.
  \end{equation*}
  Let $\bfW$ have eigenvalues $1=\lambda_1\geq\lambda_2\geq\cdots\geq\lambda_{|V|}$. The \textbf{spectral expansion } (or simply \textbf{expansion}) of $\bfG$ is the quantity
  \begin{equation*}
    \lambda(\bfG) = \lambda(\bfW) = \max\{|\lambda_2|,|\lambda_{|V|}|\}.
  \end{equation*}
\end{definition}

In some of our expansion analysis, we will use the following notion of a graph projection.

\begin{definition}
  \label{def:graphproj}
  For directed weighted graphs $\bfG,\bfG'$, a \textbf{graph projection $\Pi:\bfG'\rightarrow\bfG$} is a surjective vertex mapping $\Pi:V(\bfG')\rightarrow V(\bfG)$ such that
  \begin{equation*}
    m_{\bfG}(v_1,v_2)=\sum_{v_1'\in\Pi^{-1}(v_1),v_2'\in\Pi^{-1}(v_2)}m_{\bfG'}(v_1',v_2').
  \end{equation*}
\end{definition}

The definition above naturally extends to undirected graphs by simply replacing each undirected edge with a directed copy in each direction. Such projections of undirected graphs can only improve expansion, as stated in the well-known lemma below.

\begin{lemma}[Well known]
  \label{lem:projexp}
  If $\Pi:\bfG'\rightarrow\bfG$ is a graph projection of undirected graphs $\bfG',\bfG$, then $\lambda(\bfG)\leq\lambda(\bfG')$.
\end{lemma}

We omit the proof of Lemma~\ref{lem:projexp}, which follows from the observation that for each eigenvector of the random walk matrix of $\bfG$, we may construct an eigenvector of the random walk matrix of $\bfG'$ with the same eigenvalue.

\subsection{Posets and high-dimensional expansion}
\label{sec:posetprelim}
We now turn to defining high-dimensional expansion. While the literature on high-dimensional expanders traditionally focused on simplicial complexes, recent work has generalized such expansion notions to more general posets. For such posets, we generally follow the definitions in \cite{kaufman_garlands_2022}.

\begin{definition}
  \label{def:poset}
  We introduce the following notions regarding posets:
  \begin{itemize}
  \item A \textbf{poset $X$} is a set $X$ with a binary relation $\prec$, which specifies a subset of $X\times X$ denoted by pairs $x,x'$ with $x\prec x'$, such that if $x\prec x'$ and $x'\prec x''$, then $x\prec x''$ and $x'\not\prec x$. We write $x\preceq x'$ if either $x\prec x'$ or $x=x'$. We sometimes refer to poset elements $x\in X$ as \textbf{faces}.
  \item If $x\preceq x'$ we say that \textbf{$x'$ dominates $x$}.
  \item A poset $X$ is \textbf{graded} if it has a rank function $\rank:X\rightarrow\bZ_{\geq-1}$ such there is a single element in $X$ of rank $-1$, and such that for any $x\prec x'$ for which there is no $x''$ satisfying $x\prec x''\prec x'$, then $\rank(x)=\rank(x')-1$. We let $X(i)=\{x\in X:\rank(x)=i\}$ denote the rank-$i$ elements of $X$. The \textbf{rank of $X$}, denoted $\rank(X)$, is the maximum rank of any element $x\in X$. We sometimes refer to rank-0 elements $x\in X(0)$ as \textbf{vertices}.
  \item A graded poset $X$ is \textbf{pure} if each element $x\in X$ is dominated by some element of rank $r=\rank(X)$.
  \item A pure graded poset $X$ is \textbf{weighted} if $X$ has a weight function $m_X:X\rightarrow[0,1]$ such that for every $-1\leq i\leq r$, the restriction $m_{X(i)}=m_X|_{X(i)}$ is a probability distribution on the set of rank-$i$ elements.
  \item The weight function $m_X$ is \textbf{standard} if for every $-1\leq i\leq r-1$, we may sample $x_i\sim m_{X(i)}$ by first sampling $x_{i+1}\sim m_{X(i+1)}$, and then sampling $x_i\sim\Unif\{x_i'\in X(i):x_i'\prec x_{i+1}\}$. Observe that a standard weight function $m_X$ is determined by its distribution $m_{X(r)}$ on top-rank faces.
  \end{itemize}
\end{definition}

In this paper, we restrict attention to pure graded weighted posets with standard weight functions, and simply refer to such objects as ``posets.'' Nearly all of the posets we consider have uniform weight function, meaning each $m_{X(i)}$ is the uniform distribution on $X(i)$. Thus when a weight function is not explicitly stated, it is assumed to be uniform.

Posets come with the following natural random walk operators.

\begin{definition}
  \label{def:highwalks}
  For a poset $X$ (with standard weight function $m_X$), we have the following random walk operators:
  \begin{itemize}
  \item The \textbf{up-walk} $\bfW_{X(i)}^{\uparrow}\in\bR^{X(i)\times X(i+1)}$ given by
    \begin{equation*}
      (\bfW_{X(i)}^{\uparrow})_{x_i,x_{i+1}} = \1_{x_i\prec x_{i+1}}\cdot\frac{m_X(x_{i+1})}{m_X(x_i)\cdot|\{x_i'\in X(i):x_i'\prec x_{i+1}\}|}.
    \end{equation*}
  \item The \textbf{down-walk} $\bfW_{X(i)}^{\downarrow}\in\bR^{X(i)\times X(i-1)}$ given by
    \begin{equation*}
      (\bfW_{X(i)}^{\downarrow})_{x_i,x_{i-1}} = \1_{x_i\succ x_{i-1}}\cdot\frac{1}{|\{x_{i-1}'\in X(i-1):x_i\succ x_{i-1}\}|}.
    \end{equation*}
  \item The \textbf{up-down} walk $\bfW_{X(i)}^{\ud}\in\bR^{X(i)\times X(i)}$ given by
    \begin{equation*}
      \bfW_{X(i)}^{\ud} = \bfW_{X(i)}^{\uparrow}\bfW_{X(i+1)}^{\downarrow}.
    \end{equation*}
  \item The \textbf{down-up} walk $\bfW_{X(i)}^{\du}\in\bR^{X(i)\times X(i)}$ given by
    \begin{equation*}
      \bfW_{X(i)}^{\ud} = \bfW_{X(i)}^{\downarrow}\bfW_{X(i-1)}^{\uparrow}.
    \end{equation*}
  \end{itemize}
\end{definition}

Observe that if we sample a joint distribution over $x_i\in X(i)$ for $-1\leq i\leq r$ by first drawing $x_r\sim m_{X(r)}$ and then drawing each $x_i$ uniformly at random from the set of rank-$i$ elements dominated by $x_{i+1}$, we obtain the up and down walk transition probabilities as
\begin{align*}
  (\bfW_{X(i)}^{\uparrow})_{x_i',x_{i+1}'} &= \Pr[x_{i+1}=x_{i+1}'|x_i=x_i'] \\
  (\bfW_{X(i)}^{\downarrow})_{x_i',x_{i-1}'} &= \Pr[x_{i-1}=x_{i-1}'|x_i=x_i'].
\end{align*}

A notable feature of high-dimensional expanders is that global expansion, such as expansion of the up-down and down-up walks, can be deduced from local expansion, the latter of which captures expansion at local parts of a poset. To formally define this notion of local expansion, we require the following definitions. The local structure of a poset is captured by its links:

\begin{definition}
  For a rank-$r$ poset $x$ and an element $x\in X$, the \textbf{link $X_x$} of $x$ is the poset of rank $r-\rank(x)-1$ given by
  \begin{align*}
    X_x &= \{x'\in X:x\preceq x'\} \\
    \rank_{X_x}(x') &= \rank_X(x')-\rank_X(x)-1 \\
    m_{X_x}(x') &= \frac{m_X(x')}{\sum_{x''\in X_x:\rank(x'')=\rank(x')}m_X(x'')}.
  \end{align*}
\end{definition}

Expansion (and connectedness) at a link will be measured by the expansion (and connectedness) of the associated 1-skeleton graph, defined below.

\begin{definition}
  The \textbf{$i$-skeleton} of a poset $X$ is the subposet $\{x\in X:\rank(x)\leq i\}$. The \textbf{1-skeleton graph $\bfG_X$} of a poset $X$ with standard weight function $m_X$ is the undirected weighted graph defined by
  \begin{align*}
    V(\bfG_X) &= X(0) \\
    E(\bfG_X) &= \{\{x_0,x_0'\}:x_0\neq x_0'\in X(0),\; \exists x_1\in X(1):x_0,x_0'\prec x_1\} \\
    m_{\bfG_X}(\{x_0,x_0'\}) &= \sum_{x_1\in X(1):x_0,x_0'\prec x_1}m_X(x_1).
  \end{align*}
  We say that $X$ is \textbf{connected} if its 1-skeleton graph $\bfG_X$ is connected.
\end{definition}

We can now define local expansion.

\begin{definition}
  For a rank-$r$ poset $X$, the \textbf{rank-$i$ local expansion $\lambda^{(i)}(X)$} is defined to be
  \begin{equation*}
    \lambda^{(i)}(X) = \max_{x_i\in X(i)}\{\lambda(\bfG_{X_{x_i}})\}.
  \end{equation*}
  The \textbf{global expansion} refers to the rank-$(-1)$ local expansion $\lambda^{(-1)}(X)$, and the \textbf{local expansion $\lambda(X)$} is defined to be
  \begin{equation*}
    \lambda(X) = \max_{-1\leq i\leq r-2}\{\lambda^{(i)}(X)\}.
  \end{equation*}
\end{definition}

A key property of local expansion is the trickle-down theorem, which says that good local expansion in high ranks implies good local expansion in lower ranks. This result was proven by Oppenheim~\cite{oppenheim_local_2018} for simplicial complexes, and subsequently generalized by Kaufman and Tessler~\cite{kaufman_garlands_2022} to more general posets satisfying certain regularity conditions. As their fully general result requires more notation, we will simply state the two cases of the trickle-down theorem that are relevant for the purposes of this paper, namely, the result for simplicial complexes, and for Grassmannian complexes. However, we must first introduce these types of posets.

\begin{definition}
  A \textbf{simplicial complex} $X$ on vertex set $V$ is a poset whose elements consist of a subset $X\subseteq 2^V$, such that if $x\in X$ then all subsets $x'\subseteq x$ also satisfy $x'\in X$. The partial ordering is given by set inclusion, that is $x\preceq x'$ if and only if $x\subseteq x'$. The rank function is given by $\rank(x)=|x|-1$, and the term ``dimension'' is sometimes used as a synonym for rank, so that $\rank(x)=\dim(x)$.
\end{definition}

\begin{definition}
  A \textbf{$\bF_q$-Grassmannian complex} $X$ in ambient vector space $\bF_q^k$ is a poset whose elements are linear subspaces of $\bF_q^k$, such that if $x\in X$ then all linear subspaces $x'\subseteq x$ also satisfy $x'\in X$. The partial ordering is given by subspace inclusion, that is $x\preceq x'$ if and only if $x\subseteq x'$. The rank of a linear subspace is one less than its dimension, that is $\rank(x)=\dim(x)-1$.
\end{definition}

Note that links in both simplicial and Grassmannian complexes are complexes of the same respective type. Specifically, if $X$ is a simplicial complex, then the link $X_x$ of a face $x\in X$ is isomorphic to the simplicial complex $\{x'\setminus x:x'\in X,x\subseteq x'\}$. Meanwhile, if $X$ is a Grassmannian complex in ambient vector space $\bF_q^k$, then the link $X_x$ of a face $x\in X$ is isomorphic to the Grassmannian complex $\{x'/x:x'\in X,x\subseteq x'\}$ in ambient vector space $\bF_q^k/x$.

In the special case of $\bF_2$-Grassmannian complexes, the vertices are 1-dimensional subspaces $\{0,x\}\subseteq\bF_2^k$ for $x\in\bF_2^k$. Here we often abuse notation and simply refer to a vertex by the nonzero element $x\in\bF_2^k$.

In this paper, we will consider three different types of posets: simplicial complexes, Grassmannian complexes and the matrix poset given in Definition~\ref{def:matrixposet}. Simplicial complexes have been widely studied in mathematics and have traditionally been the principal object of study in the high-dimensional expander literature, while Grassmannian complexes have recently gained interest in this literature \cite{dikstein_boolean_2018,kaufman_garlands_2022}. To the best of our knowledge, we are the first to study the matrix poset in the context of high-dimensional expanders. Interestingly, the matrix poset is sufficiently irregular that local-to-global results from \cite{kaufman_garlands_2022} do not (at least directly) apply to it; for more details on this poset, see Section~\ref{sec:grassHDX}.

We now state the trickle-down theorems for simplicial and Grassmannian complexes.

\begin{theorem}[Trickle-down for simplicial complexes \cite{oppenheim_local_2018}]
  \label{thm:tricklesimp}
  Let $X$ be a rank-$r$ simplicial complex such that for every $x\in X$ with $\rank(x)\leq r-2$, the link $X_x$ is connected. Then for every $-1\leq i\leq r-3$, it holds that
  \begin{equation*}
    \lambda^{(i)}(X) \leq \frac{\lambda^{(i+1)}(X)}{1-\lambda^{(i+1)}(X)}.
  \end{equation*}
\end{theorem}

\begin{theorem}[Trickle-down for Grassmannian complexes \cite{kaufman_garlands_2022}]
  \label{thm:tricklegrass}
  Let $X$ be a rank-$r$ $\bF_q$-Grassmannian complex such that for every $x\in X$ with $\rank(x)\leq r-2$, the link $X_x$ is connected. Then for every $-1\leq i\leq r-3$, it holds that
  \begin{equation*}
    \lambda^{(i)}(X) \leq \frac{\lambda^{(i+1)}(X)}{q(1-\lambda^{(i+1)}(X))}.
  \end{equation*}
\end{theorem}

Observe that for simplicial complexes $X$ there is a critical value of $1/2$, in the sense that if $\lambda^{(i)}(X)=1/2-\epsilon$, then Theorem~\ref{thm:tricklesimp} implies that $\lambda^{(i-1)}(X)=1-\Theta(\epsilon)$. Thus in particular if the rank-$0$ local expansion $\lambda^{(0)}(X)<1/2$ is bounded away from $1/2$, then the global expansion $\lambda^{(-1)}(X)$ is bounded away from $1$. This local-to-global phenomenon is a key property of high-dimensional expanders, and for this reason we are particularly interested in simplicial high-dimensional expanders with local expansion $<1/2$. Our construction of Cayley HDXs in this paper indeed has arbitrarily small local expansion.

Whereas for simplicial complexes the trickle-down theorem gives decaying expansion bounds and has critical value $1/2$, for $\bF_q$-Grassmannian complexes the trickle-down theorem can actually give improved expansion in lower ranks due to behavior around a critical value of $(q-1)/q$. That is, if $X$ is a $\bF_q$-Grassmannian complex and $\lambda^{(i+1)}(X)<(q-1)/q$, then Theorem~\ref{thm:tricklegrass} implies that
\begin{equation*}
  \lambda^{(i)}(X) \leq \frac{\lambda^{(i+1)}(X)}{q(1-\lambda^{(i+1)}(X))} < \lambda^{(i+1)}(X).
\end{equation*}

The trickle-down theorem is one of two major local-to-global results for high-dimensional expanders. The other result shows that local expansion at all ranks implies good expansion of the up-down walks. A line of work has successively improved this result for simplicial complexes \cite{kaufman_high_2017,dinur_high_2017,kaufman_high_2018,alev_improved_2020,gotlib_fine_2023}, and Kaufman and Tessler~\cite{kaufman_garlands_2022} give a generalization to posets satisfying some regularity conditions. We do not state these results here, as we will not use them directly in this paper, though our expansion analysis in Section~\ref{sec:expansion} (specifically the proof of Proposition~\ref{prop:G1exp}) uses similar proof techniques.

\subsection{Basisification of Grassmannian complexes}
\label{sec:basisification}
In this section, we present an operation we call basisification, which transforms a rank-$r$ Grassmannian complex into a rank-$r$ simplicial complex. To the best of our knowledge, basisification has not been previously studied in the literature. Notably, this operation preserves local expansion, and preserves sparsity up to constant factors depending only on $q,r$. Thus the basisification of a low-degree Grassmannian HDX is a low-degree simplicial HDX, so low-degree Grassmannian HDXs can be viewed as a strictly harder object to construct than low-degree simplicial HDXs.

\begin{definition}
  Let $X$ be a rank-$r$ $\bF_q$-Grassmannian complex. The \textbf{basisification} $B=\beta(X)$ of $X$ is the rank-$r$ simplicial complex defined so that $B(i)$ contains every basis for each element of $X(i)$. Here a basis of $x_i\in X(i)$ is a set of $i+1$ lines (that is, projective points) in $x_i$ that span $x_i$. The weight function $m_B$ of $B$ is defined by $m_B(b_i)=m_X(\spn b_i)/N_i$, where $N_i$ is the normalization constant given by the number of bases of $\bF_q^{i+1}$.
\end{definition}

The lemma below shows that basisification preserves the 1-skeleton graphs of links up to tensoring by complete graphs. In particular, it immediately yields the corollary that basisification preserves local expansion.

\begin{lemma}
  Let $X$ be a rank-$r$ $\bF_q$-Grassmannian complex with basisification $B=\beta(X)$. For every $-1\leq i\leq r-2$ and every $b_i\in B(i)$, letting $\bfG_{\text{complete}}^{(N)}$ denote the complete graph with self loops on $N$ vertices, we have
  \begin{equation*}
    \bfG_{B_{b_i}} \cong \bfG_{X_{\spn b_i}}\otimes\bfG_{\text{complete}}^{(q^{i+1})}.
  \end{equation*}
\end{lemma}
\begin{proof}
  Fix $b_i\in B(i)$ and let $x_i=\spn b_i$. Then $V(\bfG_{B_{b_i}})$ is by definition the set of $b_0\in B(0)$ such that $\spn\{b_i,b_0\}\in X(i+1)$, or equivalently, such that $\spn\{b_i,b_0\}\in X_{x_i}(0)=V(\bfG_{X_{x_i}})$. Thus we may construct a bijection
  \begin{equation*}
    \phi:V(\bfG_{B_{b_i}})\rightarrow V(\bfG_{X_{\spn b_i}}\otimes\bfG_{\text{complete}}^{(q^{i+1})}) = V(\bfG_{X_{x_i}})\times[q^{i+1}]
  \end{equation*}
  such that for each $x_{i+1}\in X(i+1)$ that contains $x_i$, the map $\phi$ sends the $(q^{i+2}-q^{i+1})/(q-1)=q^{i+1}$ lines $b_0\in B(0)$ with $\spn\{b_i,b_0\}=x_{i+1}$ to the $q^{i+1}$ elements of $\{x_{i+1}\}\times[q^{i+1}]$; any bijection between these $q^{i+1}$-element vertex subsets will suffice. Now observe that $\phi$ induces our desired graph isormorphism, as $\{b_0,b_0'\}\in E(\bfG_{B_{b_i}})$ if and only if $\spn\{b_i,b_0,b_0'\}=\spn\{\spn\{b_i,b_0\},\spn\{b_i,b_0'\}\}\in X(i+2)$. More formally,
  \begin{equation*}
    m_{\bfG_{B_{b_i}}}(\{b_0,b_0'\}) = m_{\bfG_{X_{\spn b_i}}\otimes\bfG_{\text{complete}}^{(q^{i+1})}}(\{\phi(b_0),\phi(b_0')\}),
  \end{equation*}
  as both sides above are by definition proportional to $m_X(\spn\{b_i,b_0,b_0'\})$, and the equality then holds assuming we normalize the edge weights in $\bfG_{\text{complete}}^{(q^{i+1})}$ by an appropriate constant.
\end{proof}

\begin{corollary}
  \label{cor:basexp}
  Let $X$ be a rank-$r$ $\bF_q$-Grassmannian complex with basisification $B=\beta(X)$. For every $-1\leq i\leq r-2$,
  \begin{equation*}
    \lambda^{(i)}(B) = \lambda^{(i)}(X).
  \end{equation*}
\end{corollary}

Corollary~\ref{cor:basexp} implies that a family of constant-degree Grassmannian HDXs immediately yields a family of constant-degree simplicial HDXs by passing to the basisification, so in some sense, Grassmannian HDXs are the strictly harder object to construct. Yet in this paper we nevertheless use polynomial-degree Grassmannian HDXs to construct subpolynomial-degree simplicial HDXs. The key observation, presented in Section~\ref{sec:caycomplex}, is that Grassmannian complexes have an alternative notion of sparsity, namely sparsity within the ambient vector space, which corresponds to degree only upon passing to a Cayley complex.

\subsection{Chain complexes and homology}
\label{sec:topprelim}
In this section, we describe the notions of chain complexes and homology groups. For the purpose of this paper, we restrict attention to chain complexes over $\bF_2$ that are associated to simplicial complexes.

\begin{definition}
  Let $X$ be a rank-$r$ simplicial complex. The \textbf{chain complex $\cC_*(X)$ over $\bF_2$ associated to $X$} is the sequence of $\bF_2$ vector spaces and boundary maps
  \begin{equation*}
    C_r(X) \xrightarrow{\partial_r} C_{r-1}(X) \xrightarrow{\partial_{r-1}} \cdots \xrightarrow{\partial_0} C_{-1}(X),
  \end{equation*}
  where $C_i(X)=\bF_2^{X(i)}$ and $\partial_i:C_i(X)\rightarrow C_{i-1}(X)$ is defined by its action on basis vectors $\1_{x_i}\in C_i$ for $x_i\in X(i)$ by the formula $\partial_i\1_{x_i}=\sum_{x_{i-1}\in X(i-1):x_{i-1}\prec x_i}\1_{x_{i-1}}$. In particular for all $i$ it holds that $\partial_{i-1}\partial_i=0$. We furthermore define the following notation.
  \begin{itemize}
  \item Elements of $C_i(X)$ are called \textbf{$i$-chains}.
  \item Elements of $Z_i(X)=\ker\partial_i$ are called \textbf{$i$-cycles}.
  \item Elements of $B_i(X)=\im\partial_{i+1}$ are called \textbf{$i$-boundaries}.
  \item Because $\partial_{i-1}\partial_i=0$, it holds that $B_i(X)\subseteq Z_i(X)$. Therefore there is a well defined group $H_i(X)=Z_i(X)/B_i(X)$, which is called the \textbf{$i$-homology group}.
  \end{itemize}
\end{definition}

\section{Cayley simplicial complexes}
\label{sec:caycomplex}
In this section, we generalize Cayley graphs to simplicial complexes. As with Cayley graphs, these Cayley complexes have a transitive group action on the vertices, so that in particular the links of all the vertices are isomorphic. Various different types of Cayley simplicial complexes have been constructed in prior works, e.g.~\cite{lubotzky_ramanujan_2005,lubotzky_explicit_2005,vadhan_construction_2018,conlon_hypergraph_2019,conlon_hypergraph_2020}. However, our presentation is slightly different, as in Section~\ref{sec:cayleyF2k} we show how to use Grassmannian HDXs to generate Cayley simplicial complexes.

\begin{definition}
  \label{def:cayleycomplex}
  Let $G$ be a group and $r\geq 0$ be an integer. Let $S\subseteq 2^{G\setminus\{1_G\}}$ be a rank-$r$ simplicial complex on vertex set $G\setminus\{1_G\}$ that satisfies the following \textbf{Cayley symmetry condition}: for every face $s\in S$ and every vertex $g\in s$, then $g^{-1}(s\cup\{1_G\})\setminus \{1_G\}\in S$ is a face of weight $m_S(g^{-1}(s\cup\{1_G\})\setminus \{1_G\})=m_S(s)$. Then define the \textbf{Cayley simplicial complex over $G$ generated by $S$} to be the rank-$(r+1)$ simplicial complex with vertex set $G$ given by
  \begin{equation*}
    \Cay(G,S) = \{g(s\cup\{1_G\}):g\in G,s\in S\}.
  \end{equation*}
  with weight function
  \begin{equation*}
    m_{\Cay(G,S)}(g(s\cup\{1_G\})) = \frac{|s|+1}{|G|}\cdot m_S(s).
  \end{equation*}
\end{definition}


The following lemma verifies that the weight function in Definition~\ref{def:cayleycomplex} is well defined and standard (see Definition~\ref{def:poset}).

\begin{lemma}
  Let $X=\Cay(G,S)$ be a Cayley simplicial complex as in Definition~\ref{def:cayleycomplex}. Then $m_X$ is a well defined weight function, and if $m_S$ is standard then $m_X$ is standard.
\end{lemma}
\begin{proof}
  To see that $m_X$ is a well defined weight function, it suffices to show that for each rank $i$, the restriction $m_{X(i)}=m_X|_{X(i)}$ is a probability distribution. This fact holds as
  \begin{align*}
    \sum_{x_i\in X(i)}m_X(x_i)
    &= \frac{1}{i+1}\sum_{g\in G}\sum_{s_{i-1}\in S(i-1)}m_X(g(s_{i-1}\cup\{1_G\})) \\
    &= \frac{1}{i+1}\sum_{g\in G}\sum_{s_{i-1}\in S(i-1)}\frac{i+1}{|G|}\cdot m_S(s_{i-1}) \\
    &= \sum_{s_{i-1}\in S(i-1)}m_S(s_{i-1}) \\
    &= 1,
  \end{align*}
  where the $1/(i+1)$ factor in the first expression on the right above arises because we count each face $x_i$ a total of $i+1$ times, once for each vertex $g\in x_i$.

  Now assume that $m_S$ is standard. The faces $x_{i+1}\in X(i+1)$ containing a given $x_i=g(s_{i-1}\cup\{1_G\})\in X(i)$ are by definition those faces of the form $x_i=g(s_i\cup\{1_G\})$ for $s_i\in S(i)$ that contain $s_{i-1}$. Thus if we sample $x_{i+1}\sim m_{X(i+1)}$ and then sample $x_i\sim\Unif\{x_i'\in X(i):x_i'\prec x_{i+1}\}$, the probability of obtaining a given $x_i=g(s_{i-1}\cup\{1_G\})\in X(i)$ is
  \begin{align*}
    \Pr[x_i]
    &= \sum_{x_{i+1}\in X(i+1):x_{i+1}\succ x_i}\frac{1}{i+2}\cdot m_X(x_{i+1}) \\
    &= \sum_{s_i\in S(i):s_i\succ s_{i-1}}\frac{1}{i+2}\cdot\frac{i+2}{|G|}\cdot m_S(s_i) \\
    &= \frac{i+1}{|G|}\cdot m_S(s_{i-1}) \\
    &= m_X(x_i),
  \end{align*}
  where the third equality above holds by the assumption that $m_S$ is standard. Thus $m_X$ is standard.
\end{proof}

Definition~\ref{def:cayleycomplex} directly generalizes undirected Cayley graphs to simplicial complexes, as the $r=0$ case recovers undirected Cayley graphs. In the $r=0$ case, the condition that $g\in s\in S\implies g^{-1}(s\cup\{1_G\})\setminus \{1_G\}\in S$ simply requires that the generating set be symmetric. For $r\geq 1$, this condition provides the high-dimensional generalization: for every $s\in S$ and $g\in s$, then the face $g^{-1}(s\cup\{1_G\})$ must belong to $\Cay(G,S)$ and must contain $1_G$, so this face must belong to the generating set $S$.

The following fact is a direct consequence of Definition~\ref{def:cayleycomplex}:

\begin{lemma}
  \label{lem:cayleylinks}
  Let $X=\Cay(G,S)$. Then every vertex $g\in X(0)=G$ has link $X_g\cong S$.
\end{lemma}
\begin{proof}
  The faces containing $g\in X(0)$ are exactly those of the form $g(s\cup\{1_G\})=gs\cup\{g\}$ for $s\in S$. Furthermore, by definition $m_{X_g}(g(s\cup\{1_G\}))=m_S(s)$. Thus $X_g=gS\cong S$.
\end{proof}

\subsection{Cayley simplicial HDXs from Grassmannian HDXs}
\label{sec:cayleyF2k}
In this section, we consider the Cayley complexes of Definition~\ref{def:cayleycomplex} for the special case where $G=\bF_2^k$. In this case, we show below that if $S$ is the basisification of a $\bF_2$-Grassmannian complex, then $S$ must satisfy the symmetry condition in Definition~\ref{def:cayleycomplex}. 

\begin{lemma}
  \label{lem:cayleyF2k}
  Let $S=\beta(X)$ be the basisification of a rank-$r$ $\bF_2$-Grassmannian complex $X$ in ambient vector space $\bF_2^k$. Then $S$ satisfies the Cayley symmetry condition in Definition~\ref{def:cayleycomplex}. Furthermore, for all $0\leq i\leq r+1$,
  \begin{equation*}
    \lambda^{(i)}(\Cay(\bF_2^k,S)) = \lambda^{(i-1)}(X),
  \end{equation*}
  so that in particular if $\spn X(0)=\bF_2^k$, then
  \begin{equation*}
    \lambda(\Cay(\bF_2^k,S)) \geq \frac{\lambda(X)}{1-\lambda(X)}.
  \end{equation*}
\end{lemma}
\begin{proof}
  For every face $s=\{s_0^{(0)},\dots,s_0^{(i)}\}\in S$, and for every $s_0^{(j)}\in s$, then the set
  \begin{equation*}
    s_0^{(j)}+(s\cup\{0\})\setminus\{0\} = \{s_0^{(j)}+s_0^{(j')}:j'\neq j\}\cup\{s_0^{(j)}\}
  \end{equation*}
  has the same span as $s$, so it is also a basis for $\spn s$, and thus belongs to $S=\beta(X)$ with the same weight as $s$. Therefore $S$ satisfies the Cayley symmetry condition.

  The first expansion statement follow directly from Lemma~\ref{lem:cayleylinks}, which shows that the links of rank-$i$ faces in the $\Cay(\bF_2^k,S)$ are links of rank-$(i-1)$ faces in $S$, and from Corollary~\ref{cor:basexp} which shows that $X$ and its basisification $S$ have the same local expansion. The second expansion statement then follows from the trickle-down Theorem~\ref{thm:tricklesimp}, where the condition $\spn X(0)=\bF_2^k$ implies that the 1-skeleton graph of $\Cay(\bF_2^k,S)$ is connected.
\end{proof}

The degree of $\Cay(\bF_2^k,\beta(X))$ is dictated by the sparsity of $X$ within its ambient vector space. Specifically, $\Cay(\bF_2^k,\beta(X))$ is a $2^k$-vertex simplicial complex whose 1-skeleton graph has degree equal to the number $|X(0)|$ of vertices in $X$. Thus assuming $\spn X(0)=\bF_2^k$ so that the Cayley complex is connected, then to construct low-degree rank-$(r+1)$ Cayley simplicial HDXs, it suffices to construct rank-$r$ Grassmannian HDXs that are sparse within their ambient vector space. An optimally sparse Grassmannian HDX $X$ would have $|X(0)|=\Theta(k)$, which would yield a simplicial HDX $\Cay(\bF_2^k,\beta(X))$ of degree logarithmic in the number of vertices. In this paper we only obtain Grassmannian HDXs $X$ with $|X(0)|=2^{\Theta(\sqrt{k})}$, but even this level of sparsity yields simplicial HDXs of degree subpolynomial in the number of vertices.



\section{Grassmannian HDXs from low-rank matrices}
\label{sec:grassHDX}
In this section, we present our main construction of Grassmannian high-dimensional expanders, and analyze their sparsity and expansion properties. Although the Grassmannian complexes we construct are not low-degree, as their 1-skeleton graph is dense, they are sufficiently sparse within their ambient vector space to generate subpolynomial-degree Cayley simplicial HDXs via Lemma~\ref{lem:cayleyF2k}.

\subsection{Construction}
\label{sec:constructmain}
Below we present our main construction of Grassmannian HDXs from low-rank matrices. The Grassmannian posets are parametrized by parameters $r,b,n$, where we typically think of $r,b$ as fixed constants while $n$ grows large. Throughout this section and the subsequent sections, we always let $q=2^b$.

\begin{definition}
  \label{def:grassconstruct}
  Given integers $r\geq 1$, $b\geq 1$, $n\geq 2^{r+1}$, let $X=X^{r,b,n}$ be the rank-$r$ $\bF_2$-Grassmannian poset defined as follows. Let $q=2^b$. The ambient vector space for $X$ is $\bF_q^n\otimes_{\bF_q}\bF_q^n\cong\bF_2^{bn^2}$, where we use the fact that $(\bF_q^n)^{\otimes 2}$ may be viewed as a vector space over $\bF_2\subseteq\bF_q$. Let $G_{\text{Had}}^{(r)}\in\bF_2^{(2^{r+1}-1)\times(r+1)}$ denote the generator matrix for the length-$(2^{r+1}-1)$ Hadamard code. That is, for $1\leq k\leq 2^{r+1}-1$, the $k$th row of $G_{\text{Had}}^{(r)}$ contains the binary representation of the integer $2^{r+1}-k$, with the most significant bit first. Then the rank-$r$ faces $x_r\in X(r)$ are those $(r+1)$-dimensional $\bF_2$-subspaces of $(\bF_q^n)^{\otimes 2}$ that can be expressed in the form 
\begin{align*}
  x_r
  &= \im(EG_{\text{Had}}^{(r)})
\end{align*}
for any matrix $E\in\bF_2^{bn^2\times(2^{r+1}-1)}$ of the form
\begin{align*}
  E
  &= \begin{pmatrix}
    e_1^{(1)}\otimes e_1^{(2)}&\cdots&e_{2^{r+1}-1}^{(1)}\otimes e_{2^{r+1}-1}^{(2)}
  \end{pmatrix},
\end{align*}
where the vectors $e_k^{(\ell)}\in\bF_q^n$ for $i\in[2^{r+1}-1]$, $j\in[2]$ are any vectors such that for each $j\in[2]$, all vectors in the set $\{e_1^{(j)},\dots,a_{2^{r+1}-1}^{(j)}\}$ are linearly independent.
\end{definition}

Recall that a rank-$0$ face of a $\bF_2$-Grassmannian complex is a 1-dimensional subspace $\{0,x_0\}$, which by abuse of notation may be viewed as a vertex given by the vector $x_0$. Then by definition, the complex $X=X^{r,b,n}$ in Definition~\ref{def:grassconstruct} has vertex set $X(0)$ containing those $x_0\in(\bF_q^n)^{\otimes 2}$ that can be expressed in the form
\begin{align*}
  x_0
  &= e_1^{(1)}\otimes e_1^{(2)}+\cdots+e_{2^r}^{(1)}\otimes e_{2^r}^{(2)},
\end{align*}
where the vectors $e_k^{(\ell)}\in\bF_q^n$ for $i\in[2^r]$, $j\in[2]$ are any vectors such that for each for each $j\in[2]$, all vectors in the set $\{e_1^{(j)},\dots,e_{2^r}^{(j)}\}$ are linearly independent.

More generally, let $G_{\text{Had}}^{(r,i)}\in\bF_2^{(2^{r+1}-1)\times(i+1)}$ denote the first $i+1$ columns of $G_{\text{Had}}^{(r)}$. Then we have the following characterization of $X(i)$.

\begin{lemma}
  \label{lem:lowerfaces}
  For $0\leq i\leq r$, the rank-$i$ faces $X(i)$ are all those $(i+1)$-dimensional $\bF_2$-subspaces of $(\bF_q^n)^{\otimes 2}$ that can be expressed in the form 
  \begin{align*}
    x_i
    &= \im(EG_{\text{Had}}^{(r,i)})
  \end{align*}
  for some matrix $E\in\bF_2^{bn^2\times(2^{r+1}-1)}$ of the form
  \begin{align*}
    E
    &= \begin{pmatrix}
      e_1^{(1)}\otimes e_1^{(2)}&\cdots&e_{2^{r+1}-1}^{(1)}\otimes e_{2^{r+1}-1}^{(2)}
    \end{pmatrix},
  \end{align*}
  where the vectors $e_k^{(\ell)}\in\bF_q^n$ for $i\in[2^{r+1}-1]$, $j\in[2]$ are any vectors such that for each $j\in[2]$, all vectors in the set $\{e_1^{(j)},\dots,a_{2^{r+1}-1}^{(j)}\}$ are linearly independent.
\end{lemma}
\begin{proof}
  The result follows from Definition~\ref{def:grassconstruct} because by the symmetry of the Hadamard code, any $(i+1)$-dimensional subspace $C\subseteq\im(G_{\text{Had}}^{(r)})$ of the Hadamard code can be obtained by permutating the coordinates of the subspace $\im(G_{\text{Had}}^{(r,i)})$. That is, there exists a permutation $\pi:[2^{r+1}-1]\rightarrow[2^{r+1}-1]$ with associated permutation matrix $\pi\in\bR^{(2^{r+1}-1)\times(2^{r+1}-1)}$ such that $C=\im(\pi G_{\text{Had}}^{(r,i)})$.

  To see that such a permutation $\pi$ always exists, we may associate $[2^{r+1}-1]$ with the set of nonzero points in $\bF_2^{r+1}$, and then the Hadamard code $\im G_{\text{Had}}^{(r)}\subseteq\bF_2^{\bF_2^{r+1}\setminus\{0\}}$ is the space of linear functionals $f:\bF_2^{r+1}\rightarrow\bF_2$, where a codeword consists of the evaluations of $f$ at all nonzero points. Thus any two $(i+1)$-dimensional subspaces of the Hadamard code are isomorphic, with isomorphism given by a basis change on $\bF_2^{r+1}$, which in particular induces a permutation on $\bF_2^{r+1}\setminus\{0\}\cong[2^{r+1}-1]$.
\end{proof}

To understand Definition~\ref{def:grassconstruct}, we may associate $(\bF_q^n)^{\otimes 2}=\bF_q^{n\times n}$ with the space of $n\times n$ matrices over $\bF_q$. Then $X(0)$ is simply the set of all rank-$2^r$ matrices in $\bF_q^{n\times n}$. More generally, $X(i)$ is an $(i+1)$-dimensional subspace of $\bF_q^{n\times n}$ such that all matrices in this subspace have rank $2^r$. Further intuition for this construction is provided in Section~\ref{sec:grassconstructinf}, where we describe a relationship to the tensor product of two Hadamard codes, which in turn is similar to a degree-2 Reed-Muller code. See Example~\ref{ex:rank1} for an explicit description of $X$ in the special case of rank $r=1$.

\subsection{Sparsity}
\label{sec:sparsity}
In this section, we analyze the sparsity within the ambient vector space for the Grassmannian poset $X=X^{r,b,n}$ defined in Definition~\ref{def:grassconstruct}. Recall below that we let $q=2^b$.

\begin{lemma}
  \label{lem:spanX0}
  Let $X=X^{r,b,n}$ be the Grassmannian poset in Definition~\ref{def:grassconstruct}. Then $\spn_{\bF_2}X(0)=(\bF_q^n)^{\otimes 2}$.
\end{lemma}
\begin{proof}
  Viewing $(\bF_q^n)^{\otimes 2}\cong\bF_q^{n\times n}$ as the space of $n\times n$ matrices over $\bF_q$, it suffices to show that every rank-1 matrix can be expressed as a sum of matrices in $X(0)$, that is of rank-$2^r$ matrices, as every matrix in $\bF_q^{n\times n}$ can then in turn be expressed as a sum of rank-1 matrices. Given any rank-1 matrix $e_1^{(1)}\otimes e_1^{(2)}$, for $2\leq i\leq 2^r$ and $j\in\{1,2\}$, choose vectors $e_k^{(\ell)}\in\bF_q^n$ such that for each $j$, the vectors $e_1^{(j)},\dots,e_{2^r}^{(j)}$ are linearly independent. Also choose some $v\in\bF_q^n$ linearly independent to all $e_k^{(\ell)}$. Such linearly independent vectors exist because $n\geq 2^{r+1}$ by assumption. Then
  \begin{align*}
    e_1^{(1)}\otimes e_1^{(2)}
    &= \left((e_1^{(1)}+v)\otimes e_1^{(2)}+e_2^{(1)}\otimes e_2^{(2)}+\cdots+e_{2^r}^{(1)}\otimes e_{2^r}^{(2)}\right) \\
    &\hspace{.5cm}+\left(v\otimes e_1^{(2)}+e_2^{(1)}\otimes e_2^{(2)}+\cdots+e_{2^r}^{(1)}\otimes e_{2^r}^{(2)}\right),
  \end{align*}
  so $e_1^{(1)}\otimes e_1^{(2)}$ is a sum of rank-$2^r$ matrices, as desired.
\end{proof}

Lemma~\ref{lem:spanX0} shows that the smallest vector space containing all of $X$ is the whole space $(\bF_q^n)^{\otimes 2}$, which has $\bF_2$-dimension $bn^2=\Theta(n^2)$, where we assume for this discussion that $r,b=O(1)$. By definition $X(0)$ contains at most $|X(0)|\leq|\bF_q^n|^{2r}=2^{2rbn}=2^{O(n)}$ points out of this vector space of size $|(\bF_q^n)^{\otimes 2}|=2^{\Theta(n^2)}$. This sparsity of $X(0)$ within its ambient vector space $\spn X(0)=(\bF_q^n)^{\otimes 2}$ is what makes our construction interesting in comparison to a complete complex. In particular, this sparsity determines the degree of the resulting Cayley simplicial complex we will construct using $X$.

\subsection{Expansion}
\label{sec:expansion}
In this section, we analyze the expansion of the Grassmannian poset $X=X^{r,b,n}$ defined in Definition~\ref{def:grassconstruct}. Specifically, we show that the top level links of $X$ have good local expansion, and all the links are connected. The local-to-global theorems of \cite{kaufman_garlands_2022} imply good local expansion at all levels, and good expansion of the up-down walks.

Our main two results of this section are stated below.

\begin{theorem}
  \label{thm:localexp}
  For integers $r\geq 1$, $b\geq 1$, $n\geq 2^{r+1}$, let $q=2^b$, and let $X=X^{r,b,n}$ be the Grassmannian poset in Definition~\ref{def:grassconstruct}. Then
  \begin{align*}
    \lambda^{(r-2)}(X) &\leq \frac{11}{q}.
  \end{align*}
\end{theorem}

\begin{proposition}
  \label{prop:localconn}
  For integers $r\geq 1$, $b\geq 2$, $n\geq 2^{r+1}$, let $X=X^{r,b,n}$ be the Grassmannian poset in Definition~\ref{def:grassconstruct}. Then for every $-1\leq i\leq r-2$, the link $X_{x_i}$ of every $x_i\in X(i)$ has a connected 1-skeleton graph $\bfG_{X_{x_i}}$.
\end{proposition}

The trickle-down theorem of Kaufman and Tessler~\cite{kaufman_garlands_2022} for Grassmannian complexes (Theorem~\ref{thm:tricklegrass}) then implies that $X^{r,b,n}$ has good local expansion at all levels, as stated below.

\begin{corollary}
  \label{cor:applytrickle}
  For integers $r\geq 1$, $b\geq 5$, $n\geq 2^{r+1}$, let $q=2^b$, and let $X=X^{r,b,n}$ be the Grassmannian poset in Definition~\ref{def:grassconstruct}. Then
  \begin{align*}
    \lambda(X) &\leq \frac{11}{q}.
  \end{align*}
\end{corollary}
\begin{proof}
  The result follows directly by Theorem~\ref{thm:tricklegrass} with Theorem~\ref{thm:localexp} and Proposition~\ref{prop:localconn}.
\end{proof}

The remainder of this section is dedicated to proving Theorem~\ref{thm:localexp} and Proposition~\ref{prop:localconn}. We begin in Section~\ref{sec:linkdecomp} below by showing that that the 1-skeleton graph $\bfG_{X_{x_i}}$ of the link $X_{x_i}$ of a face $x_i\in X(i)$ can be decomposed as a tensor product of simpler graphs. Section~\ref{sec:matrixwalkexp} then bounds the expansion of these simpler graphs for top-level links, thereby proving Theorem~\ref{thm:localexp}. Section~\ref{sec:linksconn} shows that these simpler graphs are connected for all ranks $i$, thereby proving Proposition~\ref{prop:localconn}.

\subsubsection{Decomposing the links using the matrix poset}
\label{sec:linkdecomp}
In this section, we show how the 1-skeleton graph of each link in $X=X^{r,b,n}$ can be expressed as a tensor product of simpler graphs. These simpler graphs can be viewed as walks in the matrix poset defined below, which we use throughout our analysis.

\begin{definition}
  \label{def:matrixposet}
  Given a vector space $\bF_q^m$, the \textbf{matrix poset $\cM_q^m$} is the set of all matrices in $\bF_q^{m\times m}$, where rank is given by matrix rank, that is $\cM_q^m(r)$ is the set of rank-$r$ matrices in $\bF_q^{m\times m}$, and $M_1\preceq M_2$ if and only if $\rank(M_2-M_1)=\rank(M_2)-\rank(M_1)$.
\end{definition}

\begin{lemma}
  \label{lem:matrixposetok}
  The set $\cM_q^m$ with binary relation $\prec$ given in Definition~\ref{def:matrixposet} forms a well defined pure graded poset.
\end{lemma}
\begin{proof}
  To see that $\cM_q^m$ is a poset, consider matrices $M_1\preceq M_2$ and $M_2\preceq M_3$. We want to show that $M_1\preceq M_3$. Writing $M_3=M_1+(M_2-M_1)+(M_3-M_2)$, then by assumption
  \begin{equation*}
    \rank(M_1)+\rank(M_2-M_1)+\rank(M_3-M_2) = \rank(M_2)+\rank(M_3-M_2) = \rank(M_3).
  \end{equation*}
  Therefore
  \begin{align*}
    \rank(M_2-M_1)+\rank(M_3-M_2)
    &= \rank(M_3)-\rank(M_1) \\
    &\leq \rank(M_3-M_1) \\
    &\leq \rank(M_2-M_1)+\rank(M_3-M_2),
  \end{align*}
  where the two inequalities above hold by the subadditivity of matrix rank. The above inequalities imply that $\rank(M_3)-\rank(M_1)=\rank(M_3-M_1)$, so $M_1\preceq M_3$, as desired.

  To see that $\cM_q^m$ is graded with rank function given by matrix rank, observe that if $M_1\prec M_2$ with no $M'$ such that $M_1\prec M'\prec M_2$, then we must have $\rank(M_1)=\rank(M_2)-1$, as if $\rank(M_1)\leq\rank(M_2)-2$, then we can choose a rank-1 matrix $L\prec M_2-M_1$, and setting $M'=M_1+L$ gives $M_1\prec M'\prec M_2$, a contradiction.

  The graded poset $\cM_q^m$ is pure because every matrix $M\in\cM_q^m$ is dominated by a full-rank matrix $M'\in\cM_q^m(m)$, as can be seen by changing the basis of the rows and columns so that $M$ is diagonal; then $M'$ can be obtained by filling in the zero diagonal entries of $M$ with nonzero values.
\end{proof}

The goal of this section is to prove the following proposition, which decomposes the 1-skeleton graph $\bfG_{X_{x_i}}$ of any link $X_{x_i}$ into a tensor product of walks on the matrix poset.

\begin{proposition}
  \label{prop:tensornice}
  Let $X=X^{r,b,n}$ be the Grassmannian poset in Definition~\ref{def:grassconstruct}, and let $q=2^b$. For every $-1\leq i\leq r-2$ and every $x_i\in X(i)$, there exists a graph projection (see Definition~\ref{def:graphproj})
  \begin{equation*}
    \Pi:{\bfG_1}^{\otimes 2^{i+1}-1}\otimes{\bfG_2}\rightarrow\bfG_{X_{x_i}},
  \end{equation*}
  for graphs $\bfG_1,\bfG_2$ defined as follows:
  \begin{itemize}
  \item $\bfG_1$ is the uniformly weighted graph with
    \begin{align*}
      V(\bfG_1) = \{M\in\cM_q^{2^{r-i}}(2^{r-i-1}):\;&M\prec I_{2^{r-i}}\} \\
      E(\bfG_1) = \Big\{\{L_1+L_2,L_1+L_3\}:\;&L_1,L_2,L_3\in\cM_q^{2^{r-i}}(2^{r-i-2}),\\
                                                     &L_1+L_2+L_3\in\cM_q^{2^{r-i}}(3\cdot 2^{r-i-2}),\\
                                                     &L_1+L_2+L_3\prec I_{2^{r-i}}\Big\}.
    \end{align*}
  \item Let $R=C=\bF_q^{2^{r+1}-2^{r-i}}\times\{0\}^{n-(2^{r+1}-2^{r-i})}\subseteq\bF_q^n$ (here any $(2^{r+1}-2^{r-i})$-dimensional subspaces $R,C\subseteq\bF_q^n$ would suffice). Then $\bfG_2$ is the uniformly weighted graph with
    \begin{align*}
      V(\bfG_2) = \Big\{M\in\cM_q^n(2^{r-i-1}):\;&\text{rowspan}(M)\cap R=\{0\},\\
                                                 &\text{colspan}(M)\cap C=\{0\}\Big\} \\
      E(\bfG_2) = \Big\{\{L_1+L_2,L_1+L_3\}:\;&L_1,L_2,L_3\in\cM_q^n(2^{r-i-2}),\\
                                                 &L_1+L_2+L_3\in\cM_q^{2^{r-i}}(3\cdot 2^{r-i-2}),\\
                                                 &\text{rowspan}(L_1+L_2+L_3)\cap R=\{0\},\\
                                                 &\text{colspan}(L_1+L_2+L_3)\cap C=\{0\}\Big\}.
    \end{align*}
  \end{itemize}
\end{proposition}

The following corollary shows that Proposition~\ref{prop:tensornice} reduces the problem of bounding the local expansion (or connectedness) of $X$ to the problem of bounding the local expansion (or connectedness) of the graphs $\bfG_1,\bfG_2$ above. We later bound the expansion (and connectedness) of these graphs, which can be viewed as certain walks on the matrix poset, in Section~\ref{sec:matrixwalkexp} below.

\begin{corollary}
  \label{cor:linkexpreduce}
  Let $X=X^{r,b,n}$ be the Grassmannian poset in Definition~\ref{def:grassconstruct}. For every $-1\leq i\leq r-2$ and every $x_i\in X(i)$, we have
  \begin{equation*}
    \lambda(\bfG_{X_{x_i}}) \leq \max\{\lambda(\bfG_1),\lambda(\bfG_2)\}.
  \end{equation*}
\end{corollary}
\begin{proof}
  The corollary follows directly by Proposition~\ref{prop:tensornice} and Lemma~\ref{lem:projexp}.
\end{proof}

We now turn to proving Proposition~\ref{prop:tensornice}. To analyze the structure of links in $X$, we begin by characterizing the structure of faces in $X$ with Lemma~\ref{lem:commonprec} below and its subsequent generalization to higher ranks in Lemma~\ref{lem:commonprechigh}.

\begin{lemma}
  \label{lem:commonprec}
  Fix a face $\spn\{M_0,M_1\}\in X^{r,b,n}(1)$, so that $M_0,M_1\in\cM_q^n(2^r)$. Then there exists a unique matrix $M'\in\cM_q^n(2^{r-1})$ such that for every $M''\in\cM_q^n$ with $M''\prec M_0,M_1$, it holds that $M''\preceq M'$. That is, there is a unique maximal element $M'$ among the matrices dominated by $M_0,M_1$, and $\rank(M')=2^{r-1}$.
\end{lemma}
\begin{proof}
  By the definition of $X^{r,b,n}$ in Section~\ref{sec:constructmain}, there exist vectors $e_k^{(\ell)}\in\bF_q^n$ for $k\in[2^{r+1}-1]$, $\ell\in[2]$ such that
  \begin{equation*}
    \begin{pmatrix}M_0&M_1\end{pmatrix} = \begin{pmatrix}e_1^{(1)}\otimes e_1^{(2)}&\cdots&e_{2^{r+1}-1}^{(1)}\otimes e_{2^{r+1}-1}^{(2)}\end{pmatrix}\cdot G_{\text{Had}}^{(r,1)},
  \end{equation*}
  where we view the left hand side above as a matrix in $\bF_q^{n^2\times 2}$. Thus the matrix $M'=\sum_{k=1}^{2^{r-1}}e_k^{(1)}\otimes e_k^{(2)}$ has rank $2^{r-1}$ and is dominated by $M_0,M_1$ in the matrix poset, so it remains to be shown that $M'$ dominates all matrices that are dominated by $M_0,M_1$.
  
  Let
  \begin{align*}
    R &= \text{rowspan}(M_0)\cap\text{rowspan}(M_1)=\spn\{e_1^{(1)},\dots,e_{2^{r-1}}^{(1)}\} \\
    C &= \text{colspan}(M_0)\cap\text{colspan}(M_1)=\spn\{e_1^{(2)},\dots,e_{2^{r-1}}^{(2)}\}
  \end{align*}
  By definition, any $M''\prec M_0,M_1$ must have its row and column spans contained inside $R$ and $C$ respectively. Therefore it in particular suffices to show that every $M''\preceq M_0$ such that $\text{rowspan}(M'')\subseteq R$ and $\text{colspan}(M'')\subseteq C$ satisfies $M''\preceq M'$. Changing the basis of the rows and columns to $\{e_k^{(1)}\}$ and $\{e_k^{(2)}\}$ respectively, it then suffices to show that every $M''\preceq I_{2^r}$ that is supported inside the top left $2^{r-1}\times 2^{r-1}$ quadrant, so that the row and column spans are supported inside the span of the first $2^{r-1}$ basis vectors, must satisfy $M''\preceq M'$ where
  \begin{equation*}
    M' = \begin{pmatrix}
      I_r&0_{r\times r}\\
      0_{r\times r}&0_{r\times r}
    \end{pmatrix}.
  \end{equation*}
  By definition $I_{2^r}-M''$ is block diagonal with top left quadrant the same as $M'-M''$ and bottom right quadrant equal to $I_r$, so
  \begin{equation*}
    \rank(I_{2^r}-M'') = \rank(M'-M'')+2^{r-1}.
  \end{equation*}
  Because $M''\preceq I_{2^r}$ by assumption, the left hand side above equals $2^r-\rank(M'')$, so after rearranging terms we get that
  \begin{equation*}
    \rank(M'-M'') = 2^{r-1}-\rank(M'') = \rank(M')-\rank(M''),
  \end{equation*}
  so indeed $M''\preceq M'$ as desired.
\end{proof}

\begin{lemma}
  \label{lem:commonprechigh}
  For $0\leq i\leq r$, fix a face $\spn\{M_0,\dots,M_i\}\in X^{r,b,n}(i)$, so that $M_0,\dots,M_i\in\cM_q^n(2^r)$. Then there exists a unique matrix $M'\in\cM_q^n(2^{r-i})$ such that for every $M''\in\cM_q^n$ with $M''\preceq M_0,\dots,M_i$, it holds that $M''\preceq M'$. That is, there is a unique maximal element $M'$ among the matrices dominated by $M_0,\dots,M_i$, and $\rank(M')=2^{r-i}$.
\end{lemma}
\begin{proof}
  The $i=0$ case is immediate, and the $i=1$ case is shown in Lemma~\ref{lem:commonprec}, so assume $i\geq 2$. By the definition of $X^{r,b,n}$ in Section~\ref{sec:constructmain}, there exist vectors $e_k^{(\ell)}\in\bF_q^n$ for $k\in[2^{r+1}-1]$, $\ell\in[2]$ such that
  \begin{equation*}
    \begin{pmatrix}M_0&\cdots&M_i\end{pmatrix} = \begin{pmatrix}e_1^{(1)}\otimes e_1^{(2)}&\cdots&e_{2^{r+1}-1}^{(1)}\otimes e_{2^{r+1}-1}^{(2)}\end{pmatrix}\cdot G_{\text{Had}}^{(r,i)},
  \end{equation*}
  where we view the left hand side above as a matrix in $\bF_q^{n^2\times i}$. Thus the matrix $M'=\sum_{k=1}^{2^{r-i}}e_k^{(1)}\otimes e_k^{(2)}$ has rank $2^{r-i}$ and is dominated by $M_0,\dots,M_i$ in the matrix poset, so it remains to be shown that $M'$ dominates all matrices $M''$ that are dominated by $M_0,\dots,M_i$. For this purpose we proceed by induction.

  Assume the desired result holds for $i-1$. Then by the inductive hypothesis, $M'_0=\sum_{k=1}^{2^{r-i+1}}e_k^{(1)}\otimes e_k^{(2)}$ is the maximal element among the matrices dominated by $M_0,\dots,M_{i-1}$ in the matrix poset. Similarly, $M'_1=\sum_{k=1}^{2^{r-i}}e_k^{(1)}\otimes e_k^{(2)}+\sum_{k=2^{r-i+1}+1}^{2^{r-i+1}+2^{r-i}}e_k^{(1)}\otimes e_k^{(2)}$ is the unique maximal element among the matrices dominated by $M_0,\dots,M_{i-2},M_i$ in the matrix poset. By Lemma~\ref{lem:commonprec}, $M'=\sum_{k=1}^{2^{r-i}}e_k^{(1)}\otimes e_k^{(2)}$ is the unique maximal element among the matrices dominated by $M'_0,M'_1$ in the matrix poset. Therefore we have shown that any $M''\preceq M_0,\dots,M_i$ must satisfy $M''\preceq M'_0,M'_1$, and therefore must satisfy $M''\preceq M'$, as desired.
\end{proof}

The following lemma shows that we can characterize a face $x_i\in X(i)$ by a unique set of rank-$2^{r-i}$ matrices, which we call the \textbf{minimal matrices} of $x_i$.

\begin{lemma}
  \label{lem:minmatrices}
  Let $X=X^{r,b,n}$ be the Grassmannian poset in Definition~\ref{def:grassconstruct}.
  \begin{enumerate}
  \item\label{it:mm1} For every $0\leq i\leq r-1$ and every $(i+1)$-dimensional subspace $x_i\subseteq(\bF_q^n)^{\otimes 2}$, then $x_i\in X(i)$ if and only if there exist matrices $M_1',\dots,M_{2^{i+1}-1}'\in\cM_q^n(2^{r-i})$ with linearly independent row spans and linearly independent column spans such that
    \begin{equation}
      \label{eq:hadcompress}
      x_i = \im\left(\begin{pmatrix}M_1'&\cdots&M_{2^{i+1}-1}'\end{pmatrix}\cdot G_{\text{Had}}^{(i)}\right),
    \end{equation}
    where above we view $\begin{pmatrix}M_1'&\cdots&M_{2^{i+1}-1}\end{pmatrix}$ as a matrix in $\bF_q^{n^2\times 2^{i+1}-1}$.
  \item\label{it:mm2} Letting $e_k^{(\ell)}\in\bF_q^n$ for $k\in[2^{r+1}-1]$, $\ell\in[2]$ be vectors such that
    \begin{equation}
      \label{eq:facedecomplink}
      x_i = \im\left(\begin{pmatrix}e_1^{(1)}\otimes e_1^{(2)}&\cdots&e_{2^{r+1}-1}^{(1)}\otimes e_{2^{r+1}-1}^{(2)}\end{pmatrix}\cdot G_{\text{Had}}^{(r,i)}\right)
    \end{equation}
    as given by the definition of $X^{r,b,n}$ in Section~\ref{sec:constructmain}, then for $1\leq j\leq 2^{i+1}-1$ the matrices
    \begin{align}
      \label{eq:rank1expbiglink}
      M_j' &= \sum_{k=2^{r-i}(j-1)+1}^{2^{r-i}j}e_k^{(1)}\otimes e_k^{(2)}.
    \end{align}
    satisfy~(\ref{eq:hadcompress}).
  \item\label{it:mm3} For $x_i\in X(i)$, the set $\{M_1',\dots,M_{2^{i+1}-1}'\}\subseteq\cM_q^n(2^{r-i})$ of matrices satisfying~(\ref{eq:hadcompress}) is unique. In particular, (\ref{eq:rank1expbiglink}) gives the same matrices $M_j'$ (up to permuting the labels $j\in[2^{i+1}-1]$) for any choice of the vectors $e_k^{(\ell)}$ that satisfies~(\ref{eq:facedecomplink}).
  \end{enumerate}
\end{lemma}

\begin{definition}
  For $x_i\in X^{r,b,n}(i)$, we let the \textbf{minimal matrix set} of $x_i$ be the unique set $\{M_1',\dots,M_{2^{i+1}-1}\}\subseteq\cM_q^n(2^{r-i})$ satisfying~(\ref{eq:hadcompress}) as given by Lemma~\ref{lem:minmatrices}.
\end{definition}

\begin{proof}[Proof of Lemma~\ref{lem:minmatrices}]
  Item~\ref{it:mm2} is immediate from the definiton of $X=X^{r,b,n}$. That is, if $x_i\in X(i)$, then the existence of the minimal matrices satisfying~(\ref{eq:hadcompress}) is immediate. For the opposite direction of item~\ref{it:mm1}, if there exist matrices $M_j'$ satisfying~(\ref{eq:hadcompress}), then for each $M_j'$ we may simply choose vectors $e_k^{(\ell)}\in\bF_q^n$ that give a decomposition of $M_j'$ into rank-$1$ matrices $e_k^{(1)}\otimes e_k^{(2)}$, so that~(\ref{eq:rank1expbiglink}) is satisfied. Then $x_i$ satisfies the definition for a face in $X(i)$ as given in section~\ref{sec:constructmain}.

  It remains to prove item~\ref{it:mm3}. Fix an arbitrary basis $M_0,\dots,M_i$ of the $(i+1)$-dimensional space $x_i$. By definition, there exist vectors $e_k^{(\ell)}\in\bF_q^n$ for $k\in[2^{r+1}-1]$, $\ell\in[2]$ such that
  \begin{align}
    \label{eq:facedecomplinkhere}
    \begin{split}
      \begin{pmatrix}M_0&\cdots&M_i\end{pmatrix}
      &= \begin{pmatrix}e_1^{(1)}\otimes e_1^{(2)}&\cdots&e_{2^{r+1}-1}^{(1)}\otimes e_{2^{r+1}-1}^{(2)}\end{pmatrix}\cdot G_{\text{Had}}^{(r,i)} \\
      &= \begin{pmatrix}M_1'&\cdots&M_{2^{i+1}-1}'\end{pmatrix}\cdot G_{\text{Had}}^{(i)},
    \end{split}
  \end{align}
  where we view the left hand side above as a matrix in $\bF_q^{n^2\times i}$. We will first show that the $M_j'$ as defined in~(\ref{eq:rank1expbiglink}) only depend on the choice of basis $M_0,\dots,M_i$, and not on the choice of vectors $e_k^{(\ell)}$ satisfying~(\ref{eq:facedecomplinkhere}). As changing the basis for $x_i$ corresponds to multiplying the LHS and RHS of~(\ref{eq:facedecomplinkhere}) by a matrix in $\text{GL}(i+1;\bF_q)$, which on the RHS of~(\ref{eq:facedecomplinkhere}) is equivalent to permuting the $M_j'$s, the uniqueness of the set $\{M_1',\dots,M_{2^{i+1}-1}'\}$ then follows.
  
  To see that the matrices $M_j'$ only depend on the choice of basis $M_0,\dots,M_i$ for $x_i$, and not on the choice of vectors $e_k^{(\ell)}$, consider that for every $j$, by definition there exists some $0\leq i'\leq i$ such that $M_j'\preceq M_{i'}$, and then for every $0\leq i''\leq i$ it either holds that $M_j'\preceq M_{i''}$ or $M_j'\preceq M_{i''}+M_{i'}$. Thus let
  \begin{equation*}
    M_{i''}^{(j)} = \begin{cases}
      M_{i''},&M_j'\preceq M_{i''}\\
      M_{i''}+M_{i'},&M_j'\preceq M_{i''}+M_{i'}.
    \end{cases}
  \end{equation*}
  Observe that by construction $M_{i''}^{(j)}$ is determined as a function of $j$ and $\{M_0,\dots,M_i\}$, so that in particular $M_{i''}^{(j)}$ does not depend on the choice of vectors $e_k^{(\ell)}$ satisfying~(\ref{eq:facedecomplinkhere}). Then by definition $x_i=\spn\{M_0^{(j)},\dots,M_i^{(j)}\}$, so by Lemma~\ref{lem:commonprechigh}, $M_j'$ is the unique maximal element among the matrices dominated by $M_0^{(j)},\dots,M_i^{(j)}$. It in particular follows that $M_j'$ is well defined independent of the choice of vectors $e_k^{(\ell)}$.
\end{proof}

Below, we apply the above lemmas to characterize elements of the link $X_{x_i}$ of a rank-$i$ face $x_i\in X(i)$. Specifically, we show that matrices $M_{i+1}$ such that $\spn\{x_i,M_{i+1}\}\in X_{x_i}(0)$ (or equivalently, such that $\spn\{x_i,M_{i+1}\}\in X(i+1)$) can always be expressed as a sum of linearly independent terms of a certain form.

\begin{lemma}
  \label{lem:link0}
  Let $X=X^{r,b,n}$ be the Grassmannian poset in Definition~\ref{def:grassconstruct}. For every $-1\leq i\leq r-1$ and every $x_i\in X(i)$ with minimal matrices $M_1',\dots,M_{2^{i+1}-1}'\in\cM_q^n(2^{r-i})$, then a matrix $M_{i+1}\in\cM_q^n(2^r)$ satisfies $\spn\{x_i,M_{i+1}\}\in X(i+1)$ if and only if $M_{i+1}$ can be expressed in the form
  \begin{equation}
    \label{eq:link0decomp}
    M_{i+1}=M_1''+\cdots+M_{2^{i+1}}''
  \end{equation}
  for some matrices $M_1'',\dots,M_{2^{i+1}}''\in\cM_q^n(2^{r-i-1})$ such that $M_j'\prec M_j$ for all $1\leq j\leq 2^{i+1}-1$, and such that $M_{2^{i+1}}''$ has row and column spans disjoint from those of the matrix $\sum_{j=1}^{2^{i+1}-1}M_j'$.
\end{lemma}
\begin{proof}

  The case $i=-1$ is immediate, so assume that $i\geq 0$. Let $M_0,\dots,M_i\in\cM_q^n(2^r)$ be the basis for $x_i$ given by
  \begin{equation*}
    \begin{pmatrix}M_0&\cdots&M_i\end{pmatrix} = \begin{pmatrix}M_1'&\cdots&M_{2^{i+1}-1}'\end{pmatrix}\cdot G_{\text{Had}}^{(i)}.
  \end{equation*}
  
  If $x_{i+1}=\spn\{x_i,M_{i+1}\}\in X(i+1)$, then by Lemma~\ref{lem:minmatrices} there exist minimal matrices $K_1,\dots,K_{2^{i+2}-1}$ for the face $x_{i+1}$ such that $M_j'=K_{2j-1}+K_{2j}$ for all $j$, and such that
  \begin{equation}
    \label{eq:facedecomplink2}
    \begin{pmatrix}M_0&\cdots&M_{i+1}\end{pmatrix} = \begin{pmatrix}K_1&\cdots&K_{2^{i+2}-1}\end{pmatrix}\cdot G_{\text{Had}}^{(i+1)}.
  \end{equation}
  Then setting $M_j''=K_{2j-1}$ gives the desired decomposition of $M_{i+1}=\sum_{j=1}^{2^{i+1}}M_j''$.

  For the converse, assume that $M_{i+1}\in\cM_q^n(2^r)$ has a decomposition $M_{i+1}=\sum_{j=1}^{2^{i+1}}M_j''$ for matrices $M_j''$ as described in the lemma statement. Then for each $1\leq j\leq 2^{i+1}$ define $K_{2j-1}=M_j''$, and for each $1\leq j\leq 2^{i+1}-1$ define $K_{2j}=M_j'+M_j''$. It follows that~(\ref{eq:facedecomplink2}) holds for these matrices $K_1,\dots,K_{2^{i+2}-1}$, and the row spans as well as the column spans of $K_1,\dots,K_{2^{i+2}-1}$ are by construction linearly independent. Thus $x_{i+1}=\spn\{x_i,M_{i+1}\}\in X(i+1)$ by Lemma~\ref{lem:minmatrices}.
\end{proof}

We are now ready to express the 1-skeleton graph of the link $X_{x_i}$ for $x_i\in X(i)$ as a tensor product of simpler graphs.

\begin{lemma}
  \label{lem:tensorraw}
  Let $X=X^{r,b,n}$ be the Grassmannian poset in Definition~\ref{def:grassconstruct}. Fix any $-1\leq i\leq r-2$ and any $x_i\in X(i)$ with minimal matrices $M_1',\dots,M_{2^{i+1}-1}'\in\cM_q^n(2^{r-i})$. For any $K^{(1)},K^{(2)}\in\cM_q^n(2^r)$ such that $\spn\{x_i,K^{(1)}\},\spn\{x_i,K^{(2)}\}\in X(i+1)$, let $K^{(m)}=K_1^{(m)}+\cdots+K_{2^{i+1}}^{(m)}$ be the decomposition given by~(\ref{eq:link0decomp}) in Lemma~\ref{lem:link0} for $m\in[2]$. Then $\spn\{x_i,K^{(1)},K^{(2)}\}\in X(i+2)$ if and only if the following conditions hold:
  \begin{enumerate}
  \item\label{it:tr1} For all $j\in[2^{i+1}]$, we have $\spn\{K_j^{(1)},K_j^{(2)}\}\in X^{r-i-1,b,n}(1)$.
  \item\label{it:tr2} For all $j\in[2^{i+1}-1]$, the minimal matrices $L_1,L_2,L_3\in\cM_q^n(2^{r-i-2})$ of $\spn\{K_j^{(1)},K_j^{(2)}\}\in X^{r-i-1,b,n}(1)$ (so that $K_j^{(1)}=L_1+L_2$, $K_j^{(2)}=L_1+L_3$) satisfy
    \begin{equation*}
      L_1+L_2+L_3\preceq M_j'.
    \end{equation*}
  \item\label{it:tr3} For $j=2^{i+1}$, the minimal matrices $L_1,L_2,L_3\in\cM_q^n(2^{r-i-2})$ of $\spn\{K_{2^{i+1}}^{(1)},K_{2^{i+1}}^{(2)}\}\in X^{r-i-1,b,n}(1)$ satisfy
    \begin{align*}
      \text{rowspan}(L_1+L_2+L_3) \cap \text{rowspan}\left(\sum_{j=1}^{2^{i+1}-1}M_j'\right) &= \{0\} \\
      \text{colspan}(L_1+L_2+L_3) \cap \text{colspan}\left(\sum_{j=1}^{2^{i+1}-1}M_j'\right) &= \{0\}.
    \end{align*}
  \end{enumerate}
\end{lemma}
\begin{proof}
  If $x_{i+2}=\spn\{x_i,K^{(1)},K^{(2)}\}\in X(i+2)$, then by definition there exist vectors $e_k^{(\ell)}\in\bF_q^n$ for $k\in[2^{r+1}-1]$, $\ell\in[2]$ such that
  \begin{equation*}
    \begin{pmatrix}M_0&\cdots&M_i&K^{(1)}&K^{(2)}\end{pmatrix} = \begin{pmatrix}e_1^{(1)}\otimes e_1^{(2)}&\cdots&e_{2^{r+1}-1}^{(1)}\otimes e_{2^{r+1}-1}^{(2)}\end{pmatrix}\cdot G_{\text{Had}}^{(r,i+2)},
  \end{equation*}
  where $M_0,\dots,M_i$ is a basis for $x_i$, and if for $j\in[2^{i+3}-1]$ we define
  \begin{equation*}
    M_j'' = \sum_{k=2^{r-(i+2)}(j-1)+1}^{2^{r-(i+2)}j}e_k^{(1)}\otimes e_k^{(2)} \\
  \end{equation*}
  to be the minimal matrices for $x_{i+2}$, then it holds for all $j\in[2^{i+1}-1]$ that
  \begin{align*}
     M_j' &= M_{4j-3}''+M_{4j-2}''+M_{4j-1}''+M_{4j}'' \\
     K_j^{(1)} &= M_{4j-3}''+M_{4j-2}'' \\
     K_j^{(2)} &= M_{4j-3}''+M_{4j-1}''.
  \end{align*}
  The latter two equations above also hold for $j=2^{i+1}$. Then for all $j\in[2^{i+1}]$ it by definition holds that $\spn\{K_j^{(1)},K_j^{(2)}\}$ is a face in $X^{r-i-1,b,n}(1)$ with minimal matrices $M_{4j-3}'',M_{4j-2}'',M_{4j-1}''$, so by the first equation above it holds for all $j\in[2^{i+1}-1]$ that $M_{4j-3}''+M_{4j-2}''+M_{4j-1}''\prec M_j'$. Thus items~\ref{it:tr1} and~\ref{it:tr2} hold. Item~\ref{it:tr3} also holds because the $M_j''$ by definition all have linearly independent row spans and linearly independent column spans.

  For the converse, assume that items~\ref{it:tr1},\ref{it:tr2},\ref{it:tr3} hold. We then define matrices $M_1'',\dots,M_{2^{i+3}-1}''$ such that for each $j\in[2^{i+1}-1]$, if $L_1,L_2,L_3$ are the minimal matrices of $\spn\{K_j^{(1)},K_j^{(2)}$ as described in item~\ref{it:tr2}, then we let
  \begin{align*}
    M_{4j-3}'' &= L_1 \\
    M_{4j-2}'' &= L_2 \\
    M_{4j-1}'' &= L_3 \\
    M_{4j}'' &= M_j'+L_1+L_2+L_3.
  \end{align*}
  For $j=2^{i+1}$, we define $M_{2^{i+3}-3},M_{2^{i+3}-2},M_{2^{i+3}-1}$ according to the first 3 equations above. Then by construction
  \begin{equation*}
    \begin{pmatrix}M_0&\cdots&M_i&K^{(1)}&K^{(2)}\end{pmatrix} = \begin{pmatrix}M_1''&\cdots&M_{2^{i+3}-1}''\end{pmatrix}\cdot G_{\text{Had}}^{(i+2)},
  \end{equation*}
  where $M_0,\dots,M_i$ is a basis for $x_i$. Thus by Lemma~\ref{lem:minmatrices}, $x_{i+2}=\spn\{x_i,K^{(1)},K^{(2)}\}$ is a face in $X(i+2)$, as desired.
\end{proof}

For each $j\in[2^{i+1}]$, the conditions imposed in Lemma~\ref{lem:tensorraw} on $K_j^{(1)},K_j^{(2)}$ do not depend on the values of $K_{j'}^{(1)},K_{j'}^{(2)}$ for any $j'\neq j$. This property implies that the 1-skeleton graph $\bfG_{X_{x_i}}$ of the link $X_{x_i}$ can be decomposed as a tensor product of $2^{i+1}$ smaller graphs, as formalized in Proposition~\ref{prop:tensornice}.

\begin{proof}[Proof of Proposition~\ref{prop:tensornice}]
  The result is a direct consequence of Lemma~\ref{lem:tensorraw}. By a change of basis, the graph ${\bfG_1}^{\otimes 2^{i+1}-1}\otimes{\bfG_2}$ is isomorphic to the graph $\bfG'$ with:
  \begin{itemize}
  \item Vertex set $V(\bfG')$ consisting of all tuples of matrices $(M_1'',\dots,M_{2^{i+1}}'')$ as given by Lemma~\ref{lem:link0} that decompose some $M_{i+1}=M_1''+\cdots+M_{2^{i+1}}''$ with $\spn\{x_i,M_{i+1}\}\in X(i+1)$. That is, letting $M_1',\dots,M_{2^{i+1}-1}'$ be the minimal matrices of $x_i$, then $(M_1'',\dots,M_{2^{i+1}}'')\in V(\bfG')$ for all matrices $M_1'',\dots,M_{2^{i+1}}''\in\cM_q^n(2^{r-i-1})$ such that $M_j'\prec M_j$ for all $1\leq j\leq 2^{i+1}-1$, and such that $M_{2^{i+1}}''$ has row and column spans disjoint from those of the matrix $\sum_{j=1}^{2^{i+1}-1}M_j'$.
  \item Edge set $E(\bfG')$ consisting of all pairs of tuples $\{(K_1^{(1)},\dots,K_{2^{i+1}}^{(1)}),(K_1^{(2)},\dots,K_{2^{i+1}}^{(2)})\}$ such that the three conditions in Lemma~\ref{lem:tensorraw} are satisfied.
  \end{itemize}
  The projection $\Pi:\bfG'\rightarrow\bfG_{X_{x_i}}$ is then given by mapping each vertex $(M_1'',\dots,M_{2^{i+1}}'')\in V(\bfG')$ to
  \begin{equation*}
    \Pi(M_1'',\dots,M_{2^{i+1}}'')=\spn\{x_i,M_1''+\cdots+M_{2^{i+1}}''\}.
  \end{equation*}
\end{proof}

\subsubsection{Expansion in the top-level links}
\label{sec:matrixwalkexp}
The goal of this section is to prove Theorem~\ref{thm:localexp}. Specifically, by Corollary~\ref{cor:linkexpreduce}, Theorem~\ref{thm:localexp} follows from the following bounds on the expansion of the graphs $\bfG_1,\bfG_2$ for the $i=r-2$ case of Proposition~\ref{prop:tensornice}.

\begin{proposition}
  \label{prop:G1exp}
  For integers $r\geq 1,b\geq 1,n\geq 2^{r+1}$, let $q=2^b$, and let $\bfG$ be the graph $\bfG_1$ for $i=r-2$ as defined in Proposition~\ref{prop:tensornice}. Then
  \begin{equation*}
    \lambda(\bfG) \leq \frac{11}{q}.
  \end{equation*}
\end{proposition}

\begin{proposition}
  \label{prop:G2exp}
  For integers $r\geq 1,b\geq 1,n\geq 2^{r+1}$, let $q=2^b$, and let $\bfG$ be the graph $\bfG_2$ for $i=r-2$ as defined in Proposition~\ref{prop:tensornice}. Then
  \begin{equation*}
    \lambda(\bfG)\leq \frac{11}{q}.
  \end{equation*}
\end{proposition}

\begin{proof}[Proof of Theorem~\ref{thm:localexp}]
  The result follows directly from Corollary~\ref{cor:linkexpreduce}, Proposition~\ref{prop:G1exp}, and Proposition~\ref{prop:G2exp}.
\end{proof}

The remainder of this section is dedicated to proving Proposition~\ref{prop:G1exp}, and Proposition~\ref{prop:G2exp}. Our analysis consists of two main techniques: coupling-style arguments and localization. The first technique, which is sufficient to prove Proposition~\ref{prop:G2exp}, involves showing that a random walk of interest can be highly correlated with a random walk that is known to mix rapidly.

To prove Proposition~\ref{prop:G1exp}, we combine such coupling-style arguments with a localization argument, in which we decompose the Laplacian of $\bfG_2$ as a linear combination of Laplacians of local walks. After showing these local walks mix rapidly with coupling-style arguments, we may approximate them with walks on a complete graph. Combining these complete graph local walks back together yields a simpler global walk, which we can analyze with another coupling-style argument.

The use of localization to show rapid mixing of walks in posets has received much attention in recent years. A line of work \cite{kaufman_high_2017,dinur_high_2017,kaufman_high_2018,alev_improved_2020,gotlib_fine_2023} developed this machinery for simplicial high-dimensional expanders, which has also led to numerous applications in the sampling literature, beginning with \cite{anari_log-concave_2019}. Kaufman and Tessler \cite{kaufman_garlands_2022} extended the machinery to more general posets, including the Grassmannian poset. Our localization argument uses similar techniques to these prior works, though (to the best of our knowledge) is not encapsulated by any of the previously known results.

We use the following lemma to bound spectral expansion for our coupling-style arguments. We note that all walks we consider in this section are induced by undirected regular graphs, which have symmetric random walk matrices.

\begin{lemma}
  \label{lem:expclose}
  For vertex sets $V\subseteq V'$, let $\bfW\in\bR^{V\times V}$ and $\bfW'\in\bR^{V'\times V'}$ be symmetric random walk matrices and let $\epsilon>0$ be such that for every $v\in V$, it holds that
  \begin{equation*}
    \|\1_v^\top\bfW-\1_v^\top\bfW'\|_1 = 2\disTV(\1_v^\top\bfW,\1_v^\top\bfW') \leq \epsilon.
  \end{equation*}
  Then
  \begin{equation*}
    \lambda(\bfW)\leq\lambda(\bfW')+\epsilon.
  \end{equation*}
\end{lemma}
\begin{proof}
  Recall that symmetric random walk matrices have uniform stationary distribtuion $\Unif$. Let $E=W-W'|_{V\times V}\in\bR^{V\times V}$, so that $E$ is a symmetric matrix with every row and column of $\ell_1$-norm at most $\epsilon$. Therefore $\|E\|\leq\epsilon$, so for every $f\in\Unif_V^\perp\in\bR^V$, then the vector in $f'\in\bR^{V'}$ obtained by padding $f$ with $0$s lies in $\Unif_{V'}^\perp$, and thus it follows that
  \begin{equation*}
    f^\top Wf \leq {f'}^\top W'f'+\epsilon f^\top f \leq (\lambda(W')+\epsilon)f^\top f,
  \end{equation*}
  from which the desired result follows.
\end{proof}

We also need the following lemma, which provides a characterization of matrices $M\preceq I$ in the matrix poset.

\begin{lemma}
  \label{lem:domI}
  If $M\in\cM_q^m(r)$ can be expressed as $M=V_1V_2^\top$ for $V_1,V_2\in\bF_q^{m\times r}$, then $M\preceq I_m$ if and only if $V_2^\top V_1=I_r$.
\end{lemma}
\begin{proof}
  First assume that $M\preceq I_m$. Then $M-I$ has rank $m-r$, so we can write $M-I=W_1W_2^\top$ for some $W_1,W_2\in\bF_q^{m\times n-r}$. Then
  \begin{equation*}
    I_r=V_1V_2^\top+W_1W_2^\top=(V_1|W_1)(V_2|W_2)^\top,
  \end{equation*}
  where $(V_j|W_j)\in\bF_q^{m\times m}$ denotes the matrix obtained by concatenating $V_j$ and $W_j$ side-by-side. Thus $(V_2|W_2)^\top(V_1|W_1)=I_r$, which implies that $V_2^\top V_1=I_r$.

  For the opposite implication, assume that $V_2^\top V_1=I_r$. Then we may inductively construct $W_1,W_2\in\bF_q^{m\times n-r}$ such that $(V_2|W_2)^\top(V_1|W_1)=I_r$ by repeatedly adding a column to each of $W_1$ and $W_2$. For the induction to work, it suffices to show we can find $w_1,w_2\in\bF_2^m$ such that $V_2^\top w_1=0$, $w_2^\top V_1=0$, and $w_2^\top w_1=1$; we can then append $w_1,w_2$ to $V_1,V_2$ respectively and induct. As long as $w_2^\top w_1\neq 0$ we can always rescale one of the vectors to make their inner product $1$, so it suffices to show that there exist non-orthogonal vectors $w_1\in \ker V_2^\top$ and $w_2\in\ker V_1^\top$. If no such $w_1,w_2$ were to exist, we would have $\ker V_2^\top \subseteq (\ker V_1^\top)^\perp=\im V_1$, which in particular implies that some nonzero element of $\im V_1$ lies inside $\ker V_2^\top$. But this statement contradicts our assumption that $V_2^\top V_1=I_r$, so it must be that there exist non-orthogonal $w_1\in \ker V_2^\top$ and $w_2\in\ker V_1^\top$, as desired.
\end{proof}

We first turn to the task of bounding the expansion of $\bfG_2$ in Proposition~\ref{prop:G2exp}; we will subsequently bound the expansion of $\bfG_1$ in Proposition~\ref{prop:G1exp}, which is more involved. We will specifically show that $\bfG_2$ is close to the up-down walk on $\cM(1)$, and hence we first bound the expansion of this latter walk.

\begin{lemma}
  \label{lem:mat1ud}
  For $m\geq 2$, let $\cM=\cM_q^m$ be the matrix poset. Then
  \begin{equation*}
    \lambda(\bfW_{\cM(1)}^{\ud})\leq \frac{10}{q}.
  \end{equation*}
\end{lemma}
\begin{proof}
  Let $\bfW=\bfW_{\cM(1)}^{\ud}$. Given $e_1\otimes e_2\in\cM(1)$, we may sample $e_1'\otimes e_2'\sim\1_{e_1\otimes e_2}^\top\bfW$ as follows. First sample a uniformly random element $e_1^{(1)}\otimes e_2^{(1)}\in\cM(1)$. Let $E_1$ be the event that $e_1^{(1)}\in\spn\{e_1\}$ or $e_2^{(1)}\in\spn\{e_2\}$. If $E_1$ does not occur we sample uniformly random nonzero elements $e_1^{(2)}=\alpha_1e_1+\beta_1e_1^{(1)}\in\spn\{e_1,e_1^{(1)}\}$ and $e_2^{(2)}=\alpha_2e_2+\beta_2e_2^{(1)}\in\spn\{e_2,e_2^{(1)}\}$, so that $(\alpha_1,\beta_1)$ and $(\alpha_2,\beta_2)$ are sampled independently and uniformly from $\bF_q^2\setminus\{(0,0)\}$. Let $E_2$ be the event that $\alpha_1\alpha_2+\beta_1\beta_2=0$. If $E_2$ does not occur let $e_1^{(3)}=(\alpha_1\alpha_2+\beta_1\beta_2)^{-1}e_1^{(2)}$ and $e_2^{(3)}=e_2^{(2)}$. Let $E=E_1\cup E_2$. If $E$ does not occur, we let $e_1'\otimes e_2'=e_1^{(3)}\otimes e_2^{(3)}$, while if $E$ occurs we simply sample $e_1'\otimes e_2'$ (with fresh randomness) according to the correct distribution $\1_{e_1\otimes e_2}^\top\bfW$. The resulting $e_1'\otimes e_2'$ is distributed as $\1_{e_1\otimes e_2}^\top\bfW$ because by construction, conditioned on $E^C$ then $e_1\otimes e_2+e_1^{(1)}\otimes e_2^{(1)}$ is distributed as $\1_{e_1\otimes e_2}^\top\bfW_{\cM(1)}^{\uparrow}$, and by Lemma~\ref{lem:domI}, conditioned on $E^C$ then $e_1^{(3)}\otimes e_2^{(3)}$ is distributed as $\1_{e_1\otimes e_2+e_1^{(1)}\otimes e_2^{(1)}}^\top\bfW_{\cM(2)}^{\downarrow}$.

  Sample independent Bernoullis $B_1,B_2\sim\Ber((q-1)/(q^{m-1}-1))$ and let $F$ be the event that [$e_1^{(2)}\in\spn\{e_1\}$ and $B_1=0$] or [$e_2^{(2)}\in\spn\{e_2\}$ and $B_2=0$]. Then conditioned on $E_1^C\cap F^C$, it follows that $e_1^{(3)}$ and $e_2^{(3)}$ are independent uniformly random nonzero vectors drawn from $\bF_q^m$; note that the Bernoullis above ensure these vectors lie in $\spn\{e_1\}$ and $\spn\{e_2\}$ with the correct probability $(q-1)/(q^m-1)$ for a uniform distribution. Therefore as conditioned on $E^C\cap F^C=(E\cup F)^C$ we by definition have $e_1'\otimes e_2'=e_1^{(3)}\otimes e_2^{(3)}$, it follows that the total variation distance between $\1_{e_1\otimes e_2}^\top\bfW$ and $\Unif(\cM(1))$ is at most
  \begin{align*}
    \Pr[E\cup F]
    &\leq \Pr[E_1]+\Pr[E_2]+\Pr[F] \\
    &\leq 2\cdot\frac{q-1}{q^m-1} + \frac{1}{q} + 2\cdot\frac{q-1}{q^2-1} \\
    &\leq \frac{5}{q}.
  \end{align*}
  Thus by Lemma~\ref{lem:expclose} it follows that $\lambda(\bfW)\leq 10/q$, as desired.
\end{proof}

\begin{proof}[Proof of Proposition~\ref{prop:G2exp}]
  Let $\bfW$ be the random walk matrix of $\bfG$.  Given $M\in V(\bfG)$, so that $M\in\cM_q^n(2)$, we may sample $M'\sim\1_M^\top\bfW$ as follows. First sample a uniformly random element $L_1\in\cM_q^n(1)$ such that $L_1\prec M$. Let $L_2=M+L_1$, and sample a uniformly random element $L_3\in\cM_q^n(1)$ such that $\rank(L_1+L_3)=2$ (or equivalently, such that $L_1$ and $L_3$ have distinct row spans and distinct column spans). Let $M''=L_1+L_3$. Define $(2^{r+1}-4)$-dimensional subspaces $R,C\subseteq\bF_q^n$ as in Proposition~\ref{prop:tensornice}, and let $E_1$ be the event that $\rank(L_1+L_2+L_3)<3$, $E_2$ be the event that $\text{rowspan}(L_1+L_2+L_3)\cap R\neq\{0\}$, $E_3$ be the event that $\text{colspan}(L_1+L_2+L_3)\cap C\neq\{0\}$, and $E=E_1\cup E_2\cup E_3$. If $E^C$ occurs we let $M'=M''$, otherwise if $E$ occurs we sample $M'$ (with fresh randomness) according to the correct distribution $\1_M^\top\bfW$. By definition $M''$ is distributed according to $\1_M^\top\bfW_{\cM_q^n(2)}^{\du}$, so the total variation distance between $\1_M^\top\bfW$ and $\1_M^\top\bfW_{\cM_q^n(2)}^{\du}$ is at most
  \begin{align*}
    \Pr[E]
    &\leq \Pr[\text{rowspan}(L_3)\subseteq(R+\text{rowspan(M)})]+\Pr[\text{colspan}(L_3)\subseteq(C+\text{colspan(M)})] \\
    &\leq 2\cdot\frac{q^{2^{r+1}-4}-1}{q^n-1}.
  \end{align*}
  Thus by Lemma~\ref{lem:expclose} and Lemma~\ref{lem:mat1ud}, it follows that
  \begin{align*}
    \lambda(\bfG)
    &\leq \lambda(\bfW_{\cM_q^n(2)}^{\du})+\Pr[E] \\
    &= \lambda(\bfW_{\cM_q^n(1)}^{\ud})+\Pr[E] \\
    &\leq \frac{10}{q}+\frac{2q^{2^{r+1}-4}-1}{q^n-1} \\
    &\leq \frac{11}{q}.
  \end{align*}
\end{proof}

We now turn to analyzing the expansion of the graph $\bfG_1$ in Proposition~\ref{prop:tensornice} in order to prove Proposition~\ref{prop:G1exp}. As described earlier in this section, to prove Proposition~\ref{prop:G1exp}, we use a localization argument to show that the random walk on $\bfG_1$ is close to a simpler walk. This simpler walk is the rank $2\rightarrow 1\rightarrow 2$ up-down walk on matrices dominated by $I$; its expansion is bounded in the lemma below with a coupling-style argument, where here we consider the rank $1\rightarrow 2\rightarrow 1$ walk (which has the same expansion).

\begin{lemma}
  \label{lem:mat1uddomI}
  For $m\geq 2$, let $\cM'=\{M\in\cM_q^m:M\preceq I_m\}$ be the subposet of the matrix poset that is dominated by $I_m$. Then
  \begin{equation*}
    \lambda(\bfW_{\cM'(1)}^{\ud}) \leq \frac{8}{q}.
  \end{equation*}
\end{lemma}
\begin{proof}
  Let $\bfW=\bfW_{\cM'(1)}^{\ud}$. Given $e_1\otimes e_2\in\cM(1)$, so that $e_1^\top e_2=1$ by Lemma~\ref{lem:domI}, we may sample $e_1'\otimes e_2'\sim\1_{e_1\otimes e_2}^\top\bfW$ as follows. First sample uniformly nonzero vectors $e_1^{(1)}\in e_2^\perp\subseteq\bF_q^m$, $e_2^{(1)}\in e_1^\perp\subseteq\bF_q^m$. If ${e_1^{(1)}}^\top e_2^{(1)}\neq 0$, let $e_1^{(2)}=e_1^{(1)}$ and $e_2^{(2)}=({e_1^{(1)}}^\top e_2^{(1)})^{-1}\cdot e_2^{(1)}$. Sample uniformly random nonzero vectors $e_1^{(3)}\in\spn\{e_1,e_1^{(2)}\}$ and $e_2^{(3)}\in\spn\{e_2,e_2^{(2)}\}$. If ${e_1^{(3)}}^\top e_2^{(3)}\neq 0$, let $e_1^{(4)}=e_1^{(3)}$ and $e_2^{(4)}=({e_1^{(3)}}^\top e_2^{(3)})^{-1}\cdot e_2^{(3)}$. Let $E_1$ be the event that ${e_1^{(1)}}^\top e_2^{(1)}=0$, $E_2$ be the event that ${e_1^{(3)}}^\top e_2^{(3)}=0$, and $E=E_1\cup E_2$. If $E^C$ occurs, we let $e_1'\otimes e_2'=e_1^{(4)}\otimes e_2^{(4)}$, and if $E$ occurs we simply sample $e_1'\otimes e_2'$ (with fresh randomness) according to the correct distribution $\1_{e_1\otimes e_2}^\top\bfW$.

  To see that the above procedure indeed draws $e_1'\otimes e_2'$ from $\1_{e_1\otimes e_2}^\top\bfW$, observe that conditioned on $E^C$ (that is, conditioned on $E$ not occuring), then by Lemma~\ref{lem:domI}, the distribution of $e_1\otimes e_2+e_1^{(2)}\otimes e_1^{(2)}$ is exactly $1_{e_1\otimes e_2}^\top \bfW_{\cM'(1)}^{\uparrow}$, as by construction $e_1\otimes e_2+e_1^{(2)}\otimes e_1^{(2)}$ is a rank-2 matrix that dominates $e_1\otimes e_2$, and by symmetry it is equally likely to be any such matrix. Similarly, conditioned on $E^C$, the distribution of $e_1^{(4)}\otimes e_2^{(4)}$ is exactly $\1_{e_1\otimes e_2+e_1^{(2)}\otimes e_1^{(2)}}^\top\bfW_{\cM'(2)}^{\downarrow}$, which over the choice of $e_1\otimes e_2+e_1^{(2)}\otimes e_1^{(2)}$ yields the desired distribution $\1_{e_1\otimes e_2}^\top\bfW$.

  Now observe that when we do not condition on $E^C$, the subspaces $\spn\{e_1,e_1^{(2)}\}$ and $\spn\{e_2,e_2^{(2)}\}$ are independent uniformly random 2-dimensional subspaces containing $e_1$ and $e_2$ respectively, as $e_1^\top e_2=1$ and $e_1^{(2)}\sim\Unif(e_2^\perp)$, $e_2^{(2)}\sim\Unif(e_1^\perp)$. Sample independent Bernoullis $B_1,B_2\sim\Ber((q-1)/(q^{m-1}-1))$, and let $F_1$ be the event that [$e_1^{(3)}\in\spn\{e_1\}$ and $B_1=0$] or [$e_2^{(3)}\in\spn\{e_2\}$ and $B_2=0$]. Then conditioned on $F_1^C$, it follows that $e_1^{(3)}$ and $e_2^{(3)}$ are independent uniformly random nonzero vectors drawn from $\bF_q^m$; note that the Bernoullis above ensure these vectors lie in $\spn\{e_1\}$ and $\spn\{e_2\}$ with the correct probability $(q-1)/(q^m-1)$ for a uniform distribution. Thus letting $F=F_1\cup E_2$, then conditioned on $F^C$, $e_1^{(4)}\otimes e_2^{(4)}$ is a uniformly random element of $\cM'(1)$. Therefore as conditioned on $E^C\cap F^C=(E\cup F)^C$ we by definition have $e_1'\otimes e_2'=e_1^{(4)}\otimes e_2^{(4)}$, it follows that the total variation distance between $\1_{e_1\otimes e_2}^\top\bfW$ and $\Unif(\cM'(1))$ is at most
  \begin{align*}
    \Pr[E\cup F]
    &\leq \Pr[E_1]+\Pr[E_2|E_1^C]+\Pr[F_1] \\
    &\leq \frac{1}{q} + \frac{1}{q} + 2\cdot\frac{q-1}{q^2-1} \\
    &\leq \frac{4}{q}.
  \end{align*}
  Thus by Lemma~\ref{lem:expclose} it follows that $\lambda(\bfW)\leq 8/q$.
\end{proof}

Our localization argument described above also requires rapid mixing of local walks in order to relate the walk on $\bfG_2$ to the down-up walk whose expansion was bounded above in Lemma~\ref{lem:mat1uddomI}. Lemma~\ref{lem:localizedexp} below uses a coupling-style argument to show that the appropriate local walk mixes rapidly. To prove Lemma~\ref{lem:localizedexp}, we first need the following basic lemma.

\begin{lemma}
  \label{lem:perpwalk}
  For some $q,m$, define $\bfG$ to be the graph with
  \begin{align*}
    V(\bfG) &= \bF_q^m\setminus\{0\} \\
    E(\bfG) &= \{\{v_1,v_2\}\subseteq V(\bfG):v_1^\top v_2=0\}.
  \end{align*}
  Then
  \begin{equation*}
    \lambda(\bfG) \leq \frac{1}{\sqrt{q^{m-1}-1}}.
  \end{equation*}
\end{lemma}
\begin{proof}
  Let $\bfK$ be the complete graph without self-loops on the same vertex set $V(\bfG)$. By definition $\bfG^2=(q^{m-1}-1)\bfI+(q^{m-2}-1)\bfK$, as every nonzero $v\in\bF_q^m$ is orthogonal to $q^{m-1}-1$ other nonzero vectors, while every distinct nonzero $v_1,v_2$ share $q^{m-2}-1$ orthogonal nonzero vectors. Thus
  \begin{align*}
    \lambda(\bfG)^2
    &= \lambda(\bfG^2) \\
    &= \lambda((q^{m-1}-1)\bfI+(q^{m-2}-1)\bfK) \\
    &\leq \frac{1}{q^{m-1}-1}.
  \end{align*}
\end{proof}

\begin{lemma}
  \label{lem:localizedexp}
  For some prime power $q$ and some $m\geq 3$, define $\bfG$ to be the graph with
  \begin{align*}
    V(\bfG) &=  \{M\in\cM_q^m(1):M\prec I_m\} \\
    E(\bfG) &= \{\{L_1,L_2\}:L_1,L_2\in\cM_q^m(1),\; L_1+L_2\in\cM_q^m(2),\; L_1+L_2\prec I_m\}.
  \end{align*}
  Then
  \begin{equation*}
    \lambda(\bfG) \leq \frac{3}{q}.
  \end{equation*}
\end{lemma}
\begin{proof}
  Let $\bfG_0$ be the graph defined by
  \begin{align*}
    V(\bfG_0) &= \{(v_1,v_2)\in(\bF_q^m\setminus\{0\})^2:v_1^\top v_2\neq 0\} \\
    E(\bfG_0) &= \{\{(v_1,v_2),(v_1',v_2')\}\subseteq V(\bfG_0):v_1^\top v_2'=0,\;v_2^\top v_1'=0\},
  \end{align*}
  and let $\bfW_0$ be the associated random walk matrix. Given $(e_1,e_2)\in V(\bfW_0)$, we may sample $(e_1',e_2')\sim\1_{(e_1,e_2)}\bfW_0$ as follows. First sample uniformly random nonzero vectors $e_1^{(1)}\in e_2^\perp\subseteq\bF_q^m$, $e_2^{(1)}\in e_1^\perp\subseteq\bF_q^m$. Let $E$ be the event that ${e_1^{(1)}}^\top e_2^{(1)}=0$. If $E$ does not occur, let $e_1'=e_1^{(1)}$ and $e_2'=e_2^{(1)}$, while if $E$ occurs then simply sample $(e_1',e_2')\sim\1_{(e_1,e_2)}\bfW_0$ (with fresh randomness). Let $\bfW_\perp$ be the random walk matrix for the graph on vertex set $\bF_q^m\setminus\{0\}$ given in Lemma~\ref{lem:perpwalk}, and let $\bfW_{\text{swap}}$ be the walk on vertex set $(\bF_q^m\setminus\{0\})^2$ that simply walks from $(v_1,v_2)$ to $(v_2,v_1)$ with probability $1$. Then by definition $(e_1^{(1)},e_2^{(1)})$ is distributed as $\1_{(e_1,e_2)}^\top(\bfW_\perp\otimes\bfW_\perp)\bfW_{\text{swap}}$, so because the total variation distance between $(e_1^{(1)},e_2^{(1)})$ and $(e_1',e_2')$ is at most
  \begin{align*}
    \Pr[E]
    &\leq \frac{1}{q},
  \end{align*}
  it follows by Lemma~\ref{lem:expclose} and Lemma~\ref{lem:perpwalk} that
  \begin{align*}
    \lambda(\bfG_0)
    &\leq \lambda((\bfW_\perp\otimes\bfW_\perp)\bfW_{\text{swap}})+2\Pr[E] \\
    &\leq \lambda(\bfW_\perp)+2\Pr[E] \\
    &\leq \frac{1}{\sqrt{q^{m-1}-1}}+\frac{2}{q}.
  \end{align*}

  Now observe that by Lemma~\ref{lem:domI}, the graph $\bfG$ described in the lemma statement is precisely the image of the graph projection $\Pi:\bfG_0\rightarrow\bfG$ induced by the vertex mapping $\Pi:V(\bfG_0)\rightarrow V(\bfG)$ given by $\Pi(v_1,v_2)=(v_1^\top v_2)^{-1}\cdot v_1\otimes v_2$. Thus by Lemma~\ref{lem:projexp}, it follows that $\lambda(\bfG)\leq\lambda(\bfG_0)\leq 3/q$, so the result follows.
\end{proof}

We are now ready to present our main localization argument that proves Proposition~\ref{prop:G1exp}. Specifically, we decompose the walk on $\bfG_2$ into local walks whose expansion is bounded by Lemma~\ref{lem:localizedexp}, which allows us to relate the expansion of $\bfG_2$ to the expansion of the walk in Lemma~\ref{lem:mat1uddomI}.

\begin{proof}[Proof of Proposition~\ref{prop:G1exp}]
  Let $\bfW$ be the random walk matrix of $\bfG$. Recall that $\bfG$ has vertex set given by matrices in $M\in\cM_q^4(2)$ such that $M\prec I_4$, and $\bfG$ has an edge $\{L_1+L_2,L_1+L_3\}$ for all $L_1,L_2,L_3\in\cM_q^4(1)$ such that $L_1+L_2+L_3$ is a rank-1 matrix dominated by $I_4$. Let $\bfL=\bfI-\bfW$ be the normalized Laplacian for $\bfG$; recall that $\bfG$ is regular so $\bfW$ is symmetric. 

  We decompose the walk $\bfW$ into simpler walks using localization, and show that these simpler walks have good expansion. Given $L_1\in\cM_q^4(1)$ such that $L_1\prec I_4$, define $\bfG_{\text{complete}}^{\prec I-L_1}$ to be the complete graph with self loops on vertex set
  \begin{equation*}
    V(\bfG_{\text{complete}}^{\prec I-L_1}) = \{L_2\in\cM_q^4(1):L_2\prec I_4-L_1\},
  \end{equation*}
  and define $\bfG_{\text{local}}^{\prec I-L_1}$ to be the graph with
  \begin{align*}
    V(\bfG_{\text{local}}^{\prec I-L_1}) &= V(\bfG_{\text{complete}}^{\prec I-L_1}) \\
    E(\bfG_{\text{local}}^{\prec I-L_1}) &= \{\{L_2,L_3\}:L_2,L_3\in\cM_q^4(1),\; L_2+L_3\in\cM_q^4(2),\; L_2+L_3\prec I_4-L_1\}
  \end{align*}
  Note that in the final line above $L_2+L_3\prec I_4-L_1$ if and only if $L_1+L_2+L_3\prec I_4$. Thus $\bfG_{\text{local}}^{\prec I-L_1}$ is the graph specifying the pairs $\{L_2,L_3\}$ for which $\{L_1+L_2,L_1+L_3\}\in E(\bfG)$, while $\bfG_{\text{complete}}^{\prec I-L_1}$ is the complete graph with self loops on the same vertex set. By definition for all $L_1\in V(\bfG)$, then $\bfG_{\text{local}}^{\prec I-L_1}$ is isomorphic to the graph defined in Lemma~\ref{lem:localizedexp}, so letting $\lambda_{\text{local}}=3/q$, then Lemma~\ref{lem:localizedexp} implies that $\lambda(\bfG_{\text{local}}^{\prec I-L_1})\leq\lambda_{\text{local}}$.

  We now set up the necessary notation for our localization. Let $\bfL_{\text{local}}^{\prec I-L_1}=\bfI-\bfW_{\text{local}}^{\prec I-L_1}$ and $\bfL_{\text{complete}}^{\prec I-L_1}=\bfI-\bfW_{\text{complete}}^{\prec I-L_1}$ be the respecive normalized Laplacians. As all the random walks above are induced by regular undirected graphs, the stationary distributions are all uniform. Given such a graph $\bfG'$ on a vertex set $V'$ with random walk matrix $\bfW'$ and Laplacian $\bfL'$, then we write $\{v_1,v_2\}\sim\bfW'$ to denote the distribution obtained by sampling $v_1\sim\Unif(V(\bfG'))$ and $v_2\sim\1_{v_1}\bfW'$. Furthermore, for $f\in\bR^{V'}$ we let $\langle f,\bfL' f\rangle=\bE_{v\sim\Unif(V')}[f(v)\bfL'f(v)]=f^\top\bfL'f/|V'|$. Also let $U_1$ denote the uniform distribution over matrices in $\{L\in\cM(1):L\prec I_4\}$.

  Now for every $f\in\bR^{V(\bfG)}$, we have
  \begin{align*}
    \langle f,\bfL f\rangle
    &= \bE_{\{M,M'\}\sim\bfW}\left[\frac{(f(M)-f(M'))^2}{2}\right] \\
    &= \bE_{L_1\sim U_1}\bE_{\{L_2,L_3\}\sim\bfW_{\text{local}}^{\prec I-L_1}}\left[\frac{(f(L_1+L_2)-f(L_1+L_3))^2}{2}\right].
  \end{align*}
  Defining $f_{L_1}\in\bR^{V(\bfG_{\text{local}}^{\prec I-L_1})}$ by $f_{L_1}(L_2)=f(L_1+L_2)$, the above becomes
  \begin{align*}
    \langle f,\bfL f\rangle
    &= \bE_{L_1\sim U_1}\bE_{\{L_2,L_3\}\sim\bfW_{\text{local}}^{\prec I-L_1}}\left[\frac{(f_{L_1}(L_2)-f_{L_1}(L_3))^2}{2}\right] \\
    &= \bE_{L_1\sim U_1}\left[\langle f_{L_1},\bfL_{\text{local}}^{\prec I-L_1}f_{L_1}\rangle\right] \\
    &\geq \bE_{L_1\sim U_1}\left[(1-\lambda_{\text{local}})\langle f_{L_1},\bfL_{\text{complete}}^{\prec I-L_1}f_{L_1}\rangle\right] \\
    &= (1-\lambda_{\text{local}})\bE_{L_1\sim U_1}\bE_{\{L_2,L_3\}\sim\bfW_{\text{complete}}^{\prec I-L_1}}\left[\frac{(f(L_1+L_2)-f(L_1+L_3))^2}{2}\right].
  \end{align*}
  Consider the poset $\cM'=\{M\in\cM_q^4:M\prec I_4\}$. Observe that for $L_1\sim U_1$, $\{L_2,L_3\}\sim\bfW_{\text{complete}}^{\prec I-L_1}$, the resulting distribution of pairs $\{M=\{L_1+L_2\},M'\{L_1+L_3\}\}$ is precisely the distribution $\{M,M'\}\sim\bfW_{\cM'(2)}^{\du}$ given by the rank $2\rightarrow 1\rightarrow 2$ down-up walk $\bfW_{\cM'(2)}^{\du}$ on $\cM'$. Thus letting $\bfG_{\text{complete}}$ denote the complete graph with self-loops on vertex set $V(\bfG)$, then
  \begin{align*}
    \langle f,\bfL f\rangle
    &\geq (1-\lambda_{\text{local}})\bE_{\{M,M'\}\sim\bfW_{\cM'(2)}^{\du}}\left[\frac{f(M)-f(M')}{2}\right] \\
    &= (1-\lambda_{\text{local}})\langle f,\bfL_{\cM'(2)}^{\du}f\rangle \\
    &\geq (1-\lambda_{\text{local}})(1-\lambda(\bfW_{\cM'(2)}^{\du}))\langle f,\bfL_{\text{complete}}f\rangle.
  \end{align*}
  Therefore
  \begin{align*}
    1-\lambda(\bfG)
    &\geq (1-\lambda_{\text{local}})(1-\lambda(\bfW_{\cM'(2)}^{\du})).
  \end{align*}
  Thus because $\lambda(\bfW_{\cM'(2)}^{\du})=\lambda(\bfW_{\cM'(1)}^{\ud})\leq 8/q$ by Lemma~\ref{lem:mat1uddomI}, we have
  \begin{align*}
    \lambda(\bfG)
    &\leq 1-(1-\lambda_{\text{local}})(1-\lambda(\bfW_{\cM'(2)}^{\du})) \\
    &\leq 1-\left(1-\frac{3}{q}\right)\left(1-\frac{8}{q}\right) \\
    &\leq \frac{11}{q}.
  \end{align*}
\end{proof}

\subsubsection{Connectedness of all links}
\label{sec:linksconn}
In this section, we show that for all $-1\leq i\leq r-3$, the graphs $\bfG_1,\bfG_2$ in Proposition~\ref{prop:tensornice} are connected, thereby proving Proposition~\ref{prop:localconn}.

\begin{lemma}
  \label{lem:G1conn}
  For integers $r\geq 1,b\geq 2,n\geq 2^{r+1}$, let $\bfG$ be the graph $\bfG_1$ for some $0\leq i\leq r-2$ as defined in Proposition~\ref{prop:tensornice}. Then $\bfG$ is connected and nonbipartite.
\end{lemma}
\begin{proof}
  We prove the result by induction. For every $0\leq k\leq 2^{r-i-2}$ and $1\leq\ell\leq 2^{r-i-2}$, let $s=s(k,\ell)=k+2\ell+2^{r-i-2}$, and let $\bfG^{(k,\ell)}$ be the graph defined by
  \begin{align*}
    V(\bfG^{(k,\ell)}) = \{M\in\cM_q^s(k+\ell):\;&M\prec I_s\} \\
    E(\bfG^{(k,\ell)}) = \Big\{\{L_1+L_2,L_1+L_3\}:\;&L_1\in\cM_q^s(k),\; L_2,L_3\in\cM_q^s(\ell),\\
                       &L_1+L_2+L_3\in\cM_q^s(k+2\ell),\\
                       &L_1+L_2+L_3\prec I_s\Big\}.
  \end{align*}
  Let $\cM=\cM_q^s$, and define the subposet $\cM'=\{M\in\cM:M\preceq I_s\}$. By definition $\bfG=\bfG^{(2^{r-i-2},2^{r-i-2})}$. We will show by induction that $\bfG^{(k,\ell)}$ is connected and nonbipartite for every $0\leq k\leq 2^{r-i-2}$ and $1\leq\ell\leq 2^{r-i-2}$. For the base case, if $k=0,\ell=1$, then $\bfG^{(0,1)}$ is precisely the graph in Lemma~\ref{lem:localizedexp}, and thus because $b\geq 2$ so that $q=2^b\geq 4$, Lemma~\ref{lem:localizedexp} implies that $\lambda(\bfG^{(0,1)})<1$ and thus $\bfG^{(0,1)}$ is connected and nonbipartite.

  For the inductive step, first assume that $k\geq 1$, and assume that $\bfG^{(k-1,\ell)}$ is connected and nonbipartite. For any $M,M'\in V(\bfG^{(k,\ell)})$ we want to show there exists a path in $\bfG^{(k,\ell)}$ from $M$ to $M'$. Choosing some $L_0,L_0'\in\cM'(1)$ such that $L_0\prec M,L_0'\prec M'$, then by Lemma~\ref{lem:localizedexp} there exists a path $L_0=L_0^{(0)},L_0^{(1)},\dots,L_0^{(t)}=L_0'$ in $\cM'(1)$ from $L_0$ to $L_0'$ such that for every $j$ we have $L_0^{(j)}+L_0^{(j+1)}\in\cM'(2)$. Thus by the connectedness of $\bfG^{(k-1,\ell)}$, for each $L_0^{(j)}$ and for any $M_0,M_0'\in\cM'(k+\ell-1)$ such that $M_0,M_0'\prec I_s-L_0^{(j)}$, there exists a path from $L_0^{(j)}+M_0$ to $L_0^{(j)}+M_0'$ in $\bfG^{(k,\ell)}$ consisting of matrices in $V(\bfG^{(k,\ell)})$ that dominate $L_0^{(j)}$. Specifically, this statement follows from the fact that the subposet of $\cM_q^s$ dominated by $I_s-L_0^{(j)}$ is isomorphic to the subposet of $\cM_q^{s-1}$ dominated by $I_{s-1}$. Therefore if we choose $M_0^{(0)}=M-L_0^{(0)}$, then we may inductively choose $M_0^{(j+1)}\in\cM'(k+\ell-1)$ to be some matrix such that $L_0^{(j)}+M_0^{(j+1)}\in\cM'(k+\ell)$ dominates $L_0^{(j+1)}$, so we obtain a path in $\bfG^{(k,\ell)}$ from $M$ to $L_0'+M_0^{(t)}$. But we also showed above there exists a path in $\bfG^{(k,\ell)}$ from $L_0'+M_0^{(t)}$ to $L_0'+(M'-L_0')=M'$, so we obtain a path from $M$ to $M'$, as desired. Note that because $\bfG^{(k-1,\ell)}$ is nonbipartite, when invoking the inductive hypothesis we can choose any parity for the length of the paths, and thus we obtain paths from $M$ to $M'$ of any parity. Therefore $\bfG^{(k,\ell)}$ is connected and nonbipartite.

  Now assume that $k=0$, fix $\ell\geq 2$, and assume that $\bfG^{(0,\ell-1)}$ is connected and nonbipartite. Let $\cM''=\{M\in\cM_q^{s-1}:M\preceq I_{s-1}\}$, and define a bipartite graph $\bfG'$ by
  \begin{align*}
    V_1(\bfG') &= \cM''(\ell-1) \\
    V_2(\bfG') &= \cM''(\ell) \\
    V(\bfG') &= V_1(\bfG')\sqcup V_2(\bfG') \\
    E(\bfG') &= \{(M,M')\in V_1(\bfG')\times V_2(\bfG'):M+M'\in\cM''(2\ell-1)\}.
  \end{align*}
  We will first show that $\bfG'$ is connected using the assumption that $\bfG^{(0,\ell-1)}$ is connected, and then show that $\bfG^{(0,\ell)}$ is connected using the fact that $\bfG'$ is connected. Both proofs are similar to the proof above that $\bfG^{(k,\ell)}$ is connected for $k\geq 1$.

  For any $M,M'\in V_2(\bfG')$, we want to show that there exists a path in $\bfG'$ from $M$ to $M'$. For this purpose, again choose $L_0,L_0'\in\cM''(1)$ such that $L_0\prec M,L_0'\prec M'$, so that by Lemma~\ref{lem:localizedexp} there exists a path $L_0=L_0^{(0)},L_0^{(1)},\dots,L_0^{(t)}=L_0'$ in $\cM''(1)$ from $L_0$ to $L_0'$ such that for every $j$ we have $L_0^{(j)}+L_0^{(j+1)}\in\cM''(2)$. Letting $M^{(0)}=M$, for each $j$ we may choose $M^{(j)}\in\cM''(\ell)$ to be some matrix that dominates $L_0^{(j-1)}+L_0^{(j)}$. Then the inductive hypothesis implies that there is a path in $\bfG'$ from each $M^{(j)}$ to $M^{(j+1)}$, and from $M^{(t)}$ to $M'$. Specifically, as the subposet of $\cM_q^{s-1}$ dominated by $I_{s-1}-L_0^{(j)}$ is isomorphic to the subposet of $\cM_q^{s-2}$ dominated by $I_{s-2}$, we may construct the desired path in $\bfG'$ from $M^{(j)}$ to $M^{(j+1)}$ by taking a path of even length in $\bfG^{(0,\ell-1)}$ from $M^{(j)}-L_0^{(j)}$ to $M^{(j+1)}-L_0^{(j)}$ under the poset isomorphism above, and then adding $L_0^{(j)}$ to every other element in the path. Thus we have constructed a path in $\bfG'$ from any $M$ to any $M'$, so $\bfG'$ is connected.

  We now show that $\bfG^{(0,\ell)}$ is connected using a similar argument. For any $M,M'\in V(\bfG^{(0,\ell)})$, we want to show that there exists a path in $\bfG^{(0,\ell)}$ from $M$ to $M'$. For this purpose, again choose $L_0,L_0'\in\cM'(1)$ such that $L_0\prec M,L_0'\prec M'$, so that by Lemma~\ref{lem:localizedexp} there exists a path $L_0=L_0^{(0)},L_0^{(1)},\dots,L_0^{(t)}=L_0'$ in $\cM'(1)$ from $L_0$ to $L_0'$ such that for every $j$ we have $L_0^{(j)}+L_0^{(j+1)}\in\cM'(2)$. Letting $M^{(0)}=M$, for each $j$ we may choose $M^{(j)}\in\cM'(\ell)$ to be some matrix that dominates $L_0^{(j-1)}+L_0^{(j)}$. Then the inductive hypothesis implies that there is a path in $\bfG^{(0,\ell)}$ from each $M^{(j)}$ to $M^{(j+1)}$, and from $M^{(t)}$ to $M'$. Specifically, as the subposet of $\cM_q^s$ dominated by $I_s-L_0^{(j)}$ is isomorphic to the subposet of $\cM_q^{s-1}$ dominated by $I_{s-1}$, we may construct the desired path in $\bfG^{(0,\ell)}$ from $M^{(j)}$ to $M^{(j+1)}$ by taking a path in $\bfG'$ from $M^{(j)}-L_0^{(j)}$ to $M^{(j+1)}-L_0^{(j)}$ under the poset isomorphism above, and then adding $L_0^{(j)}$ to every other element in the path. Because $\bfG'$ is connected and bipartite, each of these paths in $\bfG'$ has even length, so we obtain a path from $M$ to $M'$ in $\bfG^{(0,\ell)}$ of even length. Therefore $\bfG^{(0,\ell)}$ has an even length path between every two vertices, so it is connected and nonbipartite.
\end{proof}

We will use the following basic lemma to show that $\bfG_2$ is connected.

\begin{lemma}
  \label{lem:G2connbase}
  Let $n\geq 3$, and let $R,C\subseteq\bF_q^n$ be subspaces of dimension $\leq n-3$. Let $\bfG$ be the graph defined by
  \begin{align*}
    V(\bfG) = \{M\in\cM_q^n(1):\;&\text{rowspan}(M)\cap R=\{0\},\\
                                 &\text{colspan}(M)\cap C=\{0\}\} \\
    E(\bfG) = \Big\{\{M,M'\}:\;&M,M'\in\cM_q^n(1),\\
                                 &M+M'\in\cM_q^n(2),\\
                                 &\text{rowspan}(M+M')\cap R=\{0\}\\
                                 &\text{colspan}(M+M')\cap C=\{0\}\Big\}.
  \end{align*}
  Then $\bfG$ is connected and nonbipartite.
\end{lemma}
\begin{proof}
  It suffices to show that $\bfG^2$ is a complete graph. For this purpose, consider any $M,M'\in V(\bfG)$. Then $R+\text{rowspan}(M)+\text{rowspan}(M')$ and $C+\text{colspan}(M)+\text{colspan}(M')$ are subspaces of $\bF_q^{n}$ of dimension $\leq n-1$, so there exist vectors $v_1$ and $v_2$ in $\bF_q^n$ that avoid these two subspaces respectively. Then $\{M,v_1\otimes v_2\}$ and $\{v_1\otimes v_2,M'\}$ are edges in $\bfG$, which form a length-2 path in $\bfG$ from $M$ to $M'$. Therefore indeed $\bfG^2$ is complete, so $\bfG$ is connected and nonbipartite.
\end{proof}

\begin{lemma}
  \label{lem:G2conn}
  For integers $r\geq 1,b\geq 2,n\geq 2^{r+1}$, let $\bfG$ be the graph $\bfG_2$ for some $-1\leq i\leq r-2$ as defined in Proposition~\ref{prop:tensornice}. Then $\bfG$ is connected and nonbipartite.
\end{lemma}
\begin{proof}[Proof sketch]
  The proof is similar to the proof of Lemma~\ref{lem:G1conn}, except that for places in the proof of Lemma~\ref{lem:G1conn} that required a matrix $M$ to be dominated by another matrix, we instead require that $\text{rowspan}(M)$ and $\text{colspan(M)}$ are disjoint from some appropriate vector spaces. We present the details below.

  We prove the result by induction. For every $0\leq k\leq 2^{r-i-2}$ and $1\leq\ell\leq 2^{r-i-2}$, let $C=C(k,\ell)=R=R(k,\ell)=\bF_q^{2^{r+1}-2^{r-i-2}-k-2\ell}\times\{0\}^{n-(2^{r+1}-2^{r-i-2}-k-2\ell)}$, and let $\bfG^{(k,\ell)}$ be the graph defined by
  \begin{align*}
    V(\bfG^{(k,\ell)}) = \{M\in\cM_q^n(k+\ell):\;&\text{rowspan}(M)\cap R=\{0\},\\
                                                 &\text{colspan}(M)\cap C=\{0\}\} \\
    E(\bfG^{(k,\ell)}) = \Big\{\{L_1+L_2,L_1+L_3\}:\;&L_1\in\cM_q^n(k),\; L_2,L_3\in\cM_q^n(\ell),\\
                                                 &L_1+L_2+L_3\in\cM_q^n(k+2\ell),\\
                                                 &\text{rowspan}(L_1+L_2+L_3)\cap R=\{0\}\\
                                                 &\text{colspan}(L_1+L_2+L_3)\cap C=\{0\}\Big\}.
  \end{align*}
  Let $\cM=\cM_q^n$, and define the subposet $\cM'=\{M\in\cM:\text{rowspan}(M)\cap R=\{0\},\;\text{colspan}(M)\cap C=\{0\}\}$. By definition $\bfG=\bfG^{(2^{r-i-2},2^{r-i-2})}$. We will show by induction that $\bfG^{(k,\ell)}$ is connected and nonbipartite for every $0\leq k\leq 2^{r-i-2}$ and $1\leq\ell\leq 2^{r-i-2}$. For the base case, if $k=0,\ell=1$, then $\bfG^{(0,1)}$ is precisely the graph in Lemma~\ref{lem:G2connbase}, and therefore $\bfG^{(0,1)}$ is connected and nonbipartite.


  For the inductive step, first assume that $k\geq 1$, and assume that $\bfG^{(k-1,\ell)}$ is connected and nonbipartite. For any $M,M'\in V(\bfG^{(k,\ell)})$ we want to show there exists a path in $\bfG^{(k,\ell)}$ from $M$ to $M'$. Choosing some $L_0,L_0'\in\cM'(1)$ such that $L_0\prec M,L_0'\prec M'$, then by Lemma~\ref{lem:G2connbase} there exists a path $L_0=L_0^{(0)},L_0^{(1)},\dots,L_0^{(t)}=L_0'$ in $\cM'(1)$ from $L_0$ to $L_0'$ such that for every $j$ we have $L_0^{(j)}+L_0^{(j+1)}\in\cM'(2)$. Now by the connectedness of $\bfG^{(k-1,\ell)}$, for each $L_0^{(j)}$ and for any $M_0,M_0'\in\cM'(k+\ell-1)$ such that the row and column spans of $M_0,M_0'$ avoid $R+\text{rowspan}(L_0^{(j)})$ and $C+\text{colspan}(L_0^{(j)})$ respectively, there exists a path from $L_0^{(j)}+M_0$ to $L_0^{(j)}+M_0'$ in $\bfG^{(k,\ell)}$ consisting of matrices in $V(\bfG^{(k,\ell)})$ that dominate $L_0^{(j)}$. Specifically, this statement follows from the fact that the subposet of $\cM_q^n$ consisting of matrices whose row and column spans avoid $R(k,\ell)+\text{rowspan}(L_0^{(j)})$ and $C(k,\ell)+\text{colspan}(L_0^{(j)})$ respectively is isomorphic to the subposet of $\cM_q^n$ whose row and column spans avoid $R(k-1,\ell)$ and $C(k-1,\ell)$ respectively. Therefore if we choose $M_0^{(0)}=M-L_0^{(0)}$, then we may inductively choose $M_0^{(j+1)}\in\cM'(k+\ell-1)$ to be some matrix such that $L_0^{(j)}+M_0^{(j+1)}\in\cM'(k+\ell)$ dominates $L_0^{(j+1)}$, so we obtain a path in $\bfG^{(k,\ell)}$ from $M$ to $L_0'+M_0^{(t)}$. But we also showed above there exists a path in $\bfG^{(k,\ell)}$ from $L_0'+M_0^{(t)}$ to $L_0'+(M'-L_0')=M'$, so we obtain a path from $M$ to $M'$, as desired. Note that because $\bfG^{(k-1,\ell)}$ is nonbipartite, when invoking the inductive hypothesis we can choose any parity for the length of the paths, and thus we obtain paths from $M$ to $M'$ of any parity. Therefore $\bfG^{(k,\ell)}$ is connected and nonbipartite.

  Now assume that $k=0$, fix $\ell\geq 2$, and assume that $\bfG^{(0,\ell-1)}$ is connected and nonbipartite. Let $R'=C'=\bF_q^{2^{r+1}-2^{r-i-2}-(2\ell-1)}\times\{0\}^{n-(2^{r+1}-2^{r-i-2}-(2\ell-1))}\subseteq\bF_q^n$ and $\cM''=\{M\in\cM:\text{rowspan}(M)\cap R'=\{0\},\;\text{colspan}(M)\cap C'=\{0\}\}$, and define a bipartite graph $\bfG'$ by
  \begin{align*}
    V_1(\bfG') &= \cM''(\ell-1) \\
    V_2(\bfG') &= \cM''(\ell) \\
    V(\bfG') &= V_1(\bfG')\sqcup V_2(\bfG') \\
    E(\bfG') &= \{(M,M')\in V_1(\bfG')\times V_2(\bfG'):M+M'\in\cM''(2\ell-1)\}.
  \end{align*}
  We will first show that $\bfG'$ is connected using the assumption that $\bfG^{(0,\ell-1)}$ is connected, and then show that $\bfG^{(0,\ell)}$ is connected using the fact that $\bfG'$ is connected. Both proofs are similar to the proof above that $\bfG^{(k,\ell)}$ is connected for $k\geq 1$.

  For any $M,M'\in V_2(\bfG')$, we want to show that there exists a path in $\bfG'$ from $M$ to $M'$. For this purpose, again choose $L_0,L_0'\in\cM''(1)$ such that $L_0\prec M,L_0'\prec M'$, so that by Lemma~\ref{lem:G2connbase} there exists a path $L_0=L_0^{(0)},L_0^{(1)},\dots,L_0^{(t)}=L_0'$ in $\cM''(1)$ from $L_0$ to $L_0'$ such that for every $j$ we have $L_0^{(j)}+L_0^{(j+1)}\in\cM''(2)$. Letting $M^{(0)}=M$, for each $j$ we may choose $M^{(j)}\in\cM''(\ell)$ to be some matrix that dominates $L_0^{(j-1)}+L_0^{(j)}$. Then the inductive hypothesis implies that there is a path in $\bfG'$ from each $M^{(j)}$ to $M^{(j+1)}$, and from $M^{(t)}$ to $M'$. Specifically, as the subposet of $\cM_q^n$ consisting of matrices whose row and column spans avoid $R'+\text{rowspan}(L_0^{(j)})$ and $C'+\text{colspan}(L_0^{(j)})$ respectively is isomorphic to the subposet of $\cM_q^n$ whose row and column spans avoid $R(k,\ell-1)$ and $C(k,\ell-1)$ respectively, we may construct the desired path in $\bfG'$ from $M^{(j)}$ to $M^{(j+1)}$ by taking a path of even length in $\bfG^{(0,\ell-1)}$ from $M^{(j)}-L_0^{(j)}$ to $M^{(j+1)}-L_0^{(j)}$ under the poset isomorphism above, and then adding $L_0^{(j)}$ to every other element in the path. Thus we have constructed a path in $\bfG'$ from any $M$ to any $M'$, so $\bfG'$ is connected.

  We now show that $\bfG^{(0,\ell)}$ is connected using a similar argument. For any $M,M'\in V(\bfG^{(0,\ell)})$, we want to show that there exists a path in $\bfG^{(0,\ell)}$ from $M$ to $M'$. For this purpose, again choose $L_0,L_0'\in\cM'(1)$ such that $L_0\prec M,L_0'\prec M'$, so that by Lemma~\ref{lem:G2connbase} there exists a path $L_0=L_0^{(0)},L_0^{(1)},\dots,L_0^{(t)}=L_0'$ in $\cM'(1)$ from $L_0$ to $L_0'$ such that for every $j$ we have $L_0^{(j)}+L_0^{(j+1)}\in\cM'(2)$. Letting $M^{(0)}=M$, for each $j$ we may choose $M^{(j)}\in\cM'(\ell)$ to be some matrix that dominates $L_0^{(j-1)}+L_0^{(j)}$. Then the inductive hypothesis implies that there is a path in $\bfG^{(0,\ell)}$ from each $M^{(j)}$ to $M^{(j+1)}$, and from $M^{(t)}$ to $M'$. Specifically, as the subposet of $\cM_q^n$ consisting of matrices whose row and column spans avoid $R+\text{rowspan}(L_0^{(j)})$ and $C+\text{colspan}(L_0^{(j)})$ is isomorphic to the subposet of $\cM_q^n$ whose row and column spans avoid $R'$ and $C'$ respectively, we may construct the desired path in $\bfG^{(0,\ell)}$ from $M^{(j)}$ to $M^{(j+1)}$ by taking a path in $\bfG'$ from $M^{(j)}-L_0^{(j)}$ to $M^{(j+1)}-L_0^{(j)}$ under the poset isomorphism above, and then adding $L_0^{(j)}$ to every other element in the path. Because $\bfG'$ is connected and bipartite, each of these paths in $\bfG'$ has even length, so we obtain a path from $M$ to $M'$ in $\bfG^{(0,\ell)}$ of even length. Therefore $\bfG^{(0,\ell)}$ has an even length path between every two vertices, so it is connected and nonbipartite.
\end{proof}


\begin{proof}[Proof of Proposition~\ref{prop:localconn}]
  The result follows from Proposition~\ref{prop:tensornice}, Lemma~\ref{lem:G1conn}, and Lemma~\ref{lem:G2conn}.
\end{proof}

\section{Simplicial HDXs of subpolynomial degree}
\label{sec:cayconstruct}
In this section, we present our main construction of simplicial high-dimensional expanders, which are Cayley complexes as defined in Section~\ref{sec:caycomplex} generated by the Grassmannian HDXs of Section~\ref{sec:grassHDX}. Below, we think of the parameters $r,b$ as fixed constants while $n\rightarrow\infty$.

\begin{theorem}
  \label{thm:maincayley}
  For integers $r\geq 1$, $b\geq 5$, $n\geq 2^{r+1}$, let $q=2^b$ and $k=bn^2$, and let $X=X^{r,b,n}$ be the rank-$r$ $\bF_2$-Grassmannian complex in ambient vector space $\bF_2^k$ given by Definition~\ref{def:grassconstruct}. Then $Y=\Cay(\bF_2^k,\beta(X))$ is a rank-$(r+1)$ simplicial complex with
  \begin{equation*}
    \lambda(Y) \leq \frac{11}{q-11}.
  \end{equation*}
  Furthermore, $Y$ has $|Y(0)|=2^{bn^2}$ vertices, and each vertex is contained in at most $2^{2^{r+2}bn}$ faces.
\end{theorem}
\begin{proof}
  The local expansion bound follows directly from Lemma~\ref{lem:cayleyF2k} and Corollary~\ref{cor:applytrickle}. To see that every vertex in $Y$ is contained in at most $2^{2^{r+2}bn}$ faces, consider that by definition the number of faces containing a vertex in $Y$ equals the number $|X|$ of faces in $X$. Now by Lemma~\ref{lem:lowerfaces}, because the matrix $G_{\text{Had}}^{(r,i)}$ has $2^{r+1}-2^{r-i}$ nonzero rows, the number $|X(i)|$ of rank-$i$ faces in $X$ is at most the number of tuples $(e_1^{(1)}\otimes e_1^{(2)},\dots,e_{2^{r+1}-2^{r-i}}^{(1)}\otimes e_{2^{r+1}-2^{r-i}}^{(2)})$ of rank-1 matrices in $(\bF_q^n)^{\otimes 2}$, which is at most $(q^n)^{2\cdot(2^{r+1}-2^{r-i})}$. Thus
  \begin{align*}
    |X|
    &\leq \sum_{i=-1}^r(q^n)^{2\cdot(2^{r+1}-2^{r-i})} \leq 2^{2^{r+2}bn}.
  \end{align*}
\end{proof}

Theorem~\ref{thm:maincayley} gives simplicial HDXs of all dimensions of subpolynomial degree. To the best of our knowledge, there were essentially only two previously known constructions of such objects, namely Ramanujan complexes \cite{cartwright_ramanujan_2003,li_ramanujan_2004,lubotzky_ramanujan_2005,lubotzky_explicit_2005} and coset complexes \cite{kaufman_construction_2018} both of which in fact achieve constant degree. There have also been several modifications, extensions, and spinoffs that apply these constructions to obtain more constant-degree simplicial HDXs \cite{friedgut_hyper-regular_2020,odonnell_high-dimensional_2022-1,dikstein_new_2022}.

These prior constructions are all group theoretic in nature, and in particular arise from matrix groups over finite fields. The analysis of Ramanujan complexes is algebraic and highly involved, while the analysis of coset complex HDXs \cite{kaufman_construction_2018,odonnell_high-dimensional_2022-1} is more elementary; see in particular the analysis in \cite{harsha_note_2019} of the Kaufman-Oppenheim construction \cite{kaufman_construction_2018}.

Though our construction in Theorem~\ref{thm:maincayley} does not achieve optimal (constant) degree, it is based on fundamentally different techniques than the prior constructions. In subsequent sections, we discuss additional properties of our construction which may be of interest. Specifically, we present a coding theoretic interpretation of the construction, which we also use to obtain a characterization of its 1st $\bF_2$-homology group. In particular, we obtain $N$-vertex simplicial HDXs with 1st $\bF_2$-homology of dimension $\Omega(\log^2N)$, whereas to the best of our knowledge prior constructions only achieved dimension $\Omega(\log N)$.

\section{Coding theoretic view of $\bF_2$-Grassmannian complexes}
\label{sec:coding}
This section provides a useful coding theoretic view of rank-1 $\bF_2$-Grassmannian complexes, or equivalently, of the 1-skeleton of higher-rank complexes. As described in Section~\ref{sec:intro}, this view builds upon the work of \cite{vadhan_construction_2018}, and can be seen as motivation for the construction in Section~\ref{sec:grassHDX}. This coding theoretic view also facilitates the study in Section~\ref{sec:cayleyhom} of the homology of Cayley simplicial complexes generated by $\bF_2$-Grassmannian complexes.

At a high level, for a rank-1 Grassmannian complexe $X$, we construct an associated LDPC code with checks of degree 3. If $X$ has good sparsity within its ambient vector space, then the associated code has good rate. If $X$ has good expansion, then the associated code has good distance. The details are given below.

Below, $\bF_2$-Grassmannian complex $X$, we often refer to rank-1 faces $x_1=\{0,x_0,x_0',x_0+x_0'\}\in X(1)$ by their three nonzero elements $\{x_0,x_0',x_0+x_0'\}$, which we call a \textbf{triangle} in $X(1)$. In this section we focuse on rank-$1$ complexes $X$, for which the basisification $\beta(X)$ and the 1-skeleton graph $\bfG_X$ are the same graph, whose vertices are $X(0)$ and whose edges are all 2-element subsets $\{x_0,x_0'\}$ of triangles in $X(1)$. Thus in the rank-1 case $\beta(X)$ and therefore also $\bfG_X$ contain all the information about $X$, so we can view $X$, $\beta(X)$, and $\bfG_X$ as the same object. For instance, all three objects have the same (local) expansion $\lambda(X)=\lambda(\beta(X))=\lambda(\bfG_X)$, which we refer to as simply the expansion of $X$. We will therefore sometimes use the shorthand $\Cay(\bF_2^k,X)$ to denote $\Cay(\bF_2^k,\beta(X))$.

\begin{definition}
  Let $X$ be a Grassmannian complex in ambient vector space $\bF_2^k$. Define an associated generating matrix $G_X\in\bF_2^{X(0)\times k}$ whose rows are given by the elements of $X(0)$. Also define an associated parity check matrix $H_X\in\bF_2^{X(1)\times X(0)}$ whose rows are given by the indicator vectors of traingles in $X(1)$. That is, the row associated to each $x_1\in X(1)$ has exactly three $1$s corresponding to the three nonzero elemments of $x_1$.
\end{definition}

For every Grassmannian complex $X$ over $\bF_2$, we have two associated codes, namely $\im G_X=\{G_Xu:u\in\bF_2^k\}$ and $\ker H_X=\{w\in\bF_2^{X(0)}:H_Xw=0\}$. Importantly, these codes are not necessarily equal. Rather, the following lemma shows that the former is contained in the latter, that is $\im G_X\subseteq\ker H_X$

\begin{lemma}
  \label{lem:GH0}
  For a $\bF_2$-Grassmannian complex $X$, we have $H_XG_X=0$.
\end{lemma}
\begin{proof}
  Each $x_1\in X(1)$ is by definition a 2-dimensional subspace of $\bF_2^k$, so the three nonzero elements of $x_1$ sum to $0$.
\end{proof}

Thus $\ker H_X$ and $\im G_X$ differ by the homology group $\ker H_X/\im G_X$ of the cochain complex
\begin{equation*}
  \bF_2^{X(1)} \xleftarrow{H_X} \bF_2^{X(0)} \xleftarrow{G_X} \bF_2^k.
\end{equation*}

We show in Section~\ref{sec:cayleyhom} that an appropriate quotient of $H_1(\Cay(\bF_2^k,X);\bF_2)$ is naturally isomorphic to $(\ker H_X/\im G_X)^*=\ker G_X^\top/\im H_X^\top$. This fact is perhaps unsurprising given the following result, which shows that an appropriate lift $\tilde{X}$ of $X$ has $\ker H_{\tilde{X}}/\im G_{\tilde{X}}=0$.

\begin{lemma}
  \label{lem:liftgrass}
  Let $X$ be a rank-1 Grassmannian complex, so that $\im G_X\subseteq\ker H_X$. Then there exists a ``universal cover'' $\tilde{X}$ of $X$, that is, a rank-1 Grassmannian complex $\tilde{X}$ such that $H_{\tilde{X}}=H_X$ and $\im G_{\tilde{X}}=\ker H_X$.
\end{lemma}

The requirement above that $H_{\tilde{X}}=H_X$ says that $\tilde{X}$ has the same incidence structure as $X$. Yet because $\im G_{\tilde{X}}=\ker H_{\tilde{X}}=\ker H_X$ may be larger than $\im G_X\subseteq\ker H_X$, the vertices in $\tilde{X}$ may span a larger vector space, that is, $\tilde{X}$ may be sparser in its ambient vector space than $X$. Equivalently, the existence of $\tilde{X}$ implies that the parity check code $\ker H_X$ may always be realized as a generator code $\im G_{\tilde{X}}$ for some rank-1 Grassmannian complex $\tilde{X}$. We will show below that $\im G_X$ has good distance if $X$ is a good expander, so it follows that $\ker H_X$ does too.

\begin{proof}[Proof of Lemma~\ref{lem:liftgrass}]
  Define $G_{\tilde{X}}$ to be any generating matrix for the code $\ker H_X$, and then let the elements of $\tilde{X}(0)$ be the rows of $G_{\tilde{X}}$, and let $\tilde{X}(1)$ consist of the same triangles as in $X(1)$, but with elements of $\tilde{X}(0)$ replacing the respective elements of $X(0)$.
\end{proof}

Lemma~\ref{lem:liftgrass} shows that we may think of a rank-1 Grassmannian complex $X$ over $\bF_2$ as a parity check matrix $H_X$ for an LDPC code with check-degree 3. Specifically, the code $\ker H_X=\ker H_{\tilde{X}}$ is realized as the image of the generator matrix $G_{\tilde{X}}$ associated to $\tilde{X}$. If we ever wish to recover $X$ as opposed to $\tilde{X}$, we may simply restrict to the message subspace $\spn X(0)\subseteq\spn\tilde{X}(0)$, which corresponds to restricting to the subcode $\im G_X\subseteq\im G_{\tilde{X}}$. We will show below how passing to such restrictions allows us to manipulate the homology of the resulting Cayley simplicial complex $\Cay(\spn X(0),X)$.

Thus we have essentially reduced the problem of finding good rank-1 Grassmannian expanders to that of finding LDPC codes of check-degree 3 such that the incidence structure of the constraints forms a good expander. Note that these objects are not generic, as good expansion can only be obtained if the parity checks have many redundancies. Specifically, the 1-skeleton graph $\bfG_X=\beta(X)$ associated to a rank-1 complex $X$ has $|X(0)|$ vertices and $3|X(1)|$ edges. Assuming for simplicity that this graph is regular, then the Alon-Boppana bound implies that as the expansion $\lambda(X)=\lambda(\bfG_X)\rightarrow 0$, the degree $6|X(1)|/|X(0)|\rightarrow\infty$. Thus for small expansion $\lambda(X)$, the parity check matrix $H_X\in\bF_2^{X(1)\times X(0)}$ must have many more rows (i.e.~constraints) than columns (i.e.~code components).

Let $k=\dim\spn X(0)$ denote the dimension of the code $\im G_X$. Then by definition $\Cay(\spn X(0),X)$ is a rank-2 simplicial complex with $2^k$ vertices and $2^k\cdot|X(1)|$ rank-2 faces, so that each vertex is contained in $3|X(1)|$ faces. Therefore to have $\Cay(\spn X(0),X)$ be low-degree (i.e.~have few faces relative to the number of vertices), for a given $|X(1)|$ we want $k$ to be large. Hence for a given incidence structure $H_X$, the complex $\tilde{X}$ by definition maximizes $\tilde{k}=\dim\spn\tilde{X}(0)$ and therefore yields the minimum degree Cayley simplicial complex. Note that if $X$ is connected then $3|X(1)|+1\geq|X(0)|$, so it must be that $k\leq|X(0)|\leq 3|X(1)|+1$. Therefore if $\Cay(\spn X(0),X)$ has connected links, that is, if $X$ is connected, then the number of rank-2 faces containing a given vertex in $\Cay(\spn X(0),X)$ is at least logarithmic in the number of vertices. This limitation is unsurprising, as Cayley expanders over abelian groups must have at least logarithmic degree.

The discussion above explains how the rate of the code $\im G_X$ determines the sparsity of $X$ in its ambient vector space, which in turn determines the degree of $\Cay(\spn X(0),X)$. We now show how the distance of $\im G_X$ is related to the expansion of $X$ and of $\Cay(\spn X(0),X)$. For this relationship to hold, we will use the equivalence of small-bias sets and well-balanced codes, which requires the technical assumption that $X$ is \textbf{regular}, meaning that all vertices have the same weight. Specifically, if rank-1 faces are given the uniform distribution, then all vertices must be contained in the same number of rank-1 faces.

Below we define $\epsilon$-balanced codes, which in particular have relative distance $\geq(1-\epsilon)/2$.

\begin{definition}
  A linear code $C\subseteq\bF_2^n$ is \textbf{$\epsilon$-balanced} if all nonzero codewords $c\in C\setminus\{0\}$ have Hamming weight $|c|\in[(1-\epsilon)n/2,(1+\epsilon)n/2]$.
\end{definition}

We also define $\epsilon$-biased sets.

\begin{definition}
  A set $S\subseteq\bF_2^k$ is \textbf{$\epsilon$-biased} if for every $u\in\bF_2^k$, it holds that $\Pr_{s\sim\Unif(S)}[s\cdot u=1]\in[(1-\epsilon)n/2,(1+\epsilon)n/2]$.
\end{definition}

The following well known lemma shows the equivalence between $\epsilon$-balanced codes, $\epsilon$-biased sets, and $\epsilon$-expanding Cayley graphs over $\bF_2^k$. The proof, which we omit, follows from the fact that the eigenvectors of the random walk matrix of a Cayley graph over an abelian group are precisely the Fourier characters of the group.

\begin{lemma}[Well known]
  \label{lem:epbias}
  For a matrix $G\in\bF_2^{n\times k}$, let $S\subseteq\bF_2^k$ be the muliset of all rows of $G$. Then $\lambda(\Cay(\bF_2^k,S))\leq\epsilon$ if and only if $S$ is an $\epsilon$-biased set, which in turn holds if and only if $\im G$ is an $\epsilon$-balanced code.
\end{lemma}

We now show that good expansion in a rank-1 $\bF_2$-Grassmannian complex $X$ implies good balance (and therefore good distance) in the associated code $\im G_X$.

\begin{proposition}
  \label{prop:exptodis}
  Let $X$ be a rank-1 $\bF_2$-Grassmannian complex in ambient vector space $\bF_2^k=\spn X(0)$ with expansion $\lambda$. Then $X(0)$ is an $\lambda/(1-\lambda)$-biased set, so if $X$ is regular then $\im G_X$ is a $\lambda/(1-\lambda)$-balanced code.
\end{proposition}
\begin{proof}
  By Lemma~\ref{lem:cayleyF2k}, $\Cay(\spn X(0),X)$ has rank-0 local expansion $\lambda$, so by the trickle-down theorem Theorem~\ref{thm:tricklesimp}, $\Cay(\spn X(0),X)$ has rank-$(-1)$ local expansion $\leq\lambda/(1-\lambda)$. But the 1-skeleton of $\Cay(\spn X(0),X)$ is by definition the Cayley graph $\Cay(\spn X(0),X(0))$. Thus $X(0)$ is $\lambda/(1-\lambda)$-biased and $\im G_X$ is $\lambda/(1-\lambda)$-balanced by Lemma~\ref{lem:epbias}.
\end{proof}

By Lemma~\ref{lem:liftgrass}, Proposition~\ref{prop:exptodis} also applies to $\ker H_X=\im G_{\tilde{X}}$, as $\tilde{X}$ by definition has the same graph structure, and therefore the same expansion, as $X$.

It is an interesting question whether good expanders $X$ of rank $>1$ have codes $\im G_X$ with stronger properties than large distance, such as local testability properties.

\section{Homology of Cayley simplicial complexes}
\label{sec:cayleyhom}
In this this section, we show a general characterization of the 1-homology groups with coefficients in $\bF_2$ of $\bF_2$-Cayley simplicial complexes. We then apply this characterization with our construction from Section~\ref{sec:cayconstruct} to construct $N$-vertex local spectral simplicial HDXs with 1-homology of dimension $\Omega(\log^2(N))$. To the best of our knowledge, prior $N$-vertex HDX constructions such as Ramanujan complexes and the Kaufman-Oppenheim coset complexes are at best known to have 1-homology of dimension order $\log N$; see Remark~\ref{rem:homdim}.
We create large homology groups by a quotienting procedure on Grassmannian complexes that enlarges the 1-homology group of the Cayley complex, while increasing the degree of the Cayley complex to polynomially large in the number of vertices. By our framework, sparser constructions of Cayley HDXs over $\bF_2^k$ would immediately yield HDXs with larger 1-homology.

In this section, all chains, cycles, boundaries, and homology groups (see Section~\ref{sec:topprelim} for a refresher) are implicitly assumed to be have coefficients in $\bF_2$.

\subsection{General result}
\label{subsec:generalhom}
In this section, we present a general characterization of the homology group of Cayley simplicial complexes generated by rank-$1$ $\bF_2$-Grassmannian complexes. Our characterization uses the following definition.

\begin{definition}
  Let $G$ be an abelian group and let $S$ be a rank-1 simplicial complex (that is, a graph) with vertex set $S(0)\subseteq G\setminus\{0\}$ such that $S$ satisfies the symmetry condition of Definition~\ref{def:cayleycomplex}, so that $\Cay(G,S)$ is a rank-1 Cayley simplicial complex. Define the \textbf{swap cycles} $S_1(\Cay(G,S))\subseteq Z_1(\Cay(G,S))$ to be the subgroup of 1-cycles
  \begin{equation*}
    S_1(\Cay(G,S)) = \spn\left\{\1_{\{\{v,v+x_0\},\{v+x_0,v+x_0+x_0'\},\{v+x_0',v+x_0+x_0'\},\{v,v+x_0'\}\}}:v\in G,x_0,x_0'\in S(0)\right\}.
  \end{equation*}
\end{definition}

That is, the swap cycles are generated by the length-4 1-cycles that start at some base point $v\in G$, and traverse the 4 edges with respective labels $x_0,x_0',-x_0,-x_0'$. More informally, swap cycles are generated by cycles that swap the order of two adjacent edge labels.

\begin{theorem}
  \label{thm:hommodswap}
  Let $X$ be a rank-1 $\bF_2$-Grassmannian complex, and let $Y=\Cay(\spn X(0),X)$. There is a natural isomorphism
  \begin{equation*}
    H_1(Y)/(S_1(Y)+B_1(Y)) \cong \ker G_X^\top/\im H_X^\top.
  \end{equation*}
\end{theorem}

In words, Theorem~\ref{thm:hommodswap} says that up to swap cycles, the 1-homology of $\Cay(\spn X(0),X)$ is the space of linear dependencies among elements of $X(0)$ that are not implied by the 3-term dependencies specified by faces in $X(1)$.

\begin{proof}[Proof of Theorem~\ref{thm:hommodswap}]
  Let $C_1,Z_1,B_1,H_1,S_1$ all refer to the respective groups for $Y$. Recall that $H_1=Z_1/B_1$, so $H_1/(S_1+B_1)=Z_1/(S_1+B_1)$. Define the label function $L:Y(1)\rightarrow X(0)$ by $L(\{y_0,y_0'\})=v+v'$, that is, $L$ maps an edge $y_1=\{y_0,y_0'\}\in Y(1)$ to the label $L(y_1)\in X(0)$ of Cayley generator for $y_1$.

  Define a map $\phi:Z_1\rightarrow\bF_2^{X(0)}$ such that for every 1-cycle $\alpha\in Z_1$,
  \begin{equation*}
    \phi(\alpha) = \sum_{y_1\in\supp\alpha}\1_{L(y_1)} \in \bF_2^{X(0)}.
  \end{equation*}
  First observe that for every $\alpha\in Z_1$, so that $\partial_1\alpha=0$, then
  \begin{align*}
    G_X^\top\phi(\alpha)
    &= \sum_{y_1\in\supp\alpha}L(y_1) = \sum_{y_1\in\supp\alpha}\sum_{y_0\in y_1}y_0 = \sum_{y_0\in\supp\partial_1\alpha}y_0 = 0.
  \end{align*}
  Thus $\im\phi=\phi(Z_1)\subseteq\ker G_X^\top$. We next show that $\phi(B_1),\phi(S_1)\subseteq\im H_X^\top$. By definition, $B_1$ is generated by elements $\1_{\partial_2y_2}$ for $y_2\in Y(2)$, where $y_2$ is of the form $y_2=\{y_0,y_0+x_0,y_0+x_0'\}$ for some $x_0,x_0'\in X(0)$ such that $x_1=\{x_0,x_0',x_0+x_0'\}\in X(1)$. Therefore $\phi(\1_{\partial_2y_2})=\1_{x_0}+\1_{x_0'}+\1_{x_0+x_0'}=H_X^\top\1_{x_1}$. Thus $\phi(B_1)\subseteq\im H_X^\top$. Similarly, $S_1$ is generated by elements $\alpha=\1_{y_1^{(1)}}+\1_{y_1^{(2)}}+\1_{y_1^{(3)}}+\1_{y_1^{(4)}}$ where $L(y_1^{(1)})=L(y_1^{(3)})$ and $L(y_1^{(2)})=L(y_1^{(4)})$, so every such $\alpha$ has $\phi(\alpha)=0$. Therefore we have shown that $\phi(Z_1)\subseteq\ker G_X^\top$ and $\phi(B_1+S_1)\subseteq\im H_X^\top$, so $\phi$ induces a map
  \begin{equation*}
    \phi:H_1/(S_1+B_1)\rightarrow\ker G_X^\top/\im H_X^\top.
  \end{equation*}

  We now construct an inverse to the map $\phi$ above. Define $\psi:\ker G_X^\top\rightarrow H_1/(S_1+B_1)$ as follows. For $u\in\ker G_X^\top\subseteq\bF_2^{X(0)}$, let $x_0^{(1)},\dots,x_0^{(k)}\in X(0)$ be an arbitrary sequence of elements such that $\sum_{i\in[k]}\1_{x_0^{(i)}}=u\in\bF_2^{X(0)}$, or equivalently such that the elements occuring an even number of times in $x_0^{(1)},\dots,x_0^{(k)}$ are precisely elements of $\supp u$. For $0\leq i\leq k$ let $v_i=\sum_{j\in[i]}x_0^{(j)}$ denote the $i$th partial sum, so that $v_k=v_0=0$ because $u\in\ker G_X^\top$. Then let
  \begin{equation*}
    \psi(u) = \sum_{i\in[k]}\1_{\{v_{i-1},v_i\}} + B_1+S_1 \in Z_1/(B_1+S_1).
  \end{equation*}
  It follows from $v_0=v_k$ that $\partial_1\sum_{i\in[k]}\1_{\{v_{i-1},v_i\}}=\sum_{i\in[k]}(\1_{v_{i-1}}+\1_{v_i})=0$, so $\psi(u)$ indeed belongs in $Z_1/(B_1+S_1)$. To show that $\psi:\ker G_X^\top\rightarrow H_1/(S_1+B_1)$ is a well defined linear map, it suffices to show that the output $\psi(u)$ does not depend on the choice of the sequence $x_0^{(1)},\dots,x_0^{(k)}$. For this purpose, fix $u\in\ker G_X^\top$. For $\vec{x}_0=(x_0^{(1)},\dots,x_0^{(k)})$ with $\sum_{i\in[k]}\1_{x_0^{(i)}}=u$, let $\psi^{\vec{x}_0}(u)$ denote the output of $\psi$ for the sequence $\vec{x}_0$. Also fix some enumeration $\vec{x}'_0=({x'}_0^{(1)},\dots,{x'}_0^{(\ell)})$ of $\supp u$. Our goal is to show that every $\vec{x}_0$ with $\sum_{i\in[k]}\1_{x_0^{(i)}}=u$ satisfies $\psi^{\vec{x}_0}(u)=\psi^{\vec{x}'_0}(u)$. Observe that the output $\psi^{\vec{x}_0}(u)$ is unchanged if two consecutive elements of $\vec{x}_0$ are swapped, as performing such a swap simply adds an element of $S_1$ to $\sum_{i\in[k]}\1_{\{v_{i-1},v_i\}}$. Then because every permutation $\pi:[k]\rightarrow[k]$ can be expressed as a composition of swaps of adjacent elements, it follows that every permutation $\pi(\vec{x}_0)=(x_0^{(\pi(1))},\dots,x_0^{(\pi(k))})$ of the sequence $\vec{x}_0$ has $\psi^{\vec{x}_0}(u)=\psi^{\pi(\vec{x}_0)}(u)$. Therefore define $\pi$ so that the first $\ell$ components of $\pi(\vec{x}_0)$ equals $\vec{x}_0'$, and the remaining $k-\ell$ components all come in consecutive pairs of equal elements, that is $\vec{x}_0^{(\ell+2i-1)}=\vec{x}_0^{(\ell+2i)}$ for all $1\leq i\leq(k-\ell)/2$. Such a permutation $\pi$ must exist by the definition of $\vec{x}_0$ and $\vec{x}_0'$. Letting $v_i^\pi$ and $v_i'$ denote the partial sums for the sequences $\pi(\vec{x}_0)$ and $\vec{x}_0'$ respectively, it follows by construction $v_i^\pi=v_i'$ for every $0\leq i\leq\ell$, and that $v_{\ell+2(i-1)}^\pi=v_{\ell+2i}^\pi$ for every $1\leq i\leq (k-\ell)/2$, so $\sum_{i=\ell+1}^k\1_{\{v_{i-1}^\pi,v_i^\pi\}}=0$, and therefore $\sum_{i\in[k]}\1_{\{v_{i-1}^\pi,v_i^\pi\}}=\sum_{i\in[\ell]}\1_{\{v_{i-1}^\pi,v_i^\pi\}}=\sum_{i\in[\ell]}\1_{\{v_{i-1}',v_i'\}}$. Thus $\psi^{\vec{x}_0}(u)=\psi^{\pi(\vec{x}_0)}(u)=\psi^{\vec{x}_0'}(u)$, as desired, so $\psi:\ker G_X^\top\rightarrow H_1/(S_1+B_1)$ is indeed a well defined linear map. It also follows that $\psi(\im H_X^\top)=0$, as $\im H_X^\top$ is generated by elements $\1_{x_1}\in\ker G_X^\top\subseteq\bF_2^{X(0)}$ for $x_1=\{x_0,x_0',x_0+x_0'\}\in X(1)$. Then letting $y_2=\{0,x_0,x_0'\}\in Y(2)$, we have $\psi(\1_{x_1})=\1_{\{0,x_0\}}+\1_{\{x_0,x_0'\}}+\1_{\{x_0',0\}}+B_1+S_1=\partial_2\1_{y_2}+B_1+S_1=B_1+S_1=0\in H_1/(S_1+B_1)$. Thus indeed $\psi(\im H_X^\top)=0$, so $\psi$ induces a map
  \begin{equation*}
    \psi:\ker G_X^\top/\im H_X^\top\rightarrow H_1/(S_1+B_1).
  \end{equation*}
  
  It remains to be shown that $\phi$ and $\psi$ are inverses, that is, that $\phi\psi$ and $\psi\phi$ are the identity map. The fact that $\phi\psi=I$ follows directly from tracing through the definitions. Specifically, for $u\in\ker G_X^\top\subseteq\bF_2^{X(0)}$, let $x_0^{(1)},\dots,x_0^{(k)}\in X(0)$ be some enumeration of $\supp u$, and let $v_i=\sum_{j\in[i]}x_0^{(j)}$ denote the $i$th partial sum. Then by definition $\psi(u+\im H_X^\top)=\sum_{i\in[k]}\1_{\{v_{i-1},v_i\}}+B_1+S_1$ and $\phi(\sum_{i\in[k]}\1_{\{v_{i-1},v_i\}}+B_1+S_1)=\sum_{i\in[k]}\1_{x_0^{(i)}}+\im H_X^\top=u+\im H_X^\top$.

  To show that $\psi\phi=I$, it suffices to show that $\psi\phi(\alpha+B_1+S_1)=\alpha+B_1+S_1$ for a set of 1-cycles $\alpha$ that generate $Z_1$. It specifically suffices to show this equality for all $\alpha$ that form a cycle passing through $0$, in the traditional sense of a graph cycle being a circular sequence of adjacent edges. Specifically, any element of $Z_1$ can be decomposed into such ``traditional'' graph cycles by greedily performing a walk on edges in $Z_1$ to find such a cycle, and then removing it and repeating until there are no edges left. By definition the 1-skeleton $\Cay(\spn X(0),X(0))$ of $Y$ is connected, so every such cycle can be modified to pass through $0$ by adding two copies of a path from $0$ to a some point in the cycle. Thus we may assume that $\alpha$ is of the form $\alpha=\sum_{i\in[k]}\1_{\{v_{i-1},v_i\}}$ for some sequence of points $v_0,v_1,\dots,v_k\in\spn X(0)$ such that $v_0=v_k=0$, and $L(\{v_{i-1},v_i\})=v_{i-1}+v_i\in X(0)$. Letting $x_0^{(i)}=v_{i-1}+v_i$, it follows that $\phi(\alpha+B_1+S_1)=\sum_{i\in[k]}\1_{x_0^{(i)}}+\im H_X^\top\in\ker G_X^\top/\im H_X^\top$, and then by definition $\psi(\sum_{i\in[k]}\1_{x_0^{(i)}}+\im H_X^\top)=\sum_{i\in[k]}\1_{\{v_{i-1},v_i\}}+B_1+S_1=\alpha+B_1+S_1$. Thus $\psi\phi$ acts as the identity on a set of $\alpha+B_1+S_1$ that generate $H_1/(S_1+B_1)$, so $\psi\phi=I$, which completes the proof that $\phi^{-1}=\psi$.
\end{proof}

Because $\ker G_X^\top/\im H_X^\top = (\ker H_X/\im G_X)^*$, Theorem~\ref{thm:hommodswap} shows that $H_1/(S_1+B_1)$ is the dual of $\ker H_X/\im G_X$. Therefore $\dim H_1\geq\dim(H_1/(S_1+B_1))=\dim(\ker G_X^\top/\im H_X^\top)$. The following lemma and corollary apply this observation to show that if $|\spn X(0)|$ is sufficiently large compared to $|X(0)|$, we may enlarge $H_1$ by restricting to an appropriate subcode of $\im G_X$.

\begin{lemma}
  \label{lem:quotient}
  Let $X$ be a $\bF_2$-Grassmannian complex. If $|\spn X(0)|\geq|X(0)|^2/2+|X(0)|/2+2$, then there exists a 1-dimensional subspace $V\subseteq\spn X(0)$ such that the ``quotient'' $\bF_2$-Grassmannian complex $X'/V$ in ambient vector space $\spn X(0)/V$ given by $X'/V=\{\spn\{x,V\}/V:x\in X\}$ has the same poset incidence structure as $X$, that is, $X$ and $X'/V$ are isomorphic as posets. In particular, $H_{X'}=H_X$. Furthermore, the following 3 equivalent conditions hold:
  \begin{itemize}
  \item $\im G_{X'}$ is a codimension-1 subspace of $\im G_X$
  \item $\ker G_{X}^\top$ is a codimension-1 subspace of $\ker G_{X'}^\top$
  \item $\spn X'(0)$ is the quotient of $\spn X(0)$ by a 1-dimensional subspace.
  \end{itemize}
\end{lemma}
\begin{proof}
  Because $|\spn X(0)|\geq|X(0)|^2/2+|X(0)|/2+2$, there exists some $v\in\spn X(0)$ such that $v\notin\{0\}\cup X(0)\cup X(0)+X(0)$. That is, $v$ cannot be expressed as the sum of $\leq 2$ elements of $X(0)$. Let $V=\spn\{0,v\}$. Then letting $X'$ be defined as in the lemma statement, because $v\notin\{0\}\cup X(0)\cup X(0)+X(0)$, all $\spn\{x_0,V\}\in X'(0)$ for $x_0\in X(0)$ are nonzero and distinct subspaces of $\spn(X(0))/V$, so $X'$ is indeed well defined with the same poset incidence structure as $X$. By construction every linear constraint satisfied by elements of $X(0)$ is also satisfied the the respective elements of $X'(0)=X(0)/V$, so $\ker G_X^\top\subseteq\ker G_{X'}^\top$. But as the matrix $G_{X'}^\top$ has one fewer row, that is rank 1 less, than $G_X^\top$, it follows that $\im G_{X'}$ is a codimension-1 subspace of $\im G_X$, and $\ker G_X^\top$ is a codimension-1 subspace of $\ker G_{X'}^\top$. Here the linear constraints that are satisfied by $X'(0)$ but not by $X(0)$, or equivalently that lie in $\ker G_{X'}^\top\setminus\ker G_X^\top$, are given by those sets of elements in $X(0)$ that sum to $v\in\spn X(0)$, so that the respective elements in $X'(0)$ sum to $0\in\spn X(0)/V$.
\end{proof}

\begin{corollary}
  \label{cor:increasehom}
  Let $X$ be a $\bF_2$-Grassmannian complex with $|\spn X(0)|\geq|X(0)|^2/2+|X(0)|/2+2$, and let $X'$ be the quotient complex given by Lemma~\ref{lem:quotient}. Let $Y=\Cay(\spn X(0),X)$ and $Y'=\Cay(\spn X'(0),X')$. Then $H_1(Y)/(S_1(Y)+B_1(Y))$ is isomorphic to a codimension-1 subspace of $H_1(Y')/(S_1(Y')+B_1(Y'))$.
\end{corollary}
\begin{proof}
  Theorem~\ref{thm:hommodswap} shows that $H_1(Y)/(S_1(Y)+B_1(Y))\cong\ker G_X^\top/\im H_X^\top$, and that an analogous isomorphism holds for $Y'$. Lemma~\ref{lem:quotient} implies that $H_X=H_X'$ and $\ker G_X^\top$ is a codimension-1 subspace of $\ker G_{X'}^\top$, so it follows that $H_1(Y)/(S_1(Y)+B_1(Y))$ is isomorphic to a codimension-1 subspace of $H_1(Y')/(S_1(Y')+B_1(Y'))$.
\end{proof}

By repeatedly applying Corollary~\ref{cor:increasehom} from an initial rank-1 Grassmannian complex $X$ that is sparse in its ambient vector space, meaning that $|\spn X(0)|\gg|X(0)|^2$, we can construct a Cayley simplicial complex with links isomorphic to $X$ and with large 1-homology. In particular, if $X$ is a good expander, then the result Cayley simplicial complex will be a high-dimensional expander with large 1-homology. The following section applies this approach to the construction from Section~\ref{sec:grassHDX}.

\subsection{Application to low-rank-matrix construction}
Below, we combine the results of Section~\ref{sec:grassHDX} and Section~\ref{subsec:generalhom} to construct $N$-vertex high-dimensional expanders with $\dim H_1=\Omega(\log N)^2$.

\begin{theorem}
  \label{thm:bighom}
  For integers $r\geq 1$, $b\geq 5$, $n\geq 2^{r+1}$, let $q=2^b$, and let $X=X^{r,b,n}$ be the rank-$r$ $\bF_2$-Grassmannian in Definition~\ref{def:grassconstruct}. Then repeated quotienting (in the sense of Lemma~\ref{lem:quotient}) gives a rank-$r$ $\bF_2$-Grassmannian complex $X'$ with rank-$(r+1)$ Cayley simplicial complex $Y'=\Cay(\spn X'(0),\beta(X'))$ that has $|Y'(0)|=2^{2^{r+3}bn}$ vertices, 1-homology of dimension $\dim H_1(Y')\geq bn^2-2^{r+3}bn$, and local expansion $\lambda(Y)\leq 11/(q-11)$.
\end{theorem}
\begin{proof}
  Recall that $X$ lies in ambient vector space $\spn X(0)=\bF_2^{bn^2}$ and has $|X(0)|\leq 2^{2^{r+2}bn}$ vertices. Therefore we may apply $bn^2-2^{r+3}bn$ iterations of Lemma~\ref{lem:quotient} to obtain $X'$, so that $Y'$ by definition has $|Y'(0)|=2^{2^{r+3}bn}$ vertices, 1-homology of dimension $\dim H_1(Y')\geq 2^{bn^2}-2^{2^{r+3}bn}$ by Corollary~\ref{cor:increasehom}, and local expansion $\lambda(Y)\leq 11/(q-11)$ by Corollary~\ref{cor:applytrickle} and Theorem~\ref{thm:tricklesimp} because by Lemma~\ref{lem:quotient} the quotienting preserves poset structure and thus preserves local expansion.
\end{proof}

Letting $r,b$ be fixed constants as $n\rightarrow\infty$, Theorem~\ref{thm:bighom} gives simplicial complexes $Y'$ on $N=2^{\Theta(n)}$ vertices with local expansion $11/(2^b-11)$, and with 1-homology of dimension with 1-homology of dimension $\Omega(n^2)=\Omega(\log^2N)$.

Because the links of $Y'$ are isomorphic to $X$, each vertex of $Y'$ is contained in $|X|=2^{\Theta(n)}=\poly(N)$ faces, so $Y'$ in Theorem~\ref{thm:bighom} has polynomially large degree. Observe that when we choose the number of quotienting iterations with Lemma~\ref{lem:quotient}, we face a tradeoff between the degree and the dimension of the 1-homology. For instance, we could alternatively apply only $bn^2/2$ iterations of Lemma~\ref{lem:quotient} to obtain a complex on $N'=2^{\Theta(n^2)}$ vertices with subpolynomial degree $2^{\Theta(n)}=2^{\Theta(\sqrt{\log N'})}$ and 1-homology of logarithmic dimension $\Omega(n^2)=\Omega(\log N')$. Alternatively, we could apply $bn^2-bn\log n$ iterations to obtain a complex on $N''=2^{\Theta(n\log n)}$ vertices still with subpolynomial degree $2^{\Theta(n)}=2^{\Theta(\log N''/\log\log N'')}$ and 1-homology of dimension $\Omega(n^2)=\Omega(\log N''/\log\log N'')^2=\tilde{\Omega}(\log^2N'')$. That is, if we are willing to incur a very slight loss in the 1-homology dimension in Theorem~\ref{thm:bighom} ($\dim H_1=\tilde{\Omega}(\log^2N)$ instead of $\Omega(\log^2N)$ where $N$ is the number of vertices), we can ensure the complex has subpolynomial degree.

More generally, Corollary~\ref{cor:increasehom} implies that if we had Grassmannian HDXs that were sparser in their ambient vector space than our construction in Section~\ref{sec:grassHDX}, we would immediately obtain simplicial HDXs with larger 1-homology. In particular, if the Grassmannian HDX $X$ was optimally sparse, so that $|X(0)|=\Theta(\dim\spn X(0))$, then we could quotient $\dim\spn X(0)-2\log|X(0)|=\Theta(|X(0)|)$ times to obtain simplicial complexes on $N=|X(0)|^2$ vertices with 1-homology of dimension $\Theta(|X(0)|)=\Theta(\sqrt{N})$. That is, with this approach we could hope to obtain simplicial HDXs with polynomially large dimension of the 1-homology.

\section{Acknowledgments}
The author thanks Ran Tessler for numerous helpful discussions and suggestions over the course of this work. The author also thanks Omar Alrabiah, Venkat Guruswami, Dan Karliner, Siqi Liu, Sidhanth Mohanty, and Salil Vadhan for helpful discussions, and thanks Omar Alrabiah, Sidhanth Mohanty, and Salil Vadhan for helping improve the presentation of this paper.

The author is supported by a National Science Foundation Graduate Research Fellowship under Grant No.~DGE 2146752.

\bibliography{library}
\bibliographystyle{alpha}

\end{document}